\documentclass[12pt]{article}
\usepackage{amssymb,amsmath,amscd, amsthm,stmaryrd,verbatim, mathrsfs}
\usepackage{enumerate}
\numberwithin{equation}{section}
\newtheorem{theorem}{Theorem}[section]
\newtheorem{lemma}[theorem]{Lemma}

\theoremstyle{definition}
\newtheorem{definition}[theorem]{Definition}
\theoremstyle{remark}
\newtheorem{remark}[theorem]{Remark}

\begin{document}
\baselineskip=15pt

\newcommand{\la}{\langle}
\newcommand{\ra}{\rangle}
\newcommand{\psp}{\vspace{0.4cm}}
\newcommand{\pse}{\vspace{0.2cm}}
\newcommand{\ptl}{\partial}
\newcommand{\dlt}{\delta}
\newcommand{\sgm}{\sigma}
\newcommand{\al}{\alpha}
\newcommand{\be}{\beta}
\newcommand{\G}{\Gamma}
\newcommand{\gm}{\gamma}
\newcommand{\vs}{\varsigma}
\newcommand{\Lmd}{\Lambda}
\newcommand{\lmd}{\lambda}
\newcommand{\td}{\tilde}
\newcommand{\vf}{\varphi}
\newcommand{\yt}{Y^{\nu}}
\newcommand{\wt}{\mbox{wt}\:}
\newcommand{\rd}{\mbox{Res}}
\newcommand{\ad}{\mbox{ad}}
\newcommand{\stl}{\stackrel}
\newcommand{\ol}{\overline}
\newcommand{\ul}{\underline}
\newcommand{\es}{\epsilon}
\newcommand{\dmd}{\diamond}
\newcommand{\clt}{\clubsuit}
\newcommand{\vt}{\vartheta}
\newcommand{\ves}{\varepsilon}
\newcommand{\dg}{\dagger}
\newcommand{\tr}{\mbox{Tr}}
\newcommand{\ga}{{\cal G}({\cal A})}
\newcommand{\hga}{\hat{\cal G}({\cal A})}
\newcommand{\Edo}{\mbox{End}\:}
\newcommand{\for}{\mbox{for}}
\newcommand{\kn}{\mbox{ker}}
\newcommand{\Dlt}{\Delta}
\newcommand{\rad}{\mbox{Rad}}
\newcommand{\rta}{\rightarrow}
\newcommand{\mbb}{\mathbb}
\newcommand{\lra}{\Longrightarrow}
\newcommand{\X}{{\cal X}}
\newcommand{\Y}{{\cal Y}}
\newcommand{\Z}{{\cal Z}}
\newcommand{\U}{{\cal U}}
\newcommand{\V}{{\cal V}}
\newcommand{\W}{{\cal W}}
\newcommand{\sta}{\theta}
\setlength{\unitlength}{3pt}
\newcommand{\msr}{\mathscr}
\newcommand{\wht}{\widehat}

\begin{center}{\large \bf Full Conformal Oscillator Representations of Orthogonal \\ \pse Lie Algebras and Combinatorial Identities} 
\end{center}
\vspace{0.2cm}

\begin{center}{\large Zhenyu Zhou$^{a,\:}$\footnote{Corresponding author.} and Xiaoping Xu$^{b,\;c,\;}$\footnote{Research supported
 by National Key R\&D Program of China 2020YFA0712600.
} }\end{center}

\begin{center}{
$^{a\;}$Chern Institute of Mathematics \& LPMC, Nankai University,\\ Tianjin 300071, China\\
$^{b\;}$HLM, Institute of Mathematics, Academy of Mathematics \& System
Sciences\\ Chinese Academy of Sciences, Beijing 100190, P.R. China
\\ $^{c\;}$School of Mathematics, University of Chinese Academy of Sciences,\\ Beijing 100049, P.R. China}
\end{center}

\begin{abstract}
\quad

 Zhao and the second author (2013) constructed a functor from
$\mathfrak{o}(k)$-{\bf Mod} to $\mathfrak{o}(k+2)$-{\bf Mod}. In this paper, we use the functor successively to obtain
an  universal first-order differential operator realization  for any highest-weight
 representation of $\mathfrak{o}(2n+3)$ in $(n+1)^2$ variables and that of $\mathfrak{o}(2n+2)$ in $n(n+1)$ variables. When the highest weight is  dominant integral,
 we determine the corresponding finite-dimensional irreducible module explicitly.  One can use the result to study tensor decompositions of finite-dimensional irreducible modules by solving certain first-order linear partial differential equations, and thereby obtain the corresponding physically interested Clebsch-Gordan coefficients and  exact solutions of Knizhnik-Zamolodchikov equation in WZW model of conformal field theory. We also find an equation of counting the dimension of an irreducible $\mathfrak{o}(k+2)$-module in terms of certain alternating sum of the dimensions of  irreducible $\mathfrak{o}(k)$-modules.  In the case of the Steinberg modules, we obtain new  combinatorial identities of classical type.

 \vspace{0.3cm}

\noindent{\it Keywords}:\hspace{0.3cm} orthogonal Lie
algebra; conformal oscillator
 representation; irreducible module;  singular vectors; combinatorial identities.
	
\end{abstract}

\section{Introduction}
	
Finite-dimensional irreducible representations of finite-dimensional simple Lie algebra over $\mbb C$ were abstractly determined
by Cartan and Weyl in early last century. However, explicit representation formulas of the root vectors in these simple algebras are in general very difficult to be given. In 1950, Gelfand and Tsetlin \cite{GT1, GT2}
used a sequence of corank-one subalgebras to obtain a basis whose elements were labeled by upside-down triangular data for finite-dimensional irreducible representations of general linear Lie algebras and orthogonal Lie algebras, respectively. Moreover, the corresponding matrix elements of simple root vectors (or Chevalley basis) were explicitly given. Zhelobenko \cite{DPZ1,DPZ2} derived the well-known branching rules for all classical Lie algebras by the methods that the representations are realized in a space of polynomials satisfying a certain indicator system of differential equations. Moreover, he constructed the lowering operators to depict the matrix element formulas for the general linear Lie algebra. His work were extended to the orthogonal Lie algebras by Nagel and Moshinsky \cite{N1}, Pang and Hechi \cite{P1} and Wong \cite{W1}.
Based on the theory of Mickelsson algebras and the representation theory of the Yangians, Molev \cite{M1,M2,M3,M4} constructed a weight basis for finite-dimensional irreducible representations of classical Lie algebras and obtained explicit formulas for the matrix elements of generators. There are also many interesting works  in this direction (e.g., cf. \cite{C,FGR1,FGR2,G1,G2,G3,GK,Lp,M1}).

 Sum-product type identities are important objects both in combinatorics and number theory. The Jacobi triple and
quintuple product identities are such well-known examples. Macdonald \cite{Mi} used affine analogues of the root systems of finite-dimensional simple Lie algebras to derive new type sum-product identities with the above two identities as special cases. Kac \cite{K} found the character formula for the integrable modules of affine
Kac-Moody algebras and showed that  the Macdonald identities are exactly the denominator identities.
The denominator identities of finite-dimensional simple Lie algebras are Vandermonde determinant type identities, which do not produce sum-product identities of numbers.

 Using the conformal oscillator representation of the orthogonal Lie algebra $\mathfrak{o}(k+2)$ and Shen's mixed
product for Witt algebras in \cite{S}, Zhao and the second author \cite{XZ} constructed a functor from
$\mathfrak{o}(k)$-{\bf Mod} to $\mathfrak{o}(k+2)$-{\bf Mod}. In this paper, we use the functor successively to obtain
an  universal first-order differential operator realization  for any highest-weight
 representation of $\mathfrak{o}(2n+3)$ in $(n+1)^2$ variables and that of $\mathfrak{o}(2n+2)$) in $n(n+1)$ variables. In comparison with Gelfand and Tsetlin's local matrix elements of simple root vectors in \cite{GT2}, we obtain the precise global pictures  of all root vectors.  When the highest weight is  dominant integral, we determine the corresponding finite-dimensional irreducible module explicitly.  Our construction also gives rise to an equation of counting the dimension of an irreducible $\mathfrak{o}(k+2)$-module in terms of certain alternating sum of the dimensions of  irreducible $\mathfrak{o}(k)$-modules. It could be viewed
  as a finite analogue of the Macdonald identities in a certain sense.
  In the case of the Steinberg modules, we obtain new nice combinatorial identities of classical type, which might be viewed as finite analogues
 of the Jacobi triple and quintuple product identities.

 In physics, the Clebsch-Gordan  coefficients are numbers that arise in angular momentum coupling in quantum mechanics. They appeared as the expansion
coefficients of total angular momentum eigenstates in an uncoupled tensor product basis (e.g., cf. \cite{dS,E}). We refer \cite{A, B,L,R} for more applications and later developments. Mathematically, the numbers are those of explicitly determining the irreducible components in the tensor of two finite-dimensional irreducible modules in terms of orthonormal bases.
 Knowing the matrix elements of simple roots in \cite{GT1, GT2} is not enough to solve the general decomposition problem of
the tensors of irreducible representations because they are not commuting operators.
One can use our result in this paper to study the problem for orthogonal Lie algebras by solving certain first-order linear partial differential equations, and thereby obtain the corresponding physically interested Clebsch-Gordan coefficients and  exact solutions of Knizhnik-Zamolodchikov equation in WZW model of two-dimensional conformal field theory (cf. \cite{KZ, TK, X1}). Indeed, our result may also be helpful in constructing higher-dimensional conformal field theory.

  Let $n\geqslant 1$ be an integer, we define the inner product $(\cdot,\cdot)$ on $2n+1$ dimensional Euclid space by
	\begin{eqnarray}
		(\vec{x},\vec{y})=\sum_{i=-n}^nx_iy_{-i},\label{1.1}
	\end{eqnarray}
	where $\vec{x}=(x_0,x_1,\cdots,x_n,x_{-1},\cdots,x_{-n}),~\vec{y}=(y_0,y_1,\cdots,y_n,y_{-1},\cdots,y_{-n})$. The conformal group with respect to $(\cdot,\cdot)$ is given by the translations, rotations, dilations and special conformal transformations
	\begin{eqnarray}
		\vec{x}\longmapsto\frac{\vec{x}-(\vec{x},\vec{x})\vec{b}}{(\vec{b},\vec{b})(\vec{x},\vec{x})-2(\vec{b},\vec{x})+1}. \label{1.2}
	\end{eqnarray}
	
	 Assuming that the base field is the field $\mathbb{C}$ of complex numbers in the rest of this paper, we will deal with odd and even orthogonal Lie algebras simultaneously in the following paragraphs. Now we take the integer $n\geqslant 0$. Let $E_{i,j}$ be the $(2n+3)\times (2n+3)$ or $(2n+2)\times (2n+2)$ matrix with $1$ as its $(i,j)$-entry and $0$ as the others. Denote
	\begin{eqnarray}
		&A_{i,j}=-A_{-j,-i}=E_{i,j}-E_{n+1+j,n+1+i},\label{1.3}\\
		&A_{i,-j}=E_{i,n+1+j}-E_{j,n+1+i},~A_{-i,j}=E_{n+1+i,j}-E_{n+1+j,i},\\
		&A_{0,i}=-A_{-i,0}=E_{0,i}-E_{n+1+i,0},~A_{0,-i}=-A_{i,0}=E_{0,n+1+i}-E_{i,0},~A_{0,0}=0.\label{1.4}
	\end{eqnarray}
	for $1\leqslant i,j\leqslant n+1$. The the odd and even orthogonal Lie algebras
	\begin{eqnarray}
			&\mathfrak{o}(2n+3)\cong\sum\limits_{i,j=1}^{n+1}\mathbb{C}A_{i,j}+\sum\limits_{1\leqslant i< j\leqslant n+1}\mathbb{C}(A_{i,-j}+\mathbb{C}A_{-i,j})+\sum\limits_{i=1}^{n+1}(\mathbb{C}A_{0,i}+\mathbb{C}A_{0,-i}), \label{1.5}\\
			&\mathfrak{o}(2n+2)\cong\sum\limits_{i,j=1}^{n+1}\mathbb{C}A_{i,j}+\sum\limits_{1\leqslant i< j\leqslant n+1}\mathbb{C}(A_{i,-j}+\mathbb{C}A_{-i,j}).\label{1.6}
	\end{eqnarray}
	Let $\mathscr{A}=\mathbb{C}[x_n,x_{n-1},\cdots,x_{-n}]$ be the polynomial algebra in $2n+1$ variables in the odd case and let $\mathscr{A}=\mathbb{C}[x_{\pm 1},x_{\pm_2},\cdots,x_{\pm n}]$ in the even case, respectively. Set
	\begin{eqnarray}
		&D_{2n+1}=\sum\limits_{s=-n}^nx_s\partial_{x_s},\quad\eta_{2n+1}=\frac{1}{2}x_0^2+\sum\limits_{t=1}^nx_tx_{-t}, \label{1.7}\\
		&D_{2n}=\sum\limits_{|s|=1}^nx_s\partial_{x_s},\quad\eta_{2n}=\sum\limits_{t=1}^nx_tx_{-t}.\label{1.8}
	\end{eqnarray}
	Differentiating the transformation in (\ref{1.2}), we get an inhomogeneous first-order differential operators representations of $\mathfrak{o}(2n+3)$. Replace $2n+1$ by $2n$ and take the product $(\vec{x},\vec{y})=\sum\limits_{|i|=1}^nx_iy_{-i}$ in (\ref{1.2}). Again by differentiating the transformation, we get the counterpart of $\mathfrak{o}(2n+2)$.
Let $\mathscr{G}_n=\mathfrak{o}(2n+1),\;\mathscr{G}_{n+1}=\mathfrak{o}(2n+3)$ in the odd case, and let $\mathscr{G}_n=\mathfrak{o}(2n),\;\mathscr{G}_{n+1}=\mathfrak{o}(2n+2)$ in the even case.
Suppose that $M$ is a $\mathscr{G}_n$-module. Denote
	\begin{eqnarray}
		\widehat{M}=\mathscr{A}\otimes_{\mathbb{C}} M. \label{1.9}
	\end{eqnarray}
	We fix $c\in \mbb C$. Using the above representation and Shen's mixed product for Witt algebra in \cite{S}, Zhao and the second author \cite{XZ} obtained the following representation of $\mathscr{G}_{n+1}$ on $\widehat{M}$: if $\mathscr{G}_{n+1}=\mathfrak{o}(2n+2)$,
	\begin{eqnarray}
		&A_{i,j}|_{\widehat{M}}=(x_i\partial_{x_j}-x_{-j}\partial_{x_{-i}})\otimes \text{Id}_M+\text{Id}_\mathscr{A}\otimes A_{i,j}, \label{1.10e}\\
		&A_{n+1,i}|_{\widehat{M}}=\partial_{x_i}\otimes \text{Id}_{M},~A_{n+1,n+1}|_{\widehat{M}}=-(D_{2n}+c)\otimes \text{Id}_M, \label{1.11e}\\
		&A_{i,n+1}|_{\widehat{M}}=(-x_i(D_{2n}+c)+\eta_{2n}\partial_{x_{-i}})\otimes \text{Id}_M- \sum\limits_{r=-n}^nx_r\otimes A_{i,r}. \label{1.12e}
	\end{eqnarray}
	for $1\leqslant |i|,|j|\leqslant n$; and if $\mathscr{G}_{n+1}=\mathfrak{o}(2n+3)$,
	\begin{eqnarray}
		&A_{i,j}|_{\widehat{M}}=(x_i\partial_{x_j}-x_{-j}\partial_{x_{-i}})\otimes \text{Id}_M+\text{Id}_\mathscr{A}\otimes A_{i,j}, \label{1.10}\\
		&A_{n+1,i}|_{\widehat{M}}=\partial_{x_i}\otimes \text{Id}_{M},~A_{n+1,n+1}|_{\widehat{M}}=-(D_{2n+1}+c)\otimes \text{Id}_M, \label{1.11}\\
		&A_{i,n+1}|_{\widehat{M}}=(-x_i(D_{2n+1}+c)+\eta_{2n+1}\partial_{x_{-i}})\otimes \text{Id}_M-\sum\limits_{r=-n}^nx_r\otimes A_{i,r}. \label{1.12}
	\end{eqnarray}
	for $0\leqslant |i|,|j|\leqslant n$. Moreover
	\begin{eqnarray}
		\overline{M}=U(\mathscr{G}_{n+1})(1\otimes M) \label{1.13}
	\end{eqnarray}
	forms a $\mathscr{G}_{n+1}$-submodule. Denote by $\mathscr{A}_r$ the subspace of homogeneous polynomials in $\mathscr{A}$ with degree $r$. Set
	\begin{eqnarray}
		\widehat{M}_r=\mathscr{A}_r\otimes_{\mathbb{C}}M,~~\overline{M}_r=\overline{M}\bigcap \widehat{M}_r. \label{1.14}
	\end{eqnarray}
	According to (\ref{1.10e})-(\ref{1.12}), $\widehat{M}_r$ is the tensor module of the $\msr G_n$-module $\mathscr{A}_r$ and $M$, and $\overline{M}_r$ is a $\mathscr{G}_{n}$-submodule of $\widehat{M}_r$. If $M$ is an irreducible $\mathscr{G}_{n}$-module, then $\overline{M}$ is an irreducible $\mathscr{G}_{n+1}$-module. When $M$ is a finite-dimensional irreducible $\mathscr{G}_{n}$-module,  Zhao and the second author \cite{XZ} found a sufficient condition for $\widehat{M}$ to be an irreducible $\mathscr{G}_{n+1}$-module, that is, $\widehat{M}=\overline{M}$.
	
	When $M$ is a highest-weight irreducible $\mathscr{G}_{n}$-module,
	\begin{eqnarray}
		\overline{M}=V_{n+1}(\lambda) \label{1.15}
	\end{eqnarray}
	is a highest-weight irreducible $\mathscr{G}_{n+1}$-module. Indeed, $\lambda$ can be any dominant integral weight of $\mathscr{G}_{n+1}$. In particular, when $M=\mbb Cv_0$ is the one-dimensional trivial $\msr G_n$-module, we identify
$\widehat{M}$ with $\msr A$ by: $f\otimes v_0\leftrightarrow f$ for $\forall f\in \msr A$. Suppose $c=-k$ and $k$ is an nonnegative integer. Then
$\overline{M}=\msr A_0+\sum_{i=1}^k (\msr A_i+\eta^i\msr
A_{k-i})$ is an irreducible $\msr G_{n+1}$-module with
highest weight $k\lmd_1$ (cf. \cite{X2}), where $\eta=\eta_{2n}$ in the even case and $\eta=\eta_{2n+1}$ in the odd case. Our early initial motivation is to generalize the above result.

Starting from a  first-order differential operator representation of $\mathfrak{o}(3)$ in odd case and $\mathfrak{o}(4)$ in even case, respectively, we apply (\ref{1.9})-(\ref{1.12}) inductively to obtain
an  universal first-order differential operator realization  for any highest-weight
 irreducible representation of $\mathfrak{o}(2n+3)$ in $(n+1)^2$ variables (cf. Theorem \ref{T3}) and that of $\mathfrak{o}(2n+2)$ in $n(n+1)$ variables (cf. Theorem \ref{Th1}).
  When the highest weight is dominant integral, we determine the corresponding finite-dimensional irreducible module of $\mathfrak{o}(2n+3)$ explicitly in Theorem \ref{T8} and that of $\mathfrak{o}(2n+2)$ explicitly in Theorem \ref{Th2}.

Let $\mbb R^m$ be the natural Euclidean space with the standard basis $\{\ves_1,\ves_2,...,\ves_m\}$.  Denote by $\mbb N$ the set of nonnegative integers. In this paper, we take the set of positive roots of $o(2m+1)$ as
  \begin{equation} \Phi_{2m+1}^+=\{\ves_i,\ves_r\pm \ves_s\mid i\in\ol{1,m},\;1\leq s<r\leq m\}\end{equation}
  and then the set of dominant integral weights is
  \begin{equation} \Lmd_{2m+1}^+=\{\sum_{i=1}^m\ell_i\ves_i\mid\ell_s\in \mbb N/2,\;s\in\ol{1,m};\;\ell_{r+1}-\ell_r\in\mbb N,\;r\in\ol{1,m-1}\}.\end{equation}
  Denote by $d_k(\lmd)$ the dimension of the finite-dimensional irreducible $\mathfrak{o}(k)$-module with highest weight $\lmd$.

   Suppose that $\ol{M}$ is a finite-dimensional irreducible $\mathfrak{o}(2n+3)$-module with highest weight $\lmd=\sum_{i=1}^{n+1}\mu_i\ves_i\in  \Lmd_{2n+3}^+$.  In order to find $\dim\ol{M}_r$, we have used the inclusion-exclusion principle to calculate $\dim \wht{M}_r/\ol{M}_r$. Then we get
   $\dim\ol{M}_r$ and sum up to lead the following identities
   \begin{eqnarray}& &2\sum_{s=2}^{n+2}(-1)^{n-s}\binom{n+s+\frac{1}{2}+\mu_{s+1}}{2n+1} d_{2n+1}(\sum\limits_{i=1}^s\mu_i\varepsilon_i+\!\!\sum\limits_{i=s+1}^n(\mu_{i+1}+1)\varepsilon_i)\nonumber\\ &=& d_{2n+3}(\sum_{i=1}^{n+1}\mu_i\ves_i)  \end{eqnarray}
if $2\mu_{n+1}$ is odd, and
\begin{eqnarray}& &\sum\limits_{s=0}^n(-1)^{n-s}\binom{n\!+\!s\!+\!\mu_{s+1}}{2n}\frac{2s\!+\!2\mu_{s+1}\!+\!1}{2n\!+\!1} d_{2n+1}(\sum\limits_{i=1}^s\mu_i\varepsilon_i+\!\!\sum\limits_{i=s+1}^n(\mu_{i+1}+1)\varepsilon_i)\nonumber\\&=&
d_{2n+3}(\sum_{i=1}^{n+1}\mu_i\ves_i)
         \end{eqnarray}
when $2\mu_{n+1}$ is even. Note that Macdonald identities are of the form of an alternating sum equal to a product (cf. \cite{Mi}), and $d_k(\lmd)$ is also a product by Weyl's character formula (e.g., cf. the  corollary in Subsection 24.3 of \cite{H}). The above identities can be viewed as finite analogues of the Macdonald identities.

   Denote by $\lmd_i$ the $i$th fundamental weight.  When $\lmd=k\sum_{i=1}^{n+1}\lmd_i=\sum_{i=1}^{n+1}(i-1/2)k\ves_i$, $\ol{M}$ is a Steinberg module and the above equations are equivalent to the following identities:
   \begin{eqnarray}
	      	2\sum_{s=0}^{n}(-1)^{n-s}\binom{n+(s+\frac{1}{2})(k+1)}{2n+1}\binom{2n+1}{n-s}=(k+1)^{2n+1}
	      \end{eqnarray}
	  if  $k$ is odd, and
\begin{eqnarray}
	      	\sum\limits_{s=0}^{n}(-1)^{n-s}\binom{n-\frac{1}{2}+(s+\frac{1}{2})(k+1)}{2n}\binom{2n+1}{n-s}\frac{2s+1}{2n+1}=(k+1)^{2n}
	      \end{eqnarray}
    when $k$ is even. We have checked an authoritative encyclopedia of combinatorial identities \cite{Ri}, the above identities are not there.
They can be viewed as finite analogues of the  Jacobi triple and quintuple product identities. Furthermore, we have also found similar identities (\ref{4.151}),(\ref{4.152}) and (\ref{4.163}) for $\mathfrak{o}(2n+2)$.

The paper is organized as follows. In section 2, we derive the  universal first-order differential operator realization  for any highest-weight
 irreducible representation of $\mathfrak{o}(2n+3)$ and determine the corresponding finite-dimensional irreducible module of $\mathfrak{o}(2n+3)$ explicitly when the
highest weight is dominant integral. More detailed information is given for the weights $k_1\lmd_1+k_2\lmd_2$ and $k\lmd_i$. Similar results for $\mathfrak{o}(2n+2)$ are obtained in Section 3. In Section 4, we find all the $\msr G_n$-singular vectors in $\wht M_r$ and $\ol{ M}_r$. Then we use them and inclusion-exclusion principle to derive finite analogues of the Macdonald identities and the identities as those in (1.24) and (1.25).

Through out this paper we make the following conventions: Let $\mathbb{C}$, $\mathbb{R}$, $\mathbb{Z}$ and $\mathbb{N}$ denote the complex number field, the real number field, the ring of integers and the set of nonnegative integers respectively. For $i,j\in\mathbb{Z}$, denote the set $\{k\in\mathbb{Z}\mid i\leqslant k\leqslant j\}$ by $\overline{i,j}$, note that $\overline{i,j}$ is nonempty if and only if $i\leqslant j$. Let $\delta_{i,j}$ denote the Kronecker symbol, namely, $\delta_{i,j}$ equals to $1$ when $i=j$ and $0$ otherwise.

	\section{Construction for Odd Orthogonal Lie Algebras}
	
	In this section, we start from $\mathfrak{o}(3)$ and repeatedly use (\ref{1.9}) and (\ref{1.10})-(\ref{1.12}) to obtain the representation of $\mathfrak{o}(2n+3)$. Moreover, we  present two special cases with particular highest weights at the end of this section.
	
	Recall that the odd orthogonal Lie algebra in (\ref{1.5}), we take the subspace of $\mathfrak{o}(2n+3)$
	\begin{eqnarray}
		H=\sum_{i=1}^{n+1}\mathbb{C}A_{i,i}
	\end{eqnarray}
	as a Cartan subalgebra and define $\{\varepsilon_i ~|~i\in\overline{1,n+1}\}\subset H^*$ by
	\begin{eqnarray}
		\varepsilon_i(A_{j,j})=\delta_{i,j}.
	\end{eqnarray}
	Then  finite-dimensional irreducible $\mathfrak{o}(2n+3)$-modules are in one-to-one correspondence with highest weight $\lambda=\sum_{i=1}^{n+1}\mu_i\varepsilon_i$ such that
	\begin{eqnarray}
		2\mu_1\in\mathbb{N} \text{ and } \mu_{i+1}-\mu_i\in\mathbb{N} \text{ for }i=\overline{1,n}.\label{2.3}
	\end{eqnarray}
	Note that $A_{i,j}=-A_{-j,-i}$ in (\ref{1.3})-(\ref{1.4}). We choose $\{A_{i,j}~|~0\leqslant |j| < i\leqslant n+1\}$ as positive root vectors, and $\{A_{i,j}~|~0\leqslant |i|< j\leqslant n+1\}$ as negative root vectors. In particular,
	\begin{eqnarray}
		\mathfrak{o}(2n+3)_+=\sum_{0\leqslant |j| < i\leqslant n+1}\mathbb{C}A_{i,j},~\mathfrak{o}(2n+3)_-=\sum_{0\leqslant |i|< j\leqslant n+1}\mathbb{C}A_{i,j}. \label{4.3}
	\end{eqnarray}
	are the nilpotent subalgebra of positive root vectors and the nilpotent subalgebra of negative root vectors, respectively. A $singular~vector$ of an $\mathfrak{o}(2n+3)$-module $V$ is a weight vector annihilated by the elements all positive root vectors. Set
	\begin{eqnarray}
		\mathscr{G}_0=\mathfrak{o}(2n+1)+\mathbb{C}A_{n+1,n+1},~\mathscr{G}_+=\sum_{j=-n}^n\mathbb{C}A_{n+1,j},~\mathscr{G}_-=\sum_{i=-n}^n\mathbb{C}A_{i,n+1}. \label{4.5}
	\end{eqnarray}
	Then $\mathscr{G}_{\pm}$ are abelian Lie subalgebras of $\mathfrak{o}(2n+3)$ and $\mathscr{G}_0$ is a reductive Lie subalgebra of $\mathfrak{o}(2n+3)$. Moreover,
	\begin{eqnarray}
		\mathfrak{o}(2n+3)=\mathscr{G}_-+\mathscr{G}_0+\mathscr{G}_+,~~[\mathscr{G}_0,\mathscr{G}_{\pm}]=\mathscr{G}_{\pm},\;[\mathscr{G}_+,\msr G_-]\subset\msr G_0.~\label{4.6}
	\end{eqnarray}
	By (\ref{1.9}) and (\ref{1.10})-(\ref{1.12}),
	\begin{eqnarray}
		\mathscr{G}_+(1\otimes M)=\{0\},~U(\mathscr{G}_0)(1\otimes M)=1\otimes M. \label{4.7}
	\end{eqnarray}
	Thus
	\begin{eqnarray}
		\overline{M}=U(\mathfrak{o}(2n+3))(1\otimes M)=U(\mathscr{G}_-)(1\otimes M). \label{4.8}
	\end{eqnarray}
	Moreover, (\ref{1.12}) and (\ref{1.14}) imply
	\begin{eqnarray}
		\overline{M}_r=\mathscr{G}_-^r(1\otimes M). \label{4.9}
	\end{eqnarray}
	
	We start with $n=1$. Let $\msr A_{(0)}=\mbb C[y]$ be the polynomial algebra in $y$. Fix $\mu\in\mbb C$, we have the following
representation of $\mathfrak{o}(3)$ on $\msr A_{(0)}$:
\begin{eqnarray}
		A_{1,1}|_{\msr A_{(0)}}=-y\partial_y+\mu,~A_{0,1}|_{\msr A_{(0)}}=-\frac{1}{2}y^2\partial_y+\mu y,~A_{0,-1}|_{\msr A_{(0)}}=-\partial_y.  \label{4.11}
	\end{eqnarray}
Moreover,
\begin{eqnarray}M_0=U(\mathfrak{o}(3))(1)=\sum_{i=0}^\infty\mbb CA_{0,1}^i(1)\end{eqnarray}
forms an irreducible $\mathfrak{o}(3)$-submodule with highest-weight vector $1$ of weight $2\mu\lmd_1$.
Recall that $\mbb N$ is the set of nonnegative integers. If $\mu\not\in \mbb N/2$, then $M_0=\msr A_{(0)}$. When $\mu\in\mbb N/2$,
$M_0=\sum_{i=0}^{2\mu}\mathbb{C} y^i$ is $(2\mu+1)$-dimensional irreducible $\mathfrak{o}(3)$-module.

Taking $M=\msr A_{(0)}$ in (1.10), we have	
	\begin{eqnarray}
		\widehat{\msr A_{(0)}}=\mathbb{C}[x_{-1},x_0,x_{1}]\otimes_{\mbb C}\mbb C[y]\cong \msr A_{(1)}=\mathbb{C}[x_{-1},x_0,x_{1},y] \label{4.13}
	\end{eqnarray}
with the representation of $\mathfrak{o}(5)$ given in (1.14)-(1.16), where $c$ is a fixed constant. Then $1\in\msr A_{(1)}$ is an $\mathfrak{o}(5)$-singular vector
with weight \begin{eqnarray}
		\lambda=\mu\varepsilon_1-c\varepsilon_2,\quad\mbox{redenoted as}\;\mu_1\ves_1+\mu_2\ves_2. \label{4.14}
	\end{eqnarray}
Take $\msr G_-$ in (2.5) with $n=1$. Then
\begin{eqnarray}M_1=U(\mathfrak{o}(5))(1)=U(\msr G_-)(M_0)\end{eqnarray}
forms a highest-weight irreducible $\mathfrak{o}(5)$-module, which is of finite dimension if and only if (2.3) with $n=2$ holds.

	According to (1.8), we simply redenote
	\begin{eqnarray}
		D=x_{-1}\partial_{x_{-1}}+x_0\partial_{x_0}+x_1\partial_{x_1},~~\eta=\frac{1}{2}x_0^2+x_1x_{-1},\label{4.16}
	\end{eqnarray}
		The representation of $\mathfrak{o}(5)$ in (1.14)-(1.16) with $n=1$ becomes:	
\begin{equation}A_{1,1}|_{\msr A_{(1)}}=x_1\partial_{x_1}-x_{-1}\partial_{x_{-1}}-y\partial_y+\mu_1,~A_{0,1}|_{\msr A_{(1)}}=x_0\partial_{x_1}-x_{-1}\partial_{x_0}-\frac{1}{2}y^2\partial_y+
\mu_1y,\label{4.17}\end{equation}
\begin{equation}A_{0,-1}|_{\msr A_{(1)}}=x_0\partial_{x_{-1}}-x_{1}\partial_{x_0}-\partial_y,~A_{2,1}|_{\msr A_{(1)}}=\partial_{x_1},\label{4.18}\end{equation}
\begin{equation}A_{0,-2}|_{\msr A_{(1)}}=-\partial_{x_0},~A_{2,-1}|_{\msr A_{(1)}}=\partial_{x_{-1}},~A_{2,2}|_{\msr A_{(1)}}=-D+\mu_2,\label{4.19}\end{equation}
\begin{equation}A_{1,2}|_{\msr A_{(1)}}=-x_1D+\eta\partial_{x_{-1}}+(\mu_2-\mu_1)x_1-(x_0-x_1y)\partial_{y},\label{4.20}\end{equation}
\begin{equation}A_{0,2}|_{\msr A_{(1)}}=-x_0D+\eta\partial_{x_0}+(\mu_2-\mu_1)x_0+(x_{-1}+\frac{1}{2}x_1y^2)\partial_y+\mu_1(x_0-x_1y),\label{4.21}\end{equation}
\begin{equation}A_{-1,2}|_{\msr A_{(1)}}=-x_{-1}D+\eta\partial_{x_1}+(\mu_2-\mu_1)x_{-1}-(x_{-1}y+\frac{1}{2}x_0y^2)\partial_y+\mu_1(2x_{-1}+x_0y).\label{4.22}\end{equation}
\pse
	
	Next we are going to realize the general case. Let
	\begin{eqnarray}
		\msr A_{(n)}=\mathbb{C}[x_{i,j}\mid0\leqslant |j|\leqslant i\leqslant n]  \label{4.23}
	\end{eqnarray}
	be the polynomial algebra in $(n+1)^2$ variables. In particular, we treat
\begin{equation}x_{0,0}=y,\;x_{1,-1}=x_{-1},\;x_{1,0}=x_0,\;x_{1,1}=x_1.\end{equation}

For later convenience, we denote
	\begin{equation}
		x_{i,j}=\left\lbrace
		\begin{array}{ll}
			\text{variables~} x_{i,j}, &\text{if~}0\leqslant |j|\leqslant i\leqslant n,\\
			1, &\text{if~}j=i+1,\\
			0, &\text{if~}j>i+1
		\end{array}
		\right.  \label{4.24}
	\end{equation}
	and define inductively on the second index $j\in\overline{i+1,n+1}$ as 
  \begin{eqnarray}
		&x_{i,-(i+1)}=-\frac{1}{2}\sum\limits_{r=-i}^ix_{i,r}x_{i,-r},\label{4.25}\label{con2} \\
		&x_{i,-j}=-\sum\limits_{r=1-j}^{j-1}x_{i,r}x_{j-1,-r}\;\;\text{for}\;i+2\leqslant j\leqslant n+1. \label{4.26}\label{con3}
	\end{eqnarray}
	Moreover, we take the notation
	\begin{eqnarray}
		\alpha_i=(\alpha_{i,-i},\cdots,\alpha_{i,i})\in\mathbb{N}^{2i+1}\;\; \for\;\;i\in\overline{0,n}, \label{4.27}
	\end{eqnarray}
	and
	\begin{eqnarray}
		\alpha=(\alpha_0,\alpha_1,\cdots,\alpha_n)\in\mathbb{N}^{(n+1)^2}. \label{4.28}
	\end{eqnarray}
	Denote  the monomials
	\begin{eqnarray}
		X_i^{\alpha_i}=\prod_{j=-i}^ix_{i,j}^{\alpha_{i,j}}~\text{and}~X^{\alpha}=\prod_{i=0}^nX_{i}^{\alpha_i}.~\label{4.29}
	\end{eqnarray}
	
	For $0\leqslant t\leqslant s\leqslant n$ and $\Theta=(\theta_1,\cdots,\theta_{s-t+1})\in\mathbb{Z}^{s-t+1}$ with $\theta_{k}\geqslant-(n+1)$ for $k\in\overline{1,s-t+1}$, we denote
	\begin{eqnarray}
		F_{s,t}(\Theta)=\begin{vmatrix}
			x_{t,\theta_1}&x_{t,\theta_2}&\cdots&x_{t,\theta_{s-t+1}}\\
			x_{t+1,\theta_1}&x_{t+1,\theta_2}&\cdots&x_{t+1,\theta_{s-t+1}}\\
			\vdots&\vdots&\vdots&\vdots\\
			x_{s,\theta_1}&x_{s,\theta_2}&\cdots&x_{s,\theta_{s-t+1}}
		\end{vmatrix}.\label{4.30}
	\end{eqnarray}
	In particular, we take $f_{t,t}(r)=x_{t,r}$ and for $t<s$,
	\begin{eqnarray}
			f_{s,t}(r)&=&F_{s,t}(t+1,t+2,\cdots, s, r)\nonumber\\
			&=&\begin{vmatrix}
				1&0&\cdots&0&x_{t,r}\\
				x_{t+1,t+1}&1&\ddots&\vdots&x_{t+1,r}\\
				\vdots&\vdots&\ddots&0&\vdots\\
				x_{s-1,t+1}&x_{s-1,t+2}&\cdots&1&x_{s-1,r}\\
x_{s,t+1}&x_{s,t+2}&\cdots&x_{s,s}&x_{s,r}
			\end{vmatrix}.\label{4.31}
	\end{eqnarray}
	
    Next we prove some lemmas which will be used frequently later.
	\begin{lemma}\label{L1}
		For $0\leqslant t< s\leqslant n$ and $|r|\leqslant n+1$, the following equations hold,
		\begin{enumerate}[(i)]
			\item $f_{s,t}(r)=x_{s,r}-\sum\limits_{k=t+1}^{s}x_{s,k}f_{k-1,t}(r)$, 
			\item $f_{s,t}(r)=x_{s,r}-\sum\limits_{k=t}^{s-1}x_{k,r}f_{s,k+1}(k+1)$. 
		\end{enumerate}
		
	\end{lemma}
    \begin{proof}
    	Note that the algebraic cofactor of $x_{s,k}$ in (\ref{4.31}) is
    	\begin{eqnarray}
    		& &(-1)^{s+k+1}\begin{vmatrix}
    			1&0&\cdots&0&0&0&\cdots&0&x_{t,r}\\
    			x_{t+1,t+1}&1&\cdots&0&0&0&\cdots&0&x_{t+1,r}\\
    			\vdots&\vdots&\ddots&\vdots&\vdots&\vdots&\ddots&\vdots&\vdots\\
    			x_{k-1,t+1}&x_{k-1,t+2}&\cdots&x_{k-1,k-1}&0&0&\cdots&0&x_{k-1,r}\\
    			x_{k,t+1}&x_{k,t+2}&\cdots&x_{k,k-1}&1&0&\cdots&0&x_{k,r}\\
    			x_{k+1,t+1}&x_{k+1,t+2}&\cdots&x_{k+1,k-1}&x_{k+1,k+1}&1&\cdots&0&x_{k+1,r}\\
    			x_{s-1,t+1}&x_{s-1,t+2}&\cdots&x_{s,k-1}&x_{s-1,k+1}&x_{s-1,k+2}&\cdots&1&x_{s-1,r}
    		\end{vmatrix}\nonumber\\&=&-\begin{vmatrix}
    		1&0&\cdots&0&x_{t,r}&0&0&\cdots&\;0\;\\
    		x_{t+1,t+1}&1&\cdots&0&x_{t+1,r}&0&0&\cdots&0\\
    		\vdots&\vdots&\ddots&\vdots&\vdots&\vdots&\vdots&\ddots&\vdots\\
    		x_{k-1,t+1}&x_{k-1,t+2}&\cdots&x_{k-1,k-1}&x_{k-1,r}&0&0&\cdots&0\\
    		x_{k,t+1}&x_{k,t+2}&\cdots&x_{k,k-1}&x_{k,r}&1&0&\cdots&0\\
    		x_{k+1,t+1}&x_{k+1,t+2}&\cdots&x_{k+1,k-1}&x_{k+1,r}&x_{k+1,k+1}&1&\cdots&0\\
    		x_{s-1,t+1}&x_{s-1,t+2}&\cdots&x_{s,k-1}&x_{s-1,r}&x_{s-1,k+1}&x_{s-1,k+2}&\cdots&1
    	\end{vmatrix}\nonumber\\
     &=&-\begin{vmatrix}
     	1&0&\cdots&0&x_{t,r}\\
     	x_{t+1,t+1}&1&\cdots&0&x_{t+1,r}\\
     	\vdots&\vdots&\ddots&\vdots&\vdots\\
     	x_{k-1,t+1}&x_{k-1,t+2}&\cdots&x_{k-1,k-1}&x_{k-1,r}
     \end{vmatrix}=-f_{k-1,t}(r).
    	\end{eqnarray}
       Thus expending the determinant in (\ref{4.31}) according to the last row we obtain $(i)$. Next we prove $(ii)$ by backward induction on $t\in\overline{0,s-1}$. If $t=s-1$,
       \begin{eqnarray}
       	f_{s-1,s}(r)=\begin{vmatrix}
       		1&x_{s-1,r}\\x_{s,s}&x_{s,r}
       	\end{vmatrix}=x_{s,r}-x_{s-1,r}x_{s,s}=x_{s,r}-x_{s-1,r}f_{s,s}(s).
       \end{eqnarray}
        Suppose that $t <s-1$,
       \begin{eqnarray}
       	f_{s,t}(r)=f_{s,t+1}(r)-x_{t,r}f_{s,t+1}(t+1)=x_{s,r}-\sum\limits_{k=t}^{s-1}x_{k,r}f_{s,k+1}(k+1).
       \end{eqnarray}
       Therefore, we obtained $(ii)$.
    \end{proof}\pse
	
	\begin{lemma}\label{L2}
		For $0\leqslant t\leqslant s\leqslant n$ and $r\geqslant 0$, we have
		\begin{equation}
			\sum\limits_{k=-(n+1)}^{n+1}x_{r,k}f_{s,t}(-k)=0. \label{4.35}
		\end{equation}
	\end{lemma}
    \begin{proof}
    	For $0\leqslant i\leqslant j\leqslant n$, note that $x_{j,j+1}=1$ and $x_{i,l-1}=x_{j,l}=0$ when $l>j+1$ by the definition in (\ref{4.24})-(\ref{4.26}). So we have
    	\begin{eqnarray}
    		\sum_{k=-(n+1)}^{n+1}x_{i,k}x_{j,-k}=\sum_{k=-(j+1)}^{j}x_{i,k}x_{j,-k}=x_{i,-(j+1)}+\sum_{k=-j}^{j}x_{i,k}x_{j,-k}=0. \label{4.34}
    	\end{eqnarray}

        Moreover,
        \begin{eqnarray}
        	\sum_{k=-(n+1)}^{n+1}x_{r,k}x_{r,-k}=\sum_{k=-r}^{r}x_{r,k}x_{r,-k}+2x_{r,-(r+1)}=0.
        \end{eqnarray}

        Therefore,
        \begin{eqnarray}
        	\sum\limits_{k=-\infty}^{-\infty}x_{r,k}f_{s,t}(-k)&=&\begin{vmatrix}
        		1&0&\cdots&0&\sum\limits_{k=-(n+1)}^{n+1}x_{r,k}x_{t,-k}\\
        		x_{t+1,t+1}&1&\ddots&\vdots&\sum\limits_{k=-(n+1)}^{n+1}x_{r,k}x_{t+1,-k}\\
        		\vdots&\vdots&\ddots&0&\vdots\\
        		x_{s-1,t+1}&x_{s-1,t+2}&\cdots&1&\sum\limits_{k=-(n+1)}^{n+1}x_{r,k}x_{s-1,-k}\\
        		x_{s,t+1}&x_{s,t+2}&\cdots&x_{s,s}&\sum\limits_{k=-(n+1)}^{n+1}x_{r,k}x_{s,-k}
        	\end{vmatrix}\nonumber\\
        &=&\begin{vmatrix}
        	1&0&\cdots&0&\;0\;\\
        	x_{t+1,t+1}&1&\ddots&\vdots&0\\
        	\vdots&\vdots&\ddots&0&\vdots\\
        	x_{s-1,t+1}&x_{s-1,t+2}&\cdots&1&0\\
        	x_{s,t+1}&x_{s,t+2}&\cdots&x_{s,s}&0
        \end{vmatrix}=0.
        \end{eqnarray}
    \end{proof}\pse
	From Lemma \ref{L2}, we obtain an expression of the polynomial defined in (\ref{4.26}),
	\begin{eqnarray}
		x_{i,-j}=-\sum_{k=-t}^ix_{i,k}f_{j-1,t}(-k)-f_{j-1,t}(-i-1),\;\;\forall\; t\in\overline{0,i}.\label{4.36}
	\end{eqnarray}
	
	Suppose that we have used   (\ref{1.10})-(\ref{1.12}) successively to give a representation of $\mathfrak{o}(2n+1)$ on $\msr A_{(n-1)}$ such that
\begin{equation}
	M_{n-1}=U(\mathfrak{o}(2n+1))(1) \label{2.32}
	\end{equation}
is an irreducible $\mathfrak{o}(2n+1)$-submodule with highest-weight vector $1$ of weight $\lmd'=\sum_{i=1}^n\mu_i\ves_i$ and is of finite dimension if and only if (2.3) with $n+1$ replaced by $n$ holds. Take $M=\msr A_{(n-1)}$ in (1.10). Now (\ref{1.10})-(\ref{1.12}) with
	\begin{eqnarray}
		c=-\mu_{n+1}\text{ and }x_i=x_{n,i} \text{ for }i\in\overline{-n,n} \label{4.38}
	\end{eqnarray}
give a representation of $\mathfrak{o}(2n+3)$ on
\begin{eqnarray}\wht{\msr A_{(n-1)}}=\mbb C[x_{n,-n},...,x_{n,0},...,x_{n.n}]\otimes \msr A_{(n-1)}\cong \msr A_{(n)}.\end{eqnarray}
Take $\msr G_-$ in (2.5). Then
\begin{eqnarray}M_n=U(\mathfrak{o}(2n+3))(1)=U(\msr G_-)(M_{n-1})\label{2.35}\end{eqnarray}
is an irreducible $\mathfrak{o}(2n+3)$-submodule with highest-weight vector $1$ of weight $\lmd=\sum_{i=1}^{n+1}\mu_i\ves_i$ and is of finite dimension if and only if (2.3) holds. Letting $\mu_0=0$,  we have

	\begin{theorem}\label{T3}
	The representation formulas of $\mathfrak{o}(2n+3)$ on $\msr A_{(n)}$ in terms of the notations in (\ref{4.24})-(\ref{4.31}) are as follows.\psp
	
    For $0< |i|<j\leqslant n+1$,
    \begin{eqnarray}
		&A_{j,i}|_{\msr A_{(n)}}=\sum\limits_{r=j}^n(x_{r,j}\partial_{x_{r,i}}-x_{r,-i}\partial_{x_{r,-j}})+\partial_{x_{j-1,i}},\label{4.42}\label{a1}
	\end{eqnarray}
	
    For $1\leqslant j\leqslant n+1$,
	\begin{eqnarray}
	    A_{j,j}|_{\msr A_{(n)}}&=&\sum\limits_{r=j}^n(x_{r,j}\partial_{x_{r,j}}-x_{r,-j}\partial_{x_{r,-j}})-\sum\limits_{s=1-j}^{j-1}x_{j-1,s}\partial_{x_{j-1,s}}
	    +\mu_j, \label{4.43}\\
	    \;\nonumber\\
	    A_{j,0}|_{\msr A_{(n)}}&=&\sum\limits_{r=j}^n(x_{r,j}\partial_{x_{r,0}}-x_{r,0}\partial_{x_{r,-j}})+\partial_{x_{j-1,0}},\label{4.42.2}\\
	    \;\nonumber\\
     	A_{0,j}|_{\msr A_{(n)}}&=&\sum\limits_{r=j}^nx_{r,0}\partial_{x_{r,j}}-\sum_{r=0}^{n}x_{r,-j}\partial_{x_{r,0}}\nonumber\\
     	& &-\sum_{r=0}^{j-1}\big[ x_{r,0}\sum\limits_{s=-r}^rf_{j-1,r}(s)\partial_{x_{r,s}}-(\mu_{r+1}-\mu_r)f_{j-1,r}(0)\big].
     	\label{5.45}
    \end{eqnarray}
	
	For $1\leqslant i<j\leqslant n+1$,
	\begin{eqnarray}
        A_{i,j}|_{\msr A_{(n)}}&=&\sum\limits_{r=j}^nx_{r,i}\partial_{x_{r,j}}-
        \sum_{r=i}^nx_{r,-j}\partial_{x_{r,-i}}-\sum\limits_{s=1-i}^{i-1}\!\!f_{j-1,i-1}(s)\partial_{x_{i-1,s}}\nonumber\\&&-\sum\limits_{r=i}^{j-1}\big[ x_{r,i}\sum\limits_{s=-r}^rf_{j-1,r}(s)\partial_{x_{r,s}}-(\mu_{r+1}-\mu_r)f_{j-1,r}(i)\big],
			\label{4.44}
	\end{eqnarray}
    and
	\begin{eqnarray}
		A_{-i,j}|_{\msr A_{(n)}}&=&\sum\limits_{r=j}^nx_{r,-i}\partial_{x_{r,j}}-\sum\limits_{r=i}^{n}x_{r,-j}\partial_{x_{r,i}}\nonumber\\
		& &-\sum\limits_{r=0}^{j-1}\big[ x_{r,-i}\sum\limits_{s=-r}^rf_{j-1,r}(s)\partial_{x_{r,s}}+(\mu_r-\mu_{r+1})f_{j-1,r}(-i)\big]\nonumber \\
		& &+\sum\limits_{r=0}^{i-1}\big[ x_{r,-j}\sum\limits_{s=-r}^rf_{i-1,r}(s)\partial_{x_{r,s}}+(\mu_r-\mu_{r+1})f_{i-1,r}(-j)\big].\label{4.46}\label{a6}
	\end{eqnarray}
	\end{theorem}
	\begin{proof}
		When $n=1$, we replace the variables in formulas in (\ref{4.17})-(\ref{4.22}) by
		\begin{eqnarray}
			x_{0,0}=y_0,\;x_{1,i}=x_i,\;\for\;i=-1,0,1.
		\end{eqnarray}
	Moreover, we redenote
	\begin{eqnarray}
		D_1=x_{1,-1}\partial_{x_{1,-1}}+x_{1,0}\partial_{x_{1,0}}+x_{1,1}\partial_{x_{1,1}}.
	\end{eqnarray}
	    Note that (\ref{4.20})-(\ref{4.22}) can be written as
	    \begin{eqnarray}
	    A_{1,2}|_{\msr A_{(1)}}\!\!&=&\!\!-x_{1,1}D_1-x_{1,-2}\partial_{x_{1,-1}}+(\mu_2-\mu_1)x_{1,1}-\begin{vmatrix}
	    		1\!&\!x_{0,0}\\x_{1,1}\!&\!x_{1,0}
	    	\end{vmatrix}\partial_{x_{0,0}},\\
	    	A_{0,2}|_{\msr A_{(1)}}\!\!&=&\!\!-x_{1,0}D_1-x_{1,-2}\partial_{x_{1,0}}+(\mu_2-\mu_1)x_{1,0}\nonumber\\
	    	\!\!& &\!\!-\big(x_{0,0}\begin{vmatrix}
	    		1\!&\!x_{0,0}\\x_{1,1}\!&\!x_{1,0}
	    	\end{vmatrix}+x_{0,-2}\big)\partial_{x_{0,0}}+\mu_1\begin{vmatrix}
	    	1\!&\!x_{0,0}\\x_{1,1}\!&\!x_{1,0}
    	\end{vmatrix},\\
	    	A_{-1,2}|_{\msr A_{(1)}}\!\!&=&\!\!-x_{1,-1}D_1-x_{1,-2}\partial_{x_{1,1}}+(\mu_2-\mu_1)x_{1,-1}\nonumber\\
	    	\!\!& &\!\!+\big(x_{0,-2}x_{0,0}\!-\!x_{0,-1}\begin{vmatrix}
	    		1\!&\!x_{0,0}\\x_{1,1}\!&\!x_{1,0}
	    	\end{vmatrix}\big)\partial_{x_{0,0}}\!-\!\mu_1(x_{0,-2}\!-\!\begin{vmatrix}
	    	1\!&\!x_{0,-1}\\x_{1,1}\!&\!x_{1,-1}
    	\end{vmatrix}).
	    \end{eqnarray}
		These formulas together with (\ref{4.17})-(\ref{4.19}) exactly coincide with (\ref{a1})-(\ref{a6}). Hence the theorem holds for $\mathfrak{o}(5)$. Suppose that the theorem holds for $\mathfrak{o}(2n+1)$.
		For $0\leqslant |i|,j<n+1$, applying induction on
		\begin{eqnarray}
			A_{i,j}|_{\msr A_{(n)}}=x_{n,i}\partial_{x_{n,j}}-x_{n,-j}\partial_{x_{n,-i}}+A_{i,j}|_{\msr A_{(n-1)}},
		\end{eqnarray}
	   we obtain following equations: For $0< |i|<j\leqslant n$,
	   \begin{eqnarray}
	   	&A_{j,i}|_{\msr A_{(n)}}=\sum\limits_{r=j}^n(x_{r,j}\partial_{x_{r,i}}-x_{r,-i}\partial_{x_{r,-j}})+\partial_{x_{j-1,i}},
	   \end{eqnarray}
	   
	   For $1\leqslant j\leqslant n$,
	   
	   \begin{eqnarray}
	   	A_{j,j}|_{\msr A_{(n)}}&=&\sum\limits_{r=j}^n(x_{r,j}\partial_{x_{r,j}}-x_{r,-j}\partial_{x_{r,-j}})-\sum\limits_{s=1-j}^{j-1}x_{j-1,s}\partial_{x_{j-1,s}}
	   	+\mu_j,\\
	   	\;\nonumber\\
	   	A_{j,0}|_{\msr A_{(n)}}&=&\sum\limits_{r=j}^n(x_{r,j}\partial_{x_{r,0}}-x_{r,0}\partial_{x_{r,-j}})+\partial_{x_{j-1,0}},\\
	   	\;\nonumber\\
	   	A_{0,j}|_{\msr A_{(n)}}&=&\sum\limits_{r=j}^nx_{r,0}\partial_{x_{r,j}}-\sum_{r=0}^{n}x_{r,-j}\partial_{x_{r,0}}\nonumber\\
	   	& &-\sum_{r=0}^{j-1}\big[ x_{r,0}\sum\limits_{s=-r}^rf_{j-1,r}(s)\partial_{x_{r,s}}-(\mu_{r+1}-\mu_r)f_{j-1,r}(0)\big].
	   \end{eqnarray}
	   
	   For $1\leqslant i<j\leqslant n$,
	   \begin{eqnarray}
	   	A_{i,j}|_{\msr A_{(n)}}&=&\sum\limits_{r=j}^nx_{r,i}\partial_{x_{r,j}}-
	   	\sum_{r=i}^nx_{r,-j}\partial_{x_{r,-i}}-\sum\limits_{s=1-i}^{i-1}\!\!f_{j-1,i-1}(s)\partial_{x_{i-1,s}}\nonumber\\&&-\sum\limits_{r=i}^{j-1}\big[ x_{r,i}\sum\limits_{s=-r}^rf_{j-1,r}(s)\partial_{x_{r,s}}-(\mu_{r+1}-\mu_r)f_{j-1,r}(i)\big],
	   \end{eqnarray}
	   and
	   \begin{eqnarray}
	   	A_{-i,j}|_{\msr A_{(n)}}&=&\sum\limits_{r=j}^nx_{r,-i}\partial_{x_{r,j}}-\sum\limits_{r=i}^{n}x_{r,-j}\partial_{x_{r,i}}\nonumber\\
	   	& &-\sum\limits_{r=0}^{j-1}\big[ x_{r,-i}\sum\limits_{s=-r}^rf_{j-1,r}(s)\partial_{x_{r,s}}+(\mu_r-\mu_{r+1})f_{j-1,r}(-i)\big]\nonumber \\
	   	& &+\sum\limits_{r=0}^{i-1}\big[ x_{r,-j}\sum\limits_{s=-r}^rf_{i-1,r}(s)\partial_{x_{r,s}}+(\mu_r-\mu_{r+1})f_{i-1,r}(-j)\big].
	   \end{eqnarray}
		
		By (\ref{1.11}), we have
		\begin{eqnarray}
			&A_{n+1,i}|_{\msr A_{(n)}}=\partial_{x_{n,i}},\;\for\;i\in\overline{-n,n},\\
			&A_{n+1,n+1}|_{\msr A_{(n)}}=-\sum\limits_{s=-n}^nx_{n,s}\partial_{x_{n,s}}+\mu_{n+1}.
		\end{eqnarray}
		
		Consider $A_{i,n+1}|_{\mathscr{A}(n)}$ for $i\in\overline{-n,n}$. Expression (\ref{1.12}) yields
		\begin{eqnarray}
			A_{i,n+1}|_{\msr A_{(n)}}&=&-x_{n,i}\sum\limits_{s=-n}^nx_{n,s}\partial_{x_{n,s}}
            -x_{n,-(n+1)}\partial_{x_{n,-i}}\nonumber\\ &&-\big(\sum\limits_{r=-n}^n\!x_{n,r}A_{i,r}-\mu_{n+1}x_{n,i}\big)|_{\msr A_{(n-1)}}.\label{4.54}
		\end{eqnarray}
		Since it is a first-order differential operator, we will use induction to figure out what the polynomial coefficient of each $\partial_{x_{t,s}}$ and the polynomial term  by calculation in parts.
		
		When $i\geqslant 0$, we denote
		\begin{equation}
			\mathfrak{D}_{i,r,t}=\left\lbrace \begin{array}{ll}
				-x_{t,i}\sum\limits_{s=-t}^tf_{r-1,t}(s)\partial_{x_{t,s}}-x_{t,-r}\partial_{x_{t,-i}},&\text{if}\;t\in\overline{i,n-1},~r\in\overline{t+1,n},\\
				x_{t,i}\partial_{x_{t,r}}-x_{t,-r}\partial_{x_{t,-i}},&\text{if}\;t\in\overline{i,n-1},~r\in\overline{-t,t},\\
				-\partial_{x_{t,-i}},&\text{if}\;t\in\overline{i,n-1},~r=-t-1,\\
				-\sum\limits_{s=-t}^tf_{r-1,t}(s)\partial_{x_{t,s}},&\text{if}\;t=i-1\geqslant 0,~r\in\overline{t+1,n},\\
				\partial_{x_{t,r}},&\text{if}\;t=i-1\geqslant0,~r\in\overline{-t,t},\\
				0,&\text{otherwise}.
			\end{array}
			\right.
		\end{equation}
	    By induction,
	    \begin{eqnarray}
	    	A_{i,r}|_{\msr A_{(n-1)}}=\sum_{t=0}^{n-1}\mathfrak{D}_{i,r,t}+\mathfrak{P}_{i,r},
	    \end{eqnarray}
		where $\mathfrak{P}_{i,r}$ is the polynomial term of $A_{i,r}|_{\mathscr{A}(n-1)}$. If $t\in\overline{i,n-1}$, by Lemma \ref{L1}, we have
		\begin{eqnarray}
			\!\!\sum_{r=-n}^n\!\!x_{n,r}\mathfrak{D}_{i,r,t}\!\!&=&\!\!x_{t,i}\sum\limits_{s=-t}^t\big(x_{n,s}-\!\!\sum\limits_{r=t+1}^n\!\!x_{n,r}f_{r-1,t}(s)\big)\partial_{x_{t,s}}-\!\!\!\sum\limits_{r=-(t+1)}^n\!\!\!x_{n,r}x_{t,-r}\partial_{x_{t,-i}}\nonumber\\
			\!\!&=&\!\!x_{t,i}\sum\limits_{s=-t}^tf_{n,t}(s)\partial_{x_{t,s}}+x_{t,-(n+1)}\partial_{x_{t,-i}},\label{4.56}
		\end{eqnarray}
	    and
	    \begin{eqnarray}
	    	\sum_{r=-n}^nx_{n,r}\mathfrak{D}_{i,r,i-1}=\sum\limits_{s=-i+1}^{i-1}f_{n,i-1}(s)\partial_{x_{i-1,s}}.\label{4.58}
	    \end{eqnarray}
It is obvious that $A_{i,r}|_{\mathscr{A}(n)}$ is free of $\partial_{x_{t,s}}$'s with $t<i-1$. Recall $f_{n,n}(i)=x_{n,i}$. The  polynomial term in $A_{i,n+1}|_{\msr A_{(n)}}$ is
		\begin{eqnarray}
			& &-c_nx_{n,i}-\sum_{r=-n}^nx_{n,r}\mathfrak{P}_{i,r}\nonumber\\
			&=&(\mu_{n+1}- \mu_i)x_{n,i}-\sum\limits_{t=i}^{n-1}\sum\limits_{r=t+1}^n(\mu_{t+1}-\mu_t)x_{n,r}f_{r-1,t}(i)\nonumber\\
			&=& (\mu_{n+1}-\mu_n)x_{n,i}+\sum\limits_{t=i}^{n-1}(\mu_{t+1}-\mu_t)(x_{n,i}-\sum\limits_{r=t+1}^nx_{n,r}f_{r-1,t}(i))\nonumber\\
			&=&(\mu_{n+1}-\mu_n)x_{n,i}+\sum\limits_{t=i}^{n-1}(\mu_{t+1}-\mu_t)f_{n,t}(i)\nonumber\\ 
			&=&\sum\limits_{t=i}^n(\mu_{t+1}-\mu_t)f_{n,t}(i).\label{4.59}
		\end{eqnarray}
		According to (\ref{4.54}) and the fact $f_{n,n}(s)=x_{n,s}$, we sum (\ref{4.56}) over $t\in\overline{i,n-1}$ together with  (\ref{4.58}) and (\ref{4.59}), it turns out to be 
		\begin{eqnarray}			
				A_{i,n+1}|_{\msr A_{(n)}}&=
				&-\sum\limits_{r=i}^{n}\big[ x_{r,i}\sum\limits_{s=-r}^rf_{n,r}(s)\partial_{x_{r,s}}-(\mu_{r+1}-\mu_r)f_{n,r}(i)+x_{r,-n-1}\partial_{x_{r,-i}}\big] \nonumber\\
				&&-\sum\limits_{s=-i+1}^{i-1}f_{n,i-1}(s)\partial_{x_{i-1,s}} \label{4.60}			
		\end{eqnarray}
		for $\in\overline{0,n}$. Note that the last term in the right hand of (\ref{4.60}) will be $0$ when $i=0$ since the scope of sum is empty.
		
		Next we will derive (\ref{4.46}) in the case of $j=n+1$. In this case, we assume $i\in\overline{1,n}$ and let
		\begin{equation}
			\mathfrak{D}_{-i,r,t}=\left\lbrace\begin{array}{ll}
				\!-x_{t,-i}\sum\limits_{s=-t}^tf_{r-1,t}(s)\partial_{x_{t,s}}-x_{t,-r}\partial_{x_{t,i}},&\!\text{if}\;t\in\overline{i,n-1},~r\in\overline{t\!+\!1,n},\\
				\!x_{t,-i}\partial_{x_{t,r}}-x_{t,-r}\partial_{x_{t,i}},&\!\text{if}\;t\in\overline{i,n-1},~r\in\overline{-t,t},\\
				\!-\partial_{x_{t,i}},&\!\text{if}\;t\in\overline{i,n-1},~r=-t\!-\!1,\\
				\!-x_{t,-i}\!\sum\limits_{s=-t}^t\!\!f_{r-1,t}(s)\partial_{x_{t,s}}\!+\!x_{t,-r}\!\sum\limits_{s=-t}^t\!\!f_{i-1,t}(s)\partial_{x_{t,s}},&\!\text{if}\;t\in\overline{0,i-1},~r\in\overline{t\!+\!1,n},\\
				\!x_{t,-i}\partial_{x_{t,r}}+x_{t,-r}\sum\limits_{s=-t}^tf_{i-1,t}(s)\partial_{x_{t,s}},&\!\text{if}\;t\in\overline{0,i-1},~r\in\overline{-t,t},\\
				\!\sum\limits_{s=-t}^tf_{i-1,t}(s)\partial_{x_{t,s}},&\!\text{if}\;t\in\overline{0,i-1},~r=-t\!-\!1,\\
				\!0,&\!\text{otherwise}.
			\end{array}\right.
		\end{equation}
		By inductional assumption,
		\begin{eqnarray}
			A_{-i,r}|_{\msr A_{(n-1)}}=\sum_{t=0}^{n-1}\mathfrak{D}_{-i,r,t}+\mathfrak{P}_{-i,r}.
		\end{eqnarray}
		where $\mathfrak{P}_{-i,r}$ is the polynomial term of $A_{-i,r}|_{\mathscr{A}(n-1)}$. Applying Lemma \ref{L1}, we have
		\begin{eqnarray}
			\!\!\sum_{r=-n}^n\!x_{n,r}\mathfrak{D}_{-i,r,t}\!\!\!&=&\!\!\!x_{t,-i}\sum\limits_{s=-t}^t\big(x_{n,s}-\!\sum\limits_{r=t+1}^nx_{n,r}f_{r-1,t}(s)\big)\partial_{x_{t,s}}-\!\!\!\sum\limits_{r=-t-1}^n\!\!\!x_{n,r}x_{t,-r}\partial_{x_{t,i}}\nonumber\\
			\!\!\!&=&\!\!\!x_{t,-i}\sum\limits_{s=-t}^tf_{n,t}(s)\partial_{x_{t,s}}+x_{t,-(n+1)}\partial_{x_{t,i}} \label{4.62}
		\end{eqnarray}
		for $t\in\overline{i,n-1}$ and
		\begin{eqnarray}			\!\!\sum_{r=-n}^n\!x_{n,r}\mathfrak{D}_{-i,r,t}\!\!\!&=&\!\!\!x_{t,-i}\!\sum\limits_{s=-t}^t\big(x_{n,s}\!-\!\!\sum\limits_{r=t+1}^n\!x_{n,r}f_{r-1,t}(s)\big)\partial_{x_{t,s}}+\!\!\!\!\!\sum\limits_{r=-t-1}^n\!\!\!x_{n,r}\big(x_{t,-r}\sum\limits_{s=-t}^t\!f_{i-1,t}(s)\partial_{x_{t,s}}\big)\nonumber\\
			\!\!\!&=&\!\!\!x_{t,-i}\sum\limits_{s=-t}^tf_{n,t}(s)\partial_{x_{t,s}}-x_{t,-(n+1)}\sum\limits_{s=-t}^tf_{i-1,t}(s)\partial_{x_{t,s}}.\label{4.64}
		\end{eqnarray}
		for $t\in\overline{0,i-1}$. Moreover, the polynomial term of $A_{-i,n+1}|_{\msr A_{(n)}}$ is
		\begin{eqnarray}
			& \!&\!\!-c_nx_{n,-i}-\sum_{r=-n}^nx_{n,r}\mathfrak{P}_{-i,r}\nonumber\\
			&=\!&\!\!\mu_{n+1}x_{n,-i}\!+\!\mu_ix_{n,-i}-\!\!\!\!\sum\limits_{0\leqslant t<r\leqslant n}\!\!\!(\mu_{t+1}\!-\!\mu_t)x_{n,r}f_{r-1,t}(-i)+\!\!\!\!\sum\limits_{0\leqslant t\leqslant i-1 \atop -t\!-\!1\leqslant r\leqslant n}\!\!\!(\mu_{t+1}\!-\!\mu_t)x_{n,r}f_{i-1,t}(-r)\nonumber\\				
			&=\!&\!\!\sum\limits_{t=0}^{n}(\mu_{t+1}-\mu_t)f_{n,t}(-i)-\sum\limits_{t=0}^{i-1}(\mu_{t+1}-\mu_t)f_{i-1,t}(-(n+1)), \label{4.65}
		\end{eqnarray}
		According to (\ref{4.54}), summing (\ref{4.62}) over $t\in\overline{i,n-1}$, (\ref{4.64}) over $t\in\overline{0,i-1}$ together with (\ref{4.65}), it turns out to be
      \begin{eqnarray}
      	& &A_{-i,n+1}|_{\msr A_{(n)}}\nonumber\\
      	&=&-\sum\limits_{t=0}^n\big[ x_{t,-i}\sum\limits_{s=-t}^tf_{n,t}(s)\partial_{x_{t,s}}+(\mu_t-\mu_{t+1})f_{n,t}(-i)\big]+\sum\limits_{t=i}^nx_{t,-(n+1)}\partial_{x_{t,i}}\nonumber \\
      	& &+\sum\limits_{t=0}^{i-1}\big[ x_{t,-(n+1)}\sum\limits_{s=-t}^tf_{i-1,t}(s)\partial_{x_{t,s}}+(\mu_t-\mu_{t+1})f_{i-1,t}(-(n+1))\big]. \label{4.66}
		\end{eqnarray}
		This completes the proof.
	\end{proof}\pse

	Next we will determine the finite-dimensional irreducible $\mathfrak{o}(2n+3)$-module $\overline{M}=V_{2n+3}(\lambda)$ with highest weight
	\begin{equation}
		\lambda=\mu_1\varepsilon_1+\mu_2\varepsilon_2+\cdots+\mu_{n+1}\varepsilon_{n+1},
	\end{equation}
	 Satisfying (\ref{2.3}). Recall that $x_{s,s+1}=1$ and $x_{s,t}=0$ if $t>s+1$ by the convention (\ref{4.24}), and $f_{n,s}(t)=0$ when $t>s$ by (\ref{4.31}). 
	 
	 For $i,j\in\overline{0,1+n}$, we denote the term of first-order differential operators on $A_{i,j}|_{\mathscr{A}(n)}$ by $\mathcal{D}_{i,j}$: (cf. (\ref{a1})-(\ref{a6}))
	 
	 \begin{eqnarray}
	 	\mathcal{D}_{j,j}&=&\sum\limits_{r=j}^n(x_{r,j}\partial_{x_{r,j}}-x_{r,-j}\partial_{x_{r,-j}})-\sum\limits_{s=1-j}^{j-1}x_{j-1,s}\partial_{x_{j-1,s}}, \label{d1}\\
	 	\;\nonumber\\
	 	\mathcal{D}_{j,0}&=&\sum\limits_{r=j}^n(x_{r,j}\partial_{x_{r,0}}-x_{r,0}\partial_{x_{r,-j}})+\partial_{x_{j-1,0}},\label{d2}\\
	 	\;\nonumber\\
	 	\mathcal{D}_{0,j}&=&\sum\limits_{r=j}^nx_{r,0}\partial_{x_{r,j}}-\sum_{r=0}^{n}x_{r,-j}\partial_{x_{r,0}}-\sum_{r=0}^{j-1}\big( x_{r,0}\sum\limits_{s=-r}^rf_{j-1,r}(s)\partial_{x_{r,s}}\big),
	 	\label{d3}
	 \end{eqnarray} 
	 if moreover, $|i|<j$,
	 \begin{eqnarray}
	 	&\mathcal{D}_{j,i}\;=\;\sum\limits_{r=j}^n(x_{r,j}\partial_{x_{r,i}}-x_{r,-i}\partial_{x_{r,-j}})+\partial_{x_{j-1,i}},\label{d4}
	 \end{eqnarray}
     if $0<i<j$,
	 \begin{eqnarray}
	 	\mathcal{D}_{i,j}&=&\sum\limits_{r=j}^nx_{r,i}\partial_{x_{r,j}}-
	 	\sum_{r=i}^nx_{r,-j}\partial_{x_{r,-i}}-\sum\limits_{s=1-i}^{i-1}\!\!f_{j-1,i-1}(s)\partial_{x_{i-1,s}}\nonumber\\&&-\sum\limits_{r=i}^{j-1}\big( x_{r,i}\sum\limits_{s=-r}^rf_{j-1,r}(s)\partial_{x_{r,s}}\big),
	 	\label{d5}
	 \end{eqnarray}
	 and
	 \begin{eqnarray}
	 	\mathcal{D}_{-i,j}&=&\sum\limits_{r=j}^nx_{r,-i}\partial_{x_{r,j}}-\sum\limits_{r=i}^{n}x_{r,-j}\partial_{x_{r,i}}\nonumber\\
	 	& &-\sum\limits_{r=0}^{j-1}\big( x_{r,-i}\sum\limits_{s=-r}^rf_{j-1,r}(s)\partial_{x_{r,s}}\big)+\sum\limits_{r=0}^{i-1}\big( x_{r,-j}\sum\limits_{s=-r}^rf_{i-1,r}(s)\partial_{x_{r,s}}\big).\label{d6}
	 \end{eqnarray}

	 \begin{lemma}\label{L4}
	 	For $0\leqslant |i|< j\leqslant n+1,\;r\in\overline{0,n}$ and $s\in\overline{-(n+1),n+1}$, we have
	 	\begin{enumerate}[(i)]
	 		\item $\mathcal{D}_{j,i}(x_{r,s})=\delta_{i,s}x_{r,j}-\delta_{-j,s}x_{r,-i}$.
	 		\item $\mathcal{D}_{j,j}(x_{r,s})=(\delta_{j,s}-\delta_{-j,s}-\delta_{j-1,r})x_{r,s}$.
	 		\item Moreover, if $j\leqslant r$, $$\mathcal{D}_{i,j}(x_{r,s})=\delta_{j,s}x_{r,i}-\delta_{-i,s}x_{r,-j},$$
	 		if $j>r,\;i\geqslant -r$, $$\mathcal{D}_{i,j}(x_{r,s})=\delta_{j,s}x_{r,i}-\delta_{-i,s}x_{r,-j}-x_{r,i}f_{j-1,r}(s),$$
	 		if $j>r,\;i<-r$,
	 		$$\mathcal{D}_{i,j}(x_{r,s})=\delta_{j,s}x_{r,i}-\delta_{-i,s}x_{r,-j}-x_{r,i}f_{j-1,r}(s)+x_{r,-j}f_{-i-1,r}(s).$$
	 	\end{enumerate}
	 \end{lemma}
     \begin{proof}
     	See Appendix.
     \end{proof}\pse
	 
	 For $0\leqslant t\leqslant s\leqslant n$, we set
	 \begin{eqnarray}
	 	\mathbb{J}_{s,t}=\{(\theta_1,\cdots,\theta_{s-t+1})\in\mathbb{Z}^{s-t+1}\mid -(n+1)\leqslant \theta_1<\cdots<\theta_{s-t+1}\leqslant s\}.
	 \end{eqnarray}
	 For $i\neq j$, $0\leqslant t\leqslant s\leqslant n$ and $\Theta\in\mathbb{J}_{s,t}$, we denote
	 \begin{eqnarray}
	 	\Delta_{i,j}(F_{s,t}(\Theta))&=&
	 		\sum_{r=1}^{s-t+1}\delta_{j,\theta_r}F_{s,t}(\theta_1,\cdots,\theta_{r-1},i,\theta_{r+1},\cdots,\theta_{s-t+1})\nonumber\\
	 		& &-\sum_{r=1}^{s-t+1}\delta_{-i,\theta_r}F_{s,t}(\theta_1,\cdots,\theta_{r-1},-j,\theta_{r+1},\cdots,\theta_{s-t+1}),
	 \end{eqnarray}
	 and
	 \begin{eqnarray}
	 	\Delta_{j,j}(F_{s,t}(\Theta))\;:=\;
	 	\sum_{r=1}^{s-t+1}(\delta_{j,\theta_r}-\delta_{-j,\theta_r})F_{s,t}(\Theta)-\sum_{r=t}^s\delta_{j-1,r}F_{s,t}(\Theta).
	 \end{eqnarray}
	 For any $\Theta=(\theta_1,\theta_2,\cdots,\theta_k)\in\mathbb{Z}^k$ and $p,q,r\in\mathbb{Z}$ with $s-r<k$, we define $\Theta_{p,q}(r)\in\mathbb{Z}^{k+p-q}$ as follows
	 \begin{eqnarray}
	 	\Theta_{p,q}(r):=\left\lbrace \begin{array}{ll}
	 		(0,0,\cdots,0),&\text{if}\;s\geqslant r,\\
	 		(\theta_1,\theta_2,\cdots,\theta_k,t),&\text{if}\;s=r-1,\\
	 		(\theta_1,\theta_2,\cdots,\theta_k,s+2,s+3\cdots,r,t),&\text{if}\;s<r-1.
	 	\end{array}\right.
	 \end{eqnarray}
	 
	 The following lemma will be used to prove the main result of this section.
	 \begin{lemma}\label{L5}
	 	For $0\leqslant |i|< j\leqslant n+1,\;0\leqslant t\leqslant s\leqslant n$ and $\Theta\in\mathbb{J}_{s,t}$, we have
	 	\begin{enumerate}[(i)]
	 		\item $\mathcal{D}_{j,i}(F_{s,t}(\Theta))=\Delta_{j,i}(F_{s,t}(\Theta))$.
	 		\item $\mathcal{D}_{j,j}(F_{s,t}(\Theta))=\Delta_{j,j}(F_{s,t}(\Theta))$.
	 		\item  Moreover, if $j\leqslant t$, $$\mathcal{D}_{i,j}(F_{s,t}(\Theta))=\Delta_{i,j}(F_{s,t}(\Theta));$$ if $j>t$ and $i\geqslant -t$, $$\mathcal{D}_{i,j}(F_{s,t}(\Theta))=\Delta_{i,j}(F_{s,t}(\Theta))-F_{s,t}(\Theta)f_{j-1,t}(i)+F_{j-1,t}(\Theta_{j-1,s}(i));$$
	 		if $j>t$ and $i<-t$, 
	 		\begin{equation*}
	 			\begin{aligned}
	 				\mathcal{D}_{i,j}(F_{s,t}(\Theta))=&\;\Delta_{i,j}(F_{s,t}(\Theta))+F_{s,t}(\Theta)(f_{-i-1,t}(-j)-f_{j-1,t}(i))\\&+F_{j-1,t}(\Theta_{j-1,s}(i))-F_{-i-1,t}(\Theta_{-i-1,s}(-j)).
	 			\end{aligned}
	 		\end{equation*}
	 	\end{enumerate}
	 \end{lemma}
	 \begin{proof}
	 	By Lemma \ref{L4}, we have
	 	\begin{eqnarray}
	 		\mathcal{D}_{j,i}(F_{s,t}(\Theta))&=&\sum_{r=1}^{s-t+1}\begin{vmatrix}
	 			x_{t,\theta_1}&\cdots&x_{t,\theta_{r-1}}& \mathcal{D}_{j,i}(x_{t,\theta_r})&x_{t,\theta_{r+1}}&\cdots&x_{t,\theta_{s-t+1}}\\
	 			x_{t+1,\theta_1}\!&\cdots&\!x_{t+1,\theta_{r-1}}\!& \!\mathcal{D}_{j,i}(x_{t+1,\theta_r})\!&\!x_{t+1,\theta_{r+1}}\!&\cdots&\!x_{t+1,\theta_{s-t+1}}\\
	 			\vdots&\ddots&\vdots& \vdots&\vdots&\ddots&\vdots\\
	 			x_{s,\theta_1}&\cdots&x_{s,\theta_{r-1}}& \mathcal{D}_{j,i}(x_{s,\theta_r})&x_{s,\theta_{r+1}}&\cdots&x_{s,\theta_{s-t+1}}
	 		\end{vmatrix}\nonumber\\
	 		&=&\sum_{r=1}^{s-t+1}\begin{vmatrix}
	 			x_{t,\theta_1}&\cdots& \delta_{i,\theta_r}x_{t,j}\!-\!\delta_{-j,\theta_r}x_{t,-i}&\cdots&x_{t,\theta_{s-t+1}}\\
	 			x_{t+1,\theta_1}\!&\cdots&\! \delta_{i,\theta_r}x_{t+1,j}\!-\!\delta_{-j,\theta_r}x_{t+1,-i}\!&\cdots&\!x_{t+1,\theta_{s-t+1}}\\
	 			\vdots&\ddots& \vdots&\ddots&\vdots\\
	 			x_{s,\theta_1}&\cdots& \delta_{i,\theta_r}x_{s,j}\!-\!\delta_{-j,\theta_r}x_{s,-i}&\cdots&x_{s,\theta_{s-t+1}}
	 		\end{vmatrix}\nonumber\\
	 		&=&\sum_{r=1}^{s-t+1}\delta_{i,\theta_r}F_{s,t}(\theta_1,\cdots,\theta_{r-1},j,\theta_{r+1},\cdots,\theta_{s-t+1})\nonumber\\
	 		& &-\sum_{r=1}^{s-t+1}\delta_{-j,\theta_r}F_{s,t}(\theta_1,\cdots,\theta_{r-1},-i,\theta_{r+1},\cdots,\theta_{s-t+1})\nonumber\\
	 		&=&\Delta_{j,i}(F_{s,t}(\Theta)).\label{L5.1}
	 	\end{eqnarray}
	 	and
	 	\begin{eqnarray}
	 		\mathcal{D}_{j,j}(F_{s,t}(\Theta))&=&\sum_{r=1}^{s-t+1}\begin{vmatrix}
	 			x_{t,\theta_1}&\cdots&x_{t,\theta_{r-1}}& \mathcal{D}_{j,j}(x_{t,\theta_r})&x_{t,\theta_{r+1}}&\cdots&x_{t,\theta_{s-t+1}}\\
	 			x_{t+1,\theta_1}\!&\cdots&\!x_{t+1,\theta_{r-1}}\!&\! \mathcal{D}_{j,j}(x_{t+1,\theta_r})\!&\!x_{t+1,\theta_{r+1}}\!&\cdots&\!x_{t+1,\theta_{s-t+1}}\\
	 			\vdots&\ddots&\vdots& \vdots&\vdots&\ddots&\vdots\\
	 			x_{s,\theta_1}&\cdots&x_{s,\theta_{r-1}}& \mathcal{D}_{j,j}(x_{s,\theta_r})&x_{s,\theta_{r+1}}&\cdots&x_{s,\theta_{s-t+1}}
	 		\end{vmatrix}\nonumber\\
	 		&=&\sum_{r=1}^{s-t+1}\begin{vmatrix}
	 			x_{t,\theta_1}&\cdots& (\delta_{j,\theta_r}\!-\!\delta_{-j,\theta_r}\!-\!\delta_{j-1,t})x_{t,\theta_r}&\cdots&x_{t,\theta_{s-t+1}}\\
	 			x_{t+1,\theta_1}\!&\cdots& \!(\delta_{j,\theta_r}\!-\!\delta_{-j,\theta_r}\!-\!\delta_{j-1,t+1})x_{t+1,\theta_r}\!&\cdots&\!x_{t+1,\theta_{s-t+1}}\\
	 			\vdots&\ddots& \vdots&\ddots&\vdots\\
	 			x_{s,\theta_1}&\cdots& (\delta_{j,\theta_r}\!-\!\delta_{-j,\theta_r}\!-\!\delta_{j-1,s})x_{s,\theta_r}&\cdots&x_{s,\theta_{s-t+1}}
	 		\end{vmatrix}\nonumber\\
	 		&=&\sum_{r=1}^{s-t+1}(\delta_{j,\theta_r}-\delta_{-j,\theta_r})F_{s,t}(\Theta)-\sum_{r=t}^s\delta_{j-1,r}F_{s,t}(\Theta)\nonumber\\
	 		&=&\Delta_{j,j}(F_{s,t}(\Theta)).\label{L5.2}
	 	\end{eqnarray}
	 	Hence $(i)$ and $(ii)$ are proved. Moreover, the first equation of $(iii)$ can be proved by the same method as (\ref{L5.1}). If $j>t$ and $i\geqslant -t$, applying Lemma \ref{L4} again, we have
	 	\begin{eqnarray}
	 		& &\mathcal{D}_{i,j}(F_{s,t}(\Theta))\;=\;\begin{vmatrix}
	 			x_{t,\theta_1}&x_{t,\theta_2}&\cdots&x_{t,\theta_{s-t+1}}\\
	 			\vdots&\vdots&\ddots&\vdots\\
	 			x_{r-1,\theta_1}&x_{r-1,\theta_2}&\cdots&x_{r-1,\theta_{s-t+1}}\\
	 			\mathcal{D}_{i,j}(x_{r,\theta_1})&\mathcal{D}_{i,j}(x_{r,\theta_2})&\cdots&\mathcal{D}_{i,j}(x_{r,\theta_{s-t+1}})\\
	 			x_{r+1,\theta_1}&x_{r+1,\theta_2}&\cdots&x_{r+1,\theta_{s-t+1}}\\
	 			\vdots&\vdots&\ddots&\vdots\\
	 			x_{s,\theta_1}&x_{s,\theta_2}&\cdots&x_{s,\theta_{s-t+1}}
	 		\end{vmatrix}\nonumber\\
	 		&=&\Delta_{i,j}(F_{s,t}(\Theta))-\!\sum_{r=t}^{\min\{s,j-1\}}\!x_{r,i}\begin{vmatrix}
	 			x_{t,\theta_1}&x_{t,\theta_2}&\cdots&x_{t,\theta_{s-t+1}}\\
	 			\vdots&\vdots&\ddots&\vdots\\
	 			x_{r-1,\theta_1}&x_{r-1,\theta_2}&\cdots&x_{r-1,\theta_{s-t+1}}\\
	 			f_{j-1,r}(\theta_1)\!&\!f_{j-1,r}(\theta_2)\!&\cdots&\!f_{j-1,r}(\theta_{s-t+1})\\
	 			x_{r+1,\theta_1}&x_{r+1,\theta_2}&\cdots&x_{r+1,\theta_{s-t+1}}\\
	 			\vdots&\vdots&\ddots&\vdots\\
	 			x_{s,\theta_1}&x_{s,\theta_2}&\cdots&x_{s,\theta_{s-t+1}}
	 		\end{vmatrix}.\label{L5.3}
	 	\end{eqnarray}
	 	By Lemma \ref{L1} $(ii)$, $f_{j-1,r}(\theta)=x_{j-1,\theta}-\sum\limits_{k=r}^{j-2}x_{k,\theta}f_{j-1,k+1}(k+1)$. Thus if $j-1\leqslant s$,
	 	\begin{eqnarray}
	 		(\ref{L5.3})&=&\Delta_{i,j}(F_{s,t}(\Theta))-\!\sum_{r=t}^{\min\{s,j-1\}}\!x_{r,i}f_{j-1,r+1}(r+1)\begin{vmatrix}
	 			x_{t,\theta_1}&x_{t,\theta_2}&\cdots&x_{t,\theta_{s-t+1}}\\
	 			\vdots&\vdots&\ddots&\vdots\\
	 			x_{r-1,\theta_1}&x_{r-1,\theta_2}&\cdots&x_{r-1,\theta_{s-t+1}}\\
	 			x_{r,\theta_1}&x_{r,\theta_2}&\cdots&x_{r,\theta_{s-t+1}}\\
	 			x_{r+1,\theta_1}&x_{r+1,\theta_2}&\cdots&x_{r+1,\theta_{s-t+1}}\\
	 			\vdots&\vdots&\ddots&\vdots\\
	 			x_{s,\theta_1}&x_{s,\theta_2}&\cdots&x_{s,\theta_{s-t+1}}
	 \end{vmatrix}\nonumber\\
	 &=&\Delta_{i,j}(F_{s,t}(\Theta))-F_{s,t}(\Theta)f_{j-1,t}(i),\label{L5.4}
	 	\end{eqnarray}
	 	if $j-1>s$,
	 	\begin{eqnarray}
	 		\!\!\!\!\!\!\!(\ref{L5.3})\!\!&\!=\!&\!\Delta_{i,j}(F_{s,t}(\Theta))-F_{s,t}(\Theta)\sum_{r=t}^{s}x_{r,i}f_{j-1,r+1}(r+1)\nonumber\\
	 		\!&\! \!&\!-\sum_{r=1}^sx_{r,i}\!\begin{vmatrix}
	 			x_{t,\theta_1}&\cdots&x_{t,\theta_{s-t+1}}\\
	 			\vdots&\ddots&\vdots\\
	 			x_{r-1,\theta_1}&\cdots&x_{r-1,\theta_{s-t+1}}\\
	 			\sum\limits_{k=s+1}^{j-1}\!\!x_{k,\theta_1}f_{j-1,k+1}(k\!+\!1)\!&\cdots&\!\sum\limits_{k=s+1}^{j-1}\!\!x_{k,\theta_{s-t+1}}f_{j-1,k+1}(k\!+\!1)\\
	 			x_{r+1,\theta_1}&\cdots&x_{r+1,\theta_{s-t+1}}\\
	 			\vdots&\ddots&\vdots\\
	 			x_{s,\theta_1}&\cdots&x_{s,\theta_{s-t+1}}
	 		\end{vmatrix}.\label{L5.5}
	 	\end{eqnarray}
	 	Applying Lemma 2.1 $(ii)$ again, we have
	 	\begin{eqnarray}
	 		&\sum\limits_{r=t}^sx_{r,i}f_{j-1,r+1}(r+1)=f_{j-1,t}(i)-f_{j-1,s+1}(i),\\
	 		&\sum\limits_{k=s+1}^{j-1}x_{k,\theta_l}f_{j-1,k+1}(k+1)=f_{j-1,s+1}(\theta_l),\;l\in\overline{1,s-t+1}.
	 	\end{eqnarray}
	 	Furthermore, 
	 	\begin{eqnarray}
	 		\!\!\!\!\!\!& &F_{s,t}(\Theta)f_{j-1,s+1}(i)-\sum_{r=1}^sx_{r,i}\begin{vmatrix}
	 			x_{t,\theta_1}&\cdots&x_{t,\theta_{s-t+1}}\\
	 			\vdots&\ddots&\vdots\\
	 			x_{r-1,\theta_1}&\cdots&x_{r-1,\theta_{s-t+1}}\\
	 			f_{j-1,s+1}(\theta_1)&\cdots&f_{j-1,s+1}(\theta_{s-t+1})\\
	 			x_{r+1,\theta_1}&\cdots&x_{r+1,\theta_{s-t+1}}\\
	 			\vdots&\ddots&\vdots\\
	 			x_{s,\theta_1}&\cdots&x_{s,\theta_{s-t+1}}
	 		\end{vmatrix}\nonumber\\
	 		\!\!\!\!\!\!&=&\begin{vmatrix}
	 			x_{t,\theta_1}&\cdots&x_{t,\theta_{s-t+1}}&x_{t,i}\\
	 			x_{t+1,\theta_1}&\cdots&x_{t+1,\theta_{s-t+1}}&x_{t+1,i}\\
	 			\vdots&\ddots&\vdots&\vdots&\\
	 			x_{s,\theta_1}&\cdots&x_{s,\theta_{s-t+1}}&x_{s,i}\\
	 			f_{j-1,s+1}(\Theta_1)&\cdots&f_{j-1,s+1}(\Theta_{s-t+1})&f_{j-1,s+1}(i)
	 		\end{vmatrix}.\label{L5.6.1}
	 	\end{eqnarray}
	 	By Laplace expansion,
	 	\begin{eqnarray}
	 		(\ref{L5.6.1})=\begin{vmatrix}
	 			x_{t,\theta_1}&\cdots&x_{t,\theta_{s-t+1}}&0&0&\cdots&0&x_{t,i}\\
	 			x_{t+1,\theta_1}&\cdots&x_{t+1,\theta_{s-t+1}}&0&0&\cdots&0&x_{t+1,i}\\
	 			\vdots&\ddots&\vdots&\vdots&\vdots&\ddots&\vdots&\vdots&\\
	 			x_{s,\theta_1}&\cdots&x_{s,\theta_{s-t+1}}&0&0&\cdots&0&x_{s,i}\\
	 			x_{s+1,\theta_1}&\cdots&x_{s+1,\theta_{s-t+1}}&1&0&\cdots&0&x_{s+1,i}\\
	 			x_{s+2,\theta_1}&\cdots&x_{s+2,\theta_{s-t+1}}&x_{s+2,s+2}&1&\cdots&0&x_{s+2,i}\\
	 			\vdots&\ddots&\vdots&\vdots&\vdots&\ddots&\vdots&\vdots&\\
	 			x_{j-2,\theta_1}&\cdots&x_{j-2,\theta_{s-t+1}}&x_{j-2,s+2}&x_{j-2,s+3}&\cdots&1&x_{j-2,i}\\
	 			x_{j-1,\theta_1}&\cdots&x_{j-1,\theta_{s-t+1}}&x_{j-1,s+2}&x_{j-1,s+3}&\cdots&x_{j-1,j-1}&x_{j-1,i}.
	 		\end{vmatrix},\label{L5.6}
	 	\end{eqnarray}
	 	(\ref{L5.5})-(\ref{L5.6}) imply that
	 	\begin{eqnarray}
	 		(\ref{L5.3})=\Delta_{i,j}(F_{s,t}(\Theta))-F_{s,t}(\Theta)f_{j-1,t}(i)+F_{j-1,t}(\Theta_{j-1,s}(i)).
	 	\end{eqnarray}
	 	Hence the second assertion of $(iii)$ follows. If $j-1>t$ and $i<-t$,
	 	\begin{eqnarray}
	 		\!\!\!\!& &\mathcal{D}_{i,j}(F_{s,t}(\Theta))\;=\;\begin{vmatrix}
	 			x_{t,\theta_1}&x_{t,\theta_2}&\cdots&x_{t,\theta_{s-t+1}}\\
	 			\vdots&\vdots&\ddots&\vdots\\
	 			x_{r-1,\theta_1}&x_{r-1,\theta_2}&\cdots&x_{r-1,\theta_{s-t+1}}\\
	 			\mathcal{D}_{i,j}(x_{r,\theta_1})&\mathcal{D}_{i,j}(x_{r,\theta_2})&\cdots&\mathcal{D}_{i,j}(x_{r,\theta_{s-t+1}})\\
	 			x_{r+1,\theta_1}&x_{r+1,\theta_2}&\cdots&x_{r+1,\theta_{s-t+1}}\\
	 			\vdots&\vdots&\ddots&\vdots\\
	 			x_{s,\theta_1}&x_{s,\theta_2}&\cdots&x_{s,\theta_{s-t+1}}
	 		\end{vmatrix}\nonumber\\
	 		\!\!\!\!&=&\Delta_{i,j}(F_{s,t}(\Theta))-\!\sum_{r=t}^{\min\{s,j-1\}}\!x_{r,i}\begin{vmatrix}
	 			x_{t,\theta_1}&x_{t,\theta_2}&\cdots&x_{t,\theta_{s-t+1}}\\
	 			\vdots&\vdots&\ddots&\vdots\\
	 			x_{r-1,\theta_1}&x_{r-1,\theta_2}&\cdots&x_{r-1,\theta_{s-t+1}}\\
	 			f_{j-1,r}(\theta_1)\!&\!f_{j-1,r}(\theta_2)\!&\cdots&\!f_{j-1,r}(\theta_{s-t+1})\\
	 			x_{r+1,\theta_1}&x_{r+1,\theta_2}&\cdots&x_{r+1,\theta_{s-t+1}}\\
	 			\vdots&\vdots&\ddots&\vdots\\
	 			x_{s,\theta_1}&x_{s,\theta_2}&\cdots&x_{s,\theta_{s-t+1}}
	 		\end{vmatrix}\nonumber\\
	 		\!\!\!\!& &+\!\sum_{r=t}^{\min\{s,-i-1\}}\!x_{r,-j}\begin{vmatrix}
	 			x_{t,\theta_1}&x_{t,\theta_2}&\cdots&x_{t,\theta_{s-t+1}}\\
	 			\vdots&\vdots&\ddots&\vdots\\
	 			x_{r-1,\theta_1}&x_{r-1,\theta_2}&\cdots&x_{r-1,\theta_{s-t+1}}\\
	 			f_{-i-1,r}(\theta_1)\!&\!f_{-i-1,r}(\theta_2)\!&\cdots&\!f_{-i-1,r}(\theta_{s-t+1})\\
	 			x_{r+1,\theta_1}&x_{r+1,\theta_2}&\cdots&x_{r+1,\theta_{s-t+1}}\\
	 			\vdots&\vdots&\ddots&\vdots\\
	 			x_{s,\theta_1}&x_{s,\theta_2}&\cdots&x_{s,\theta_{s-t+1}}
	 		\end{vmatrix}.\label{L5.7}
	 	\end{eqnarray}
	 	By the same calculation as (\ref{L5.3})-(\ref{L5.6}), we have
	 	\begin{eqnarray}
	 		(\ref{L5.7})&=&\Delta_{i,j}(F_{s,t}(\Theta))+F_{s,t}(\Theta)(f_{-i-1,t}(-j)-f_{j-1,t}(i))\nonumber\\
	 		& &+F_{j-1,t}(\Theta_{j-1,s}(i))-F_{-i-1,t}(\Theta_{-i-1,s}(-j)).
	 	\end{eqnarray}
	 	Therefore, the last assertion of $(iii)$ is proved.
	 \end{proof}\pse
 
	 We set
	 \begin{eqnarray}
	 	\mathcal{T}_{0}\!=\!\bigl\{\big(\mathcal{D}_{0,n+1}\!+\!\frac{1}{2}f_{n,0}(0)\big)^{\iota_n}\!\cdots\!\big(\mathcal{D}_{0,2}\!+\!\frac{1}{2}f_{1,0}(0)\big)^{\iota_1}\!\big(\mathcal{D}_{0,1}\!+\!\frac{1}{2}f_{0,0}(0)\big)^{\iota_0}(1)\mid\iota_0,\cdots\!,\iota_n\in\!\{0,1\}\bigr\}\nonumber
	\end{eqnarray}
	and
	\begin{eqnarray}
	 	\mathcal{T}_t=\{F_{s,t}(\Theta)\mid s\in\overline{t,n},\;\Theta\in\mathbb{J}_{s,t}\}\cup\{1\}\;\;\for\;\;t\in\overline{1,n}.
	 \end{eqnarray}
	 
	 \begin{lemma}\label{L6}
	 	 $\mathcal{T}_{0}$ is a basis of $V_{2n+3}\big(\frac{1}{2}\sum_{i=1}^{n+1}\varepsilon_i\big)$ as a vector space over $\mathbb{C}$.
	 \end{lemma}
	 \begin{proof}
	 	In this case, we have
	 	\begin{eqnarray}
	 		& &\big(\mathcal{D}_{0,n+1}+\frac{1}{2}f_{n,0}(0)\big)^{\iota_{n+1}}\cdots\big(\mathcal{D}_{0,2}+\frac{1}{2}f_{1,0}(0)\big)^{\iota_2}\big(\mathcal{D}_{0,1}+\frac{1}{2}f_{0,0}(0)\big)^{\iota_1}(1)\nonumber\\
	 		&=&A_{0,n+1}^{\iota_{n+1}}\cdots A_{0,2}^{\iota_2}A_{0,1}^{\iota_1}(1),
	 	\end{eqnarray}
	 	thus $\mathcal{T}_{0}$ is a subset of $V_{2n+3}(\frac{1}{2}\sum_{i=1}^{n+1}\varepsilon_i)$. By Weyl's formula, we can calculate the dimension of $V_{2n+3}(\frac{1}{2}\sum_{i=1}^{n+1}\varepsilon_i)$,
	 	\begin{eqnarray}
	 		\dim V_{2n+3}(\frac{1}{2}\sum_{i=1}^{n+1}\varepsilon_i)= 2^{n+1}=|\mathcal{T}_{n+1}|.
	 	\end{eqnarray}
	 	
       It lefts to show that $\mathcal{T}_{0}$ is linearly independent. For $k\in\overline{1,n+1}$, let $\mathcal{T}_{0,0}=\{1\}$ and $$\mathcal{T}_{0,k}\!=\!\bigl\{\!\big(\mathcal{D}_{0,k}+\frac{1}{2}f_{k-1,0}(0)\big)^{\iota_{k\!-\!1}}\cdots\big(\mathcal{D}_{0,2}+\frac{1}{2}f_{1,0}(0)\big)^{\iota_1}\!\big(\mathcal{D}_{0,1}+\frac{1}{2}f_{0,0}(0)\big)^{\iota_0}(1)\mid \iota_0,\cdots\!,\iota_k\in\!\{0,1\}\!\bigr\}$$ be the subsets of $\mathcal{T}_0$. In fact, we have $\mathcal{T}_{0,n+1}=\mathcal{T}_0$ and
        \begin{eqnarray}
        	\mathcal{T}_{0,k+1}=\mathcal{T}_{0,k}\cup \big(\mathcal{D}_{0,k+1}+\frac{1}{2}f_{k,0}(0)\big)(\mathcal{T}_{0,k})\;\for\;\text{any}\;k\in\overline{0,n}.
        \end{eqnarray}
       It is clearly that $\mathcal{T}_{0,0}$ is linearly independent, suppose that $\mathcal{T}_{0,k}=\{g_1,g_2,\cdots,g_{2^k}\}$ is linearly independent for $k\in\overline{0,n}$. If there exist $a_1,a_2,\cdots,a_{2^k},b_1,b_2,\cdots,b_{2^k}\in\mathbb{C}$ such that
       \begin{eqnarray}
       	\sum_{i=1}^{2^k}\Big(a_ig_i+b_i\big(\mathcal{D}_{0,k+1}+\frac{1}{2}f_{k,0}(0)\big)(g_i)\Big)=0.
       \end{eqnarray}
       According to (\ref{5.45}),
       \begin{eqnarray}
       	& &\partial_{x_{0,k}}\Big(\sum_{i=1}^{2^k}(a_ig_i+b_i\big(\mathcal{D}_{0,k+1}+\frac{1}{2}f_{k,0}(0)\big)(g_i))\Big)\\&=&\sum_{i=1}^{2^k}\Big(a_i\partial_{x_{0,k}}(g_i)+b_i\partial_{x_{0,k}}\big(\big(\mathcal{D}_{0,k+1}+\frac{1}{2}f_{k,0}(0)\big)(g_i)\big)\Big)\\
       	&=&\frac{1}{2}\sum_{i=1}^{2^k}b_ig_i=0,
       \end{eqnarray}
       which implies $b_1=b_2=\cdots=b_{2^k}=0$, hence $a_1=a_2=\cdots=a_{2^k}=0$ as well. By induction on $k$, $\mathcal{T}_k$ is linearly independent for all $k\in \overline{0,n+1}$. 
	 \end{proof}\pse

	 For $0\leqslant t\leqslant s\leqslant n$, recall that
	 \begin{eqnarray}
	 	\mathbb{J}_{s,t}=\{(\theta_1,\cdots,\theta_{s-t+1})\in\mathbb{Z}^{s-t+1}\mid -(n+1)\leqslant \theta_1<\cdots<\theta_{s-t+1}\leqslant s\},
	 \end{eqnarray}
	 Furthermore, we set
	 \begin{eqnarray}
	 	\mathbb{J}^*_{s,t}=\{(\theta_1,\cdots,\theta_{s-t+1})\in\mathbb{J}_{s,t}\mid \theta_1=-(n+1)\},
	 \end{eqnarray}

	 \begin{definition}\label{D7}
	 	Let $S_1,S_2,\cdots,S_k,S$ be subsets of $\mathscr{A}(n)$, the product of $S_1,S_2,\cdots,S_k$, denoted by $S_1S_2\cdots S_k$ or $\prod_{r=1}^kS_r$, defined to be the subsets of $\mathscr{A}(n)$ as follows:
	 	\begin{eqnarray}
	 		\{f_1f_2\cdots f_k\mid f_r\in S_r\;\for\;r\in\overline{1,k}\}.
	 	\end{eqnarray}
	 	Furthermore, we denote $\prod_{r=1}^kS$ by $S^k$ briefly.
	 \end{definition}

	 Suppose that $\lambda=\sum_{i=1}^{n+1}\mu_i\varepsilon_i$ is a dominant integral weight of $\mathfrak{o}(2n+3)$. Denote
	 \begin{eqnarray}
	 	\mathcal{S}(\lambda)=\mathcal{T}_0^{2\mu_1}\prod_{t=1}^n\mathcal{T}_t^{\mu_{t+1}-\mu_t}
	 \end{eqnarray} 
	 (cf. (\ref{2.3})). Note that 
	 \begin{eqnarray}
	 	&g\in\mathcal{T}_{0,n+1}\setminus \mathcal{T}_{0,n},\;F_{n,t}(\Theta),~(0\leqslant t\leqslant n,~\Theta\in\mathbb{J}_{n,t}\setminus\mathbb{J}_{n,t}^*)\\
	 	&\text{~and~}F_{s,t}(\Theta'),~(0\leqslant t\leqslant s<n,~\Theta'\in\mathbb{J}_{s,t}^*)
	 \end{eqnarray}
	 are homogeneous polynomials with degree $1$ in variables $\{x_{n,j}\mid j\in\overline{-n,n}\}$,
	 \begin{eqnarray}
	 	F_{n,t}(\Theta''),~(0\leqslant t\leqslant n,~\Theta''\in\mathbb{J}_{n,t}^*)
	 \end{eqnarray}
	 are homogeneous with degree $2$ in $\{x_{n,j}\mid j\in\overline{-n,n}\}$, and
	 \begin{eqnarray}
	 	g\in\mathcal{T}_{0,n}\text{~and~}F_{s,t}(\Theta'''),~(0\leqslant t\leqslant s<n,~\Theta'''\in\mathbb{J}_{s,t}\setminus\mathbb{J}_{s,t}^*)
	 \end{eqnarray}
	 are independent of $\{x_{n,j}\mid j\in\overline{-n,n}\}$. Hence the elements of $\mathcal{T}_0,\mathcal{T}_1,\cdots,\mathcal{T}_{n}$ and $\mathcal{S}(\lambda)$ are all homogeneous in $\{x_{n,j}\mid j\in\overline{-n,n}\}$. Let $k\in\mathbb{N}$, set
	 \begin{eqnarray}
	 	S(\lambda)_k=\bigl\{f\in\mathcal{S}(\lambda)\mid \sum_{j=-n}^nx_{n,j}\partial_{n,j}(f)=kf\bigr\}.
	 \end{eqnarray}
	 It is easy to see that $\mathcal{S}(\lambda)_k$ is empty when $k>2\mu_{n+1}$ and 
	 \begin{eqnarray}
	 	\mathcal{S}(\lambda)=\bigcup\limits_{k=0}^{2\mu_{n+1}}\mathcal{S}(\lambda)_k. \label{S}
	 \end{eqnarray}
	 \begin{theorem}\label{T8}
	 	As a vector space, $V_{2n+3}(\lambda)$ is spanned by $S(\lambda)$. Moreover, the homogeneous subspace $(V_{2n+3}(\lambda))_k$ is spanned by $\mathcal{S}(\lambda)_k$.
	 \end{theorem}
	 \begin{proof}
	 	The second conclusion is deduced from the first one and (\ref{S}). Note that $1$ is the unique polynomial in $\mathscr{A}_{(n)}$ annihilated by positive root vectors $\{A_{j,i}\mid 0\leqslant |i|<j\leqslant n+1\}$ (up to a scalar). Next we prove that $\text{Span}\:S(\lambda)$ is an $\mathfrak{o}(2n+3)$-submodule of $\mathscr{A}_{(n)}$. Then by Weyl's theorem, $\text{Span}\:S(\lambda)$ is irreducible since it is finite dimensional and has unique singular vector, hence $V_{2n+3}(\lambda)= \text{Span}\:S(\lambda)$.\psp
	 	
	 	Take any 
	 	\begin{eqnarray}
	 		f=\prod_{t=0}^nf_t\in\mathcal{S}(\lambda),
	 	\end{eqnarray}
	 	with $f_0\in\mathcal{T}_0^{2\mu_1},\;f_t\in\mathcal{T}_t^{\mu_{t+1}-\mu_t}\;\for\;t\in\overline{1,n}$. We write
	 	\begin{eqnarray}
	 		f_0=\prod_{r=1}^{2\mu_1}g_r\;\;\;\text{and}\;\;\;f_t=\prod_{s=t}^n\prod_{\Theta\in\mathbb{J}_{s,t}}F_{s,t}(\Theta)^{\alpha_{s,t}(\Theta)}\;\for\;t\in\overline{1,n},
	 	\end{eqnarray}
	 	where $g_1,g_2,\cdots,g_{2\mu_1}\in\mathcal{T}_0$ and $\alpha_{s,t}(\Theta)\in\mathbb{N}$ with
	 	\begin{eqnarray}
           \sum_{s=t}^n\sum_{\Theta\in\mathbb{J}_{s,t}}\alpha_{s,t}(\Theta)\leqslant \mu_{t+1}-\mu_t\;\for\;t\in\overline{1,n}.
	 	\end{eqnarray}

	 	For $0<-i<j\leqslant n+1$, by Lemma \ref{L5} $(iii)$,
	 	\begin{eqnarray}
	 		A_{i,j}(f)&=&\mathcal{D}_{i,j}(f)+\mu_1(f_{j-1,0}(i)-f_{-i-1,0}(-j))f\\
	 		& &+\sum_{t=1}^{j-1}(\mu_{t+1}-\mu_{t})f_{j-1,t}(i)f-\sum_{t=1}^{-i-1}(\mu_{t+1}-\mu_{t})f_{-i-1,t}(-j)f\\
	 		&=&f\sum_{t=1}^{j-1}\big(\mu_{t+1}-\mu_t-\sum_{s=t}^n\sum_{\Theta\in\mathbb{J}_{s,t}}\alpha_{s,t}(\Theta)\big)f_{j-1,t}(i)\nonumber\\
	 		&&-f\sum_{t=1}^{-i-1}\big(\mu_{t+1}-\mu_t-\sum_{s=t}^n\sum_{\Theta\in\mathbb{J}_{s,t}}\alpha_{s,t}(\Theta)\big)f_{-i-1,t}(-j)\nonumber\\
	 		& &+f\sum_{1\leqslant t\leqslant s\leqslant n\atop \Theta\in\mathbb{J}_{s,t}}\alpha_{s,t}(\Theta)F_{s,t}(\Theta)^{-1}\Delta_{i,j}(F_{s,t}(\Theta))\nonumber\\
	 		& &+f\sum_{t=1}^{j-1}\sum_{s=t}^n\sum_{\Theta\in\mathbb{J}_{s,t}}\alpha_{s,t}(\Theta)F_{s,t}(\Theta)^{-1}F_{j-1,t}(\Theta_{j-1,s}(i))\nonumber\\
	 		& &-f\sum_{t=1}^{-i-1}\sum_{s=t}^n\sum_{\Theta\in\mathbb{J}_{s,t}}\alpha_{s,t}(\Theta)F_{s,t}(\Theta)^{-1}F_{-i-1,t}(\Theta_{-i-1,s}(-j))\nonumber\\
	 		& &+\mu_1\big(f_{j-1,0}(i)-f_{-i-1,0}(-j)\big)f+\mathcal{D}_{i,j}(f_0)\prod_{t=1}^nf_t.
	 	\end{eqnarray}
	 	Note that 
	 	\begin{eqnarray}
	 		\big(\mu_{t+1}-\mu_t-\sum_{s=t}^n\sum_{\Theta\in\mathbb{J}_{s,t}}\alpha_{s,t}(\Theta)\big)f_{j-1,t}(i)f_t\in\text{Span}\:\mathcal{T}_t^{\mu_{t+1}-\mu_t},\;\forall t\in\overline{1,j-1},
	 	\end{eqnarray}
	 	\begin{eqnarray}
	 		\big(\mu_{t+1}-\mu_t-\sum_{s=t}^n\sum_{\Theta\in\mathbb{J}_{s,t}}\alpha_{s,t}(\Theta)\big)f_{-i-1,t}(-j)f_t\in\text{Span}\:\mathcal{T}_t^{\mu_{t+1}-\mu_t},\;\forall t\in\overline{1,-i-1},
	 	\end{eqnarray}
	 	\begin{eqnarray}
	 		\alpha_{s,t}(\Theta)F_{s,t}(\Theta)^{-1}\Delta_{i,j}(F_{s,t}(\Theta))f_t\in\text{Span}\:\mathcal{T}_t^{\mu_{t+1}-\mu_t},\;\forall t\in\overline{1,n},
	 	\end{eqnarray}
	 	\begin{eqnarray} \alpha_{s,t}(\Theta)F_{s,t}(\Theta)^{-1}F_{j-1,t}(\Theta_{j-1,s}(i))f_t\in\text{Span}\:\mathcal{T}_t^{\mu_{t+1}-\mu_t},\;\forall t\in\overline{1,j-1}
	 \end{eqnarray} 
	 and 
	 \begin{eqnarray} \alpha_{s,t}(\Theta)F_{s,t}(\Theta)^{-1}F_{-i-1,t}(\Theta_{-i-1,s}(-j))f_t\in\text{Span}\:\mathcal{T}_t^{\mu_{t+1}-\mu_t},\;\forall t\in\overline{1,-i-1}.
	 	\end{eqnarray}
	 	Moreover,
	 	\begin{eqnarray}
	 		& &\mathcal{D}_{i,j}(f_0)+\mu_1\big(f_{j-1,0}(i)-f_{-i-1,0}(-j)\big)f_0\nonumber\\
	 		&=&\sum_{r=1}^{2\mu_1}g_1\cdots g_{r-1}\mathcal{D}_{i,j}(g_r)g_{r+1}\cdots g_{2\mu_1}+\mu_1\big(f_{j-1,0}(i)-f_{-i-1,0}(-j)\big)g_1g_2\cdots g_{2\mu_1}\nonumber\\
	 		&=&\sum_{r=1}^{2\mu_1}g_1\cdots g_{r-1}\Big(\mathcal{D}_{i,j}(g_r)+\frac{1}{2}\mu_1\big(f_{j-1,0}(i)-f_{-i-1,0}(-j)\big)g_r\Big)g_{r+1}\cdots g_{2\mu_1}
	 	\end{eqnarray}
	 	By Lemma \ref{L6} and applying (\ref{4.46}) with $\lambda=\frac{1}{2}\sum_{i=1}^{n+1}\varepsilon_i$, we have
	 	\begin{eqnarray}
	 		\Big(\mathcal{D}_{i,j}+\frac{1}{2}\mu_1\big(f_{j-1,0}(i)-f_{-i-1,0}(-j)\big)\Big)(g_r)\in V_{2n+3}\big(\frac{1}{2}\sum_{i=1}^{n+1}\varepsilon_i\big)=\text{Span}\:\mathcal{T}_0.
	 	\end{eqnarray}
	 	Hence $\mathcal{D}_{i,j}(f_0)+\mu_1(f_{j-1,0}(i)-f_{-i-1,0}(-j))f_0\in\text{Span}\:\mathcal{T}_0^{2\mu_1}$. Therefore, $A_{i,j}(f)\in\text{Span}\:S(\lambda)$.
	 	
	 	For $0\leqslant i<j\leqslant n+1$, by Lemma \ref{L5} $(iii)$,
	 	\begin{eqnarray}
	 		\!\!A_{i,j}(f)&=&f\sum_{t=i}^{j-1}\big(\mu_{t+1}-\mu_t-\sum_{s=t}^n\sum_{\Theta\in\mathbb{J}_{s,t}}\alpha_{s,t}(\Theta)\big)f_{j-1,t}(i)\nonumber\\
	 		& &+f\!\sum_{1\leqslant t\leqslant s\leqslant n\atop \Theta\in\mathbb{J}_{s,t}}\!\alpha_{s,t}(\Theta)F_{s,t}(\Theta)^{-1}\Delta_{i,j}(F_{s,t}(\Theta))\nonumber\\
	 		& &+f\sum_{t=1}^{j-1}\sum_{s=t}^n\sum_{\Theta\in\mathbb{J}_{s,t}}\alpha_{s,t}(\Theta)F_{s,t}(\Theta)^{-1}F_{j-1,t}(\Theta_{j-1,s}(i))\nonumber\\
	 		& &+\delta_{i,0}\mu_1\big(f_{j-1,0}(i)-f_{-i-1,0}(-j)\big)f+\mathcal{D}_{i,j}(f_0)\prod_{t=1}^nf_t.\label{T8.1}
	 	\end{eqnarray}
	 	
	 	For $0\leqslant |i|<j\leqslant n+1$, by Lemma \ref{L5} $(i)$ and $(ii)$, we have
	 	\begin{eqnarray}
	 		A_{j,i}(f)&=&\mathcal{D}_{j,i}(f)=\mathcal{D}_{j,i}(f_0)\prod_{t=1}^nf_t+f\!\!\sum_{1\leqslant t\leqslant s\leqslant n\atop \Theta\in\mathbb{J}_{s,t}}\!\!\alpha_{s,t}(\Theta)F_{s,t}(\Theta)^{-1}\mathcal{D}_{j,i}(F_{s,t}(\Theta))\nonumber\\
	 			&=&f\!\!\sum_{1\leqslant t\leqslant s\leqslant n\atop \Theta\in\mathbb{J}_{s,t}}\!\!\alpha_{s,t}(\Theta)F_{s,t}(\Theta)^{-1}\Delta_{j,i}(F_{s,t}(\Theta))+\mathcal{D}_{j,i}(f_0)\prod_{t=1}^nf_t\label{T8.2}
	 	\end{eqnarray}
	 	and
	 	\begin{eqnarray}
	 		A_{j,j}(f)&=&\mathcal{D}_{j,j}(f)+\mu_jf\nonumber\\
	 		&=&f\!\!\sum_{1\leqslant t\leqslant s\leqslant n\atop \Theta\in\mathbb{J}_{s,t}}\!\!\alpha_{s,t}(\Theta)F_{s,t}(\Theta)^{-1}\Delta_{j,i}(F_{s,t}(\Theta))+\mathcal{D}_{j,j}(f_0)\prod_{t=1}^nf_t+\mu_jf.\label{T8.3}
	 	\end{eqnarray}
	    Similarly to the formal situation, we can prove (\ref{T8.1})-(\ref{T8.3}) are contained in  $\text{Span}\:\mathcal{S}(\lambda)$. Hence $\text{Span}\:\mathcal{S}(\lambda)$ is an $\mathfrak{o}(2n+3)$-submodule.
	 \end{proof}\pse

	 \begin{remark}\label{R9}
	 	Set
	 	\begin{eqnarray}
	 		\mathcal{T}_0^{\star}=\{F_{s,0}(\Theta)\mid s\in\overline{0,n},\Theta\in\mathbb{J}_{s,0}\}\cup\{1\},
	 	\end{eqnarray}
	 	We can prove that $V_{2n+3}(\sum_{i=1}^{n+1}\varepsilon)=\text{Span}\:\mathcal{T}_0^{\star}=\text{Span}\:\mathcal{T}_0^2$ by the proof of Theorem 3.8. Thus
	 	\begin{eqnarray}
	 		V_{2n+3}(\lambda)=\left\lbrace \begin{array}{ll}
	 			\text{Span}\:(\mathcal{T}_0^{\star})^{\mu_1}\prod_{t=1}^{n}\mathcal{T}_t^{\mu_{t+1}-\mu_t},&\text{if}\;\mu_1\in\mathbb{N};\\
	 			\\
	 			\text{Span}\:\mathcal{T}_0(\mathcal{T}_0^{\star})^{\mu_1-1/2}\prod_{t=1}^{n}\mathcal{T}_t^{\mu_{t+1}-\mu_t},&\text{if}\;\mu_1-\frac{1}{2}\in\mathbb{N}.
	 		\end{array}\right.
	 	\end{eqnarray}
	 \end{remark}

		In the rest of this section, we present two examples of irreducible modules with particular dominant highest weights. Firstly consider the case when $\mu_1=\mu_2=\cdots=\mu_{n-1}=0$ as an example. For simplicity, let
		\begin{eqnarray}
			&x_i=x_{n,i},\;\;y_j=x_{n-1,j},\;\;\for\;\;i\in\overline{-n,n},\;\;j\in\overline{-n+1,n-1},\nonumber\\
			&\text{and }k_{1}=\mu_{n+1}-\mu_n,~k_2=\mu_n\in\mathbb{N}.
		\end{eqnarray}
		Denote
		\begin{eqnarray}
			&D_x=\sum\limits_{s=-n}^nx_{s}\partial_{x_s},\;D_y=\sum\limits_{s=-n+1}^{n-1}y_s\partial_{y_s},\;\eta_x=-x_{-n-1}=\frac{1}{2}x_0^2+\sum\limits_{s=1}^{n}x_sx_{-s},\\
			& \eta_y=-y_{-n}=\frac{1}{2}y_0^2+\sum\limits_{s=1}^{n-1}y_sy_{-s},\;\eta_{xy}=-y_{-n-1}=\sum\limits_{t=-n+1}^{n-1}x_ty_{-t}+x_{-n}-x_n\eta_y.
		\end{eqnarray}

		Theorem \ref{T3} presents the representation formulas of $\mathfrak{o}(2n+3)$:
		For $0\leqslant |i|\leqslant |j|\leqslant n-1$,
\begin{equation}A_{j,i}|_{V_{2n+3}(\lambda)}=x_{j}\partial_{x_{i}}-x_{-i}\partial_{x_{-j}}+y_{j}\partial_{y_{i}}-y_{-i}\partial_{y_{-j}},\end{equation}
\begin{equation}	A_{n,i}|_{V_{2n+3}(\lambda)}=x_{n}\partial_{x_{i}}-x_{-i}\partial_{x_{-n}}+\partial_{y_i},~A_{n+1,\pm n}|_{V_{2n+3}(\lambda)}=\partial_{x_{\pm n}},\end{equation}
\begin{equation}A_{n+1,i}|_{V_{2n+3}(\lambda)}=\partial_{x_i},~A_{n+1,n+1}|_{V_{2n+3}(\lambda)}=-D_x+k_2+k_{1},\end{equation}	
\begin{eqnarray}
			&A_{n,n}|_{V_{2n+3}(\lambda)}=x_{n}\partial_{x_{n}}-x_{-n}\partial_{x_{-n}}-D_y+k_2,\\
			&A_{i,n}|_{V_{2n+3}(\lambda)}=x_i\partial_{x_{n}}-x_{-n}\partial_{x_{-i}}-y_i(D_y-k_2)+\eta_y\partial_{y_{-i}},\\
			&A_{n,n+1}|_{V_{2n+3}(\lambda)}=-x_n(D_x-k_{1})+\eta_x\partial_{x_{-n}}-\sum\limits_{t=-n+1}^{n-1}(x_t-x_ny_t)\partial_{y_t},\\
			&A_{i,n+1}|_{V_{2n+3}(\lambda)}=-x_i(D_x-k_1)+\eta_x\partial_{x_{-i}}\nonumber\\
			&-y_i\sum\limits_{t=-n+1}^{n-1}(x_t-x_ny_t)\partial_{y_t}+k_2(x_i-x_ny_i)+\eta_{xy}\partial_{y_{-i}},\\
			&A_{-n,n+1}|_{V_{2n+3}(\lmd)}=-x_{-n}(D_x-k_{1})+\eta_x\partial_{x_{n}}\nonumber\\
			&+\eta_y\sum\limits_{t=-n+1}^{n-1}(x_t-x_ny_t)\partial_{y_t}+k_2(x_{-n}+x_n\eta_y)-\eta_{xy}(D_y-k_2).
		\end{eqnarray}
		In this case, $\mathcal{S}(\lambda)$ becomes
		\begin{eqnarray}
			\mathcal{S}(\lambda)&=&\big\{\prod\limits_{n-1\leqslant t\leqslant s\leqslant n}\prod\limits_{\Theta\in \mathbb{J}_{s,t}} F_{s,t}(\Theta)^{\alpha_{s,t}(\Theta)} \mid \alpha_{s,t}(\theta)\in\mathbb{N};\nonumber\\
			& &\sum\limits_{\Theta\in\mathbb{J}_{n,n} }\alpha_{n,n}(\Theta)\leqslant k_1\text{ and }\sum\limits_{s=n\!-\!1}^n\!\sum\limits_{\Theta\in\mathbb{J}_{s,n\!-\!1} }\!\!\alpha_{s,n-1}(\Theta)\leqslant k_2\big\}.\label{6.10}
		\end{eqnarray}
		Since $\mu_1=\mu_2=\cdots=\mu_{n-1}=0$. (\ref{6.10}) shows that $\mathcal{S}(\lambda)$ only involves $4n$ variables $\{x_{i},y_j\mid i\in\overline{-n,n},j\in\overline{-n+1,n-1}\}$, which by Theorem \ref{T8} means
		\begin{eqnarray}
			V_{2n+3}(\lambda)\subset\mathbb{C}[x_i,y_j~|~i\in\overline{-n,n},j\in\overline{-n+1,n-1}].
		\end{eqnarray}
		Then the set
		\begin{eqnarray}
			& &\big\{\prod\limits_{i=-n-1}^nx_i^{\alpha_i}\prod_{-n-1\leqslant t<s\leqslant n}(x_sy_t-x_ty_s)^{\gamma_{s,t}}\prod_{j=-n-1}^{n-1}y_j^{\beta_j} \mid \alpha_i,\beta_j,\gamma_{s,t}\in\mathbb{N};\nonumber\\
			& &\sum\limits_{i=-n-1}^n\alpha_i\leqslant k_1;~\sum_{-n-1\leqslant t<s\leqslant n}\gamma_{s,t}+\sum\limits_{j=-n-1}^{n-1}\beta_j\leqslant k_2\big\}
		\end{eqnarray}
		spans $V_{2n+3}(\lambda)$ by (\ref{6.10}).\psp

 Next, we consider the case $\lambda=k(\varepsilon_{i+1}+\varepsilon_{i+2}+\cdots+\varepsilon_{n+1})$ for some $i\in\overline{1,n-1}$. For simplicity, we denote
       \begin{eqnarray}
       	F_j(\Theta)=\begin{vmatrix}
       		x_{i,\theta_1}&x_{i,\theta_2}&\cdots&x_{i,\theta_{j-i+1}}\\
       		x_{i+1,\theta_1}\!&\!x_{i+1,\theta_2}\!&\cdots&\!x_{i+1,\theta_{j-i+1}}\\
            \vdots&\vdots&\ddots&\vdots\\
            x_{j,\theta_1}&x_{j,\theta_2}&\cdots&x_{j,\theta_{j-i+1}}\\		
       	\end{vmatrix}
       \end{eqnarray}
       for $j\in\overline{i,n}$ and $\Theta=(\theta_1,\theta_2,\cdots,\theta_{j-i+1})\in\mathbb{J}_{j,i}$. 
       By Theorem \ref{T8}, $V_{2n+3}(\lambda)$ is spanned by
       \begin{eqnarray}
       	\mathcal{T}_{i}^k=\big\{\prod\limits_{j=i}^n\prod\limits_{\Theta\in\mathbb{J}_{j,i}}F_j(\Theta)^{\alpha_j(\Theta)}~|~\alpha_j(\Theta)\in\mathbb{N};~\sum\limits_{j=i}^n\sum\limits_{\Theta\in\mathbb{J}_{j,i}}\alpha_{j}(\Theta)\leqslant k\big\}.
       \end{eqnarray}
       So $V_{2n+3}(\lambda)$ dose not involve the variables $\{x_{s,t}~|~0\leqslant |t|\leqslant s\leqslant i-1\}$. The representation of $\mathfrak{o}(2n+3)$ is given in (\ref{4.42})-(\ref{4.46}) with $\mu_1=\cdots=\mu_i=0$, $\mu_{i+1}=\cdots=\mu_{n+1}=k$ and all the ingredients containing $\{x_{s,t}~|~0\leqslant |t|\leqslant s\leqslant i-1\}$ are dropped.

		\section{Construction for Even Orthogonal Lie Algebras}
		In this section, we turn to deal with the even case, in which we obtain the parallel results as in the odd case. As well as before, we also put two special cases to show these formulas in more intuitive ways.
		
Let $\msr B_{(1)}=\mathbb{C}[x_1,x_{-1}]$. For any $\mu_1,\mu_2\in\mbb C$, we have the following representation of $\mathfrak{o}(4)$ on $\msr B_{(1)}$:
	\begin{eqnarray}
			&A_{1,1}|_{\msr B_{(1)}}=x_1\partial_{x_1}-x_{-1}\partial_{x_{-1}}+\mu_1,~A_{2,2}|_{\msr B_{(1)}}=-x_1\partial_{x_1}-x_{-1}\partial_{x_{-1}}+\mu_2,\label{3.2}\\
			&A_{2,1}|_{\msr B_{(1)}}=\partial_{x_1},~A_{2,-1}|_{\msr B_{(1)}}=\partial_{x_{-1}},\label{3.3}\\
			&A_{1,2}|_{\msr B_{(1)}}=-x_1^2\partial_{x_1}+(\mu_2-\mu_1)\partial_{x_1},~A_{-1,2}|_{\msr B_{(1)}}=-x_{-1}^2\partial_{x_{-1}}+(\mu_2+\mu_1)\partial_{x_{-1}}.\label{3.4}
	\end{eqnarray}
Then\begin{equation}M_1=U(\mathfrak{o}(4))(1)\end{equation}
forms a highest-weight irreducible $\mathfrak{o}(4)$-module with highest weight $\lambda=\mu_1\varepsilon_1+\mu_2\varepsilon_2$. When $\mu_2\pm\mu_1\in\mathbb{N}$,
\begin{eqnarray}
		M_1=\sum\limits_{i=0}^{\mu_2-\mu_1}\sum\limits_{j=0}^{\mu_2+\mu_1}\mathbb{C}x_1^ix_{-1}^j
\end{eqnarray}
is a finite-dimensional  irreducible $\mathfrak{o}(4)$-module.

		Comparing those formulas with the representation of $\mathfrak{o}(5)$ in (\ref{4.17})-(\ref{4.22}), (\ref{3.2})-(\ref{3.4}) are just obtained from those of $\mathfrak{o}(5)$ by deleting $A_{0,\pm 1},~A_{0,\pm 2}$ and forcing $x_0,~y$ to be $0$. Moreover, according to (\ref{1.5}), (\ref{1.6}) and (\ref{1.10})-(\ref{1.12}), we conclude that the representation formulas of $\mathfrak{o}(2n+2)$ is obtained from $\mathfrak{o}(2n+3)$ by deleting $A_{0,\pm 1},A_{0\pm 2},\cdots,A_{0,\pm n+1}$ and impose $x_{n,0}=x_{n-1,0}=\cdots=x_{0,0}=0$. We revise the notations in (\ref{4.23})-(\ref{4.31}) as follows. Then analogous results for $\mathfrak{o}(2n+2)$ can be derived by the same methods as in the odd case. Hence we only state these parallel results and omit explicit calculation. Let
		 \begin{eqnarray}
			\msr B_{(n)}=\mathbb{C}[x_{i,j}~|~0<|j|\leqslant i\leqslant n]
		\end{eqnarray}
	  be the polynomial algebra in $n(n+1)$ variables. Denote
		\begin{eqnarray}
			x_{i,j}=\left\lbrace
			\begin{array}{ll}
				\text{variables}\; x_{i,j}, &\text{if}\;0< |j|\leqslant i\leqslant n,\\
				1, &\text{if}\;j=i+1,\\
				0, &\text{if}\;j>i+1\;\text{or}\;j=0.
			\end{array}
			\right.  \label{3.6}
		\end{eqnarray}
		For each $i\in \overline{1,n}$, we define the polynomials $x_{i,-j}$ for $i\in\overline{r+1,n+1}$ in $\msr B_{(n)}$ by recurrence as follows,
		\begin{eqnarray}
			&x_{i,-(i+1)}=-\sum\limits_{r=1}^ix_{i,r}x_{i,-r}, \label{3.7} \\
			&x_{i,-j}=-\sum\limits_{|r|=1}^{j-1}x_{i,r}x_{j-1,r}~~\text{for~}j\in\overline{i+2,n+1}. \label{3.8}
		\end{eqnarray}
		Moreover, we take the notation
		\begin{eqnarray}
			\alpha_i=(\alpha_{i,-i},\cdots,\alpha_{i,-1},\alpha_{i,1},\cdots,\alpha_{i,i})\in\mathbb{N}^{2i},~i\in\overline{1,n},
		\end{eqnarray}
		and
		\begin{eqnarray}
			\alpha=(\alpha_1,\alpha_2,\cdots,\alpha_n)\in\mathbb{N}^{n(n+1)}.
		\end{eqnarray}
		Denote by $X_i^{\alpha_1}$ and $X^{\alpha}$ the monomials
		\begin{eqnarray}
			X_i^{\alpha_1}=\prod_{|j|=1}^ix_{i,j}^{\alpha_{i,j}}~\text{and}~X^{\alpha}=\prod_{i=1}^nX_{i}^{\alpha_i}.
		\end{eqnarray}
		
		Finally, for $1\leqslant t\leqslant s\leqslant n$ and $\Theta=(\theta_1,\cdots,\theta_{s-t+1})\in\mathbb{Z}^{s-t+1}$ with $0\neq\theta_k\geqslant -(n+1)$ for every $k\in\overline{1,s-t+1}$, we denote
		\begin{eqnarray}
			F_{s,t}(\Theta)=\begin{vmatrix}
				x_{t,\theta_1}&x_{t,\theta_2}&\cdots&x_{t,\theta_{s-t+1}}\\
				x_{t+1,\theta_1}\!&\!x_{t+1,\theta_2}\!&\cdots&\!x_{t+1,\theta_{s-t+1}}\\
				\cdots&\cdots&\cdots&\cdots\\
				x_{s,\theta_1}&x_{s,\theta_2}&\cdots&x_{s,\theta_{s-t+1}}
			\end{vmatrix},
		\end{eqnarray}
		In particular, we take
		\begin{eqnarray}
			\begin{aligned}
				f_{s,t}(r)=&F_{s,t}(t+1,t+2,\cdots, s, r)\\
				=&\begin{vmatrix}
					1&0&\cdots&0&x_{t,r}\\
					x_{t,t+1}&1&\cdots&0&x_{t+1,r}\\
					\vdots&\vdots&\ddots&\vdots\\
					x_{s,t+1}&x_{n,t+2}&\cdots&x_{s,s}&x_{s,r}
				\end{vmatrix}.
			\end{aligned}\label{3.13}
		\end{eqnarray}
		
		Thus we have the first-order differential operators realization for any highest weight representation of $\mathfrak{o}(2n+2)$ with highest-weight $\lambda=\sum_{i=1}^{n+1}\mu_i\varepsilon_i$.
		\begin{theorem}\label{Th1}
			The representation formulas of $\mathfrak{o}(2n+2)$ on $\msr B_{(n)}$ are as follows.\psp
			
			For $1\leqslant |i|<j\leqslant n+1$,
			\begin{eqnarray}
				&A_{j,i}|_{\msr B_{(n)}}=\sum\limits_{r=j}^n(x_{r,j}\partial_{x_{r,i}}-x_{r,-i}\partial_{x_{r,-j}})+\partial_{x_{j-1,i}},\label{Th1a}
			\end{eqnarray}
			
			For $j\in\overline{2,n+1}$,
			\begin{eqnarray}
				&A_{j,j}|_{\msr B_{(n)}}=\sum\limits_{r=j}^n(x_{r,j}\partial_{x_{r,j}}-x_{r,-j}\partial_{x_{r,-j}})-\sum\limits_{|s|=1}^{j-1}x_{j-1,s}\partial_{x_{j-1,s}}+\mu_j, \\
				&A_{1,1}|_{\msr B_{(n)}}=\sum\limits_{r=1}^n(x_{r,1}\partial_{x_{r,1}}-x_{r,-1}\partial_{x_{r,-1}})+\mu_1.\label{3.15}
			\end{eqnarray}
			\begin{eqnarray}
				A_{1,j}|_{\msr B_{(n)}}&=&\sum\limits_{r=j}^nx_{r,1}\partial_{x_{r,j}}-\sum_{r=1}^{n}x_{r,-j}\partial_{x_{r,-1}}\nonumber\\
				& &-\sum\limits_{r=1}^{j-1}\big[ x_{r,1}\sum\limits_{|s|=1}^rf_{j-1,r}(s)\partial_{x_{r,s}}-(\mu_{r+1}-\mu_r)f_{j-1,r}(1)\big],
			\end{eqnarray}
			\begin{eqnarray}
				A_{-1,j}|_{\msr B_{(n)}}&=&\sum\limits_{r=j}^nx_{r,-i}\partial_{x_{r,j}}-\sum_{r=1}^nx_{r,-j}\partial_{x_{r,i}}+2\mu_1f_{j-1,1}(-1)\nonumber\\
				& &-\sum\limits_{r=1}^{j-1}\big[ x_{r,-1}\sum\limits_{|s|=1}^rf_{j-1,r}(s)\partial_{x_{r,s}}-(\mu_{r+1}-\mu_r)f_{j-1,r}(-1)\big].
			\end{eqnarray}
			
			For $2\leqslant i<j\leqslant n+1$,
			\begin{eqnarray}
					A_{i,j}|_{\msr B_{(n)}}&=&\sum\limits_{r=j}^nx_{r,i}\partial_{x_{r,j}}-\sum_{r=1}^{n}x_{r,-j}\partial_{x_{r,-i}}-\sum\limits_{|s|=1}^{i-1}f_{j-1,i-1}(s)\partial_{x_{i-1,s}}\nonumber\\
					& &-\sum\limits_{r=i}^{j-1}\big[ x_{r,i}\sum\limits_{|s|=1}^rf_{j-1,r}(s)\partial_{x_{r,s}}-(\mu_{r+1}-\mu_r)f_{j-1,r}(i)\big],
			\end{eqnarray}
			and
			\begin{eqnarray}
					A_{-i,j}|_{\msr B_{(n)}}&=&\sum\limits_{r=j}^nx_{r,-i}\partial_{x_{r,j}}-\sum\limits_{r=i}^{n}x_{r,-j}\partial_{x_{r,i}}\nonumber\\
					& &-\sum\limits_{r=1}^{j-1}\big(x_{r,-i}\sum\limits_{|s|=1}^rf_{j-1,r}(s)\partial_{x_{r,s}}\big)+\sum\limits_{r=1}^{i-1}\big(x_{r,-j}\sum\limits_{|s|=1}^rf_{i-1,r}(s)\partial_{x_{r,s}}\big) \nonumber\\
					& &+\sum_{r=2}^{j-1}(\mu_{r+1}-\mu_r)f_{j-1,r}(-i)-\sum\limits_{r=2}^{i-1}(\mu_{r+1}-\mu_r)f_{i-1,r}(-j)\nonumber\\
					& &+\mu_2\big(f_{j-1,1}(-i)-f_{i-1,1}(-j)\big)\nonumber\\
					& &+\mu_1\big(f_{i-1,1}(-1)f_{j-1,1}(1)-f_{i-1,1}(1)f_{j-1,1}(-1)\big)\label{3.17}\label{Th1b}
			\end{eqnarray}
			Moreover, $V_{2n+2}(\lmd)=U(\mathfrak{o}(2n+2))(1)$ is an irreducible highest-weight $\mathfrak{o}(2n+2)$-module with highest weight $\lambda=\sum_{i=1}^{n+1}\mu_i\varepsilon_i$.
		\end{theorem}
		\pse

		Suppose that $\lambda=\sum_{i=1}^{n+1}\mu_i\varepsilon_i$ is a dominant integral weight of $\mathfrak{o}(2n+2)$, where the number $\mu_i$'s satisfy the conditions
		\begin{eqnarray}
			\mu_2+\mu_1,\;\mu_i-\mu_{i-1}\in\mathbb{N},\;\for\;i\in\overline{2,n+1}.
		\end{eqnarray}
		Denote the part of first-order differential operator in the representation formula of $A_{i,j}|_{\mathscr{B}(n)}$ by $\mathcal{D}_{i,j}$ for $i,j\in\overline{1,n+1}$ (cf. (\ref{Th1a})-(\ref{Th1b})). Set
		\begin{eqnarray}
			\mathbb{J}_{s,t}=\{(\theta_1,\cdots,\theta_{s-t+1})\in(\mathbb{Z}\setminus\{0\})^{s-t+1}\mid -(n+1)\leqslant\theta_1<\cdots<\theta_{s-t+1}\leqslant s\}
		\end{eqnarray}
		for $1\leqslant t\leqslant s\leqslant n$ and
		\begin{eqnarray}
			&\mathcal{T}_{1}=\big\{\big(\mathcal{D}_{1,n+1}+f_{n,1}(1)\big)^{\iota_n}\cdots\!\big(\mathcal{D}_{1,2}+f_{1,1}(1)\big)^{\iota_1}(1)\mid\iota_1,\cdots,\iota_n\in\!\{0,1\}\big\},\\
			&\mathcal{T}_{-1}=\big\{\big(\mathcal{D}_{-1,n+1}+f_{n,1}(-1)\big)^{\iota_n}\cdots\!\big(\mathcal{D}_{-1,2}+f_{1,1}(-1)\big)^{\iota_1}(1)\mid\iota_1,\cdots,\iota_n\in\!\{0,1\}\big\},\\
			&\mathcal{T}_t=\{F_{s,t}(\Theta)\mid s\in\overline{t,n},\;\Theta\in\mathbb{J}_{s,t}\}\cup\{1\}\;\;\for\;\;t\in\overline{2,n}.
		\end{eqnarray}
		Similarly to Lemma \ref{L6}, we can prove $\mathcal{T}_{1}$ forms a basis of $V_{2n+2}(-\frac{1}{2}\varepsilon_1+\frac{1}{2}\sum_{k=2}^{n+1}\varepsilon_k))$ and $\mathcal{T}_{-1}$ forms a basis of $V_{2n+2}(\frac{1}{2}\sum_{k=1}^{n+1}\varepsilon_k)$. Denote
		\begin{eqnarray}
			\mathcal{S}(\lambda)=\big\{f_{-1}\prod\limits_{r=1}^nf_r\mid f_{-1}\in\mathcal{T}_{-1}^{\mu_2+\mu_1}\;\text{and}\;f_r\in\mathcal{T}_r^{\mu_{r+1}-\mu_r}\;\for\;t\in\overline{1,n}\big\}.
		\end{eqnarray}
	Moreover, let $S(\lambda)_r$ be the homogeneous subset of $S(\lambda)$ consisting of polynomial with degree $r$ in $\{x_{n,\pm 1},\cdots,x_{n,\pm n}\}$. It is clear that $S(\lambda)_r=\emptyset$ when $r>2\mu_{n+1}$. We have following result.
		\begin{theorem}\label{Th2}
			 The finite-dimensional irreducible $\mathfrak{o}(2n+2)$-module
				\begin{eqnarray}
					V_{2n+2}(\lambda)=\text{Span}\:S(\lambda),
				\end{eqnarray}
			    and its homogeneous subspace
			    \begin{eqnarray}
			    	(V_{2n+2}(\lambda))_r=\text{Span}\:S(\lambda)_r.
			    \end{eqnarray}
		\end{theorem}
		
	 \begin{remark}\label{Re3}
	 	Set \begin{eqnarray}
	 		\mathcal{T}_1^{\ast}=\{F_{s,1}(\Theta)\mid s\in\overline{1,n},\;\Theta\in\mathbb{J}_{s,1}\} \cup \{1\},
	 	\end{eqnarray}
	Then $\text{Span}\:\mathcal{T}_1^{\ast}=\text{Span}\:\mathcal{T}_1\mathcal{T}_{-1}$. Thus if when $\mu_1=0$, we have
	\begin{eqnarray}
		V_{2n+2}(\lambda)=\text{Span}\:\big\{\prod\limits_{r=1}^nf_r\mid f_1\in(\mathcal{T}_1^{\ast})^{\mu_2}\;\text{and}\; f_r\in\mathcal{T}_r^{\mu_{r+1}-\mu_r}\;\for\;r\in\overline{2,n}\big\}.
	\end{eqnarray}
	 \end{remark}
		
		At the end of this section, we consider two examples. Suppose that $\mu_1=\mu_2=\cdots=\mu_{n-1}=0$. For simplicity, let
		\begin{eqnarray}
			&x_i=x_{n,i},\;\;y_j=x_{n-1,j},\;\;\for\; i\in\overline{-n,n}\setminus\{0\},\;\;j\in\overline{-n+1,n-1}\setminus\{0\},\nonumber\\
			&\text{and}\;\;k_1=\mu_{n+1}-\mu_n,\;k_2=\mu_n\in\mathbb{N}.
		\end{eqnarray}
		Denote
		\begin{eqnarray}
			&D_x=\sum\limits_{|s|=1}^nx_{s}\partial_{x_s},\;D_y=\sum\limits_{|s|=1}^{n-1}y_s\partial_{y_s},\;\eta_x=-x_{-n-1}=\sum\limits_{s=1}^{n}x_sx_{-s},\\
			& \eta_y=-y_{-n}=\sum\limits_{s=1}^{n-1}y_sy_{-s},\;\eta_{xy}=-y_{-n-1}=\sum\limits_{|t|=1}^{n-1}x_ty_{-t}+x_{-n}-x_n\eta_y.
		\end{eqnarray}
		
		 Theorem \ref{Th1} presents the representation formulas of $\mathfrak{o}(2n+2)$:
		for $1\leqslant |i|\leqslant j\leqslant n-1$,
		\begin{equation}
			A_{j,i}|_{V_{2n+2}(\lambda)}=x_{j}\partial_{x_{i}}-x_{-i}\partial_{x_{-j}}+y_{j}\partial_{y_{i}}-y_{-i}\partial_{y_{-j}},
		\end{equation}
		\begin{equation}
			A_{n,i}|_{V_{2n+2}(\lambda)}=x_{n}\partial_{x_{i}}-x_{-i}\partial_{x_{-n}}+\partial_{y_i},~A_{n+1,\pm n}|_{V_{2n+3}(\lambda)}=\partial_{x_{\pm n}},
		\end{equation}
		\begin{equation}
			A_{n+1,i}|_{V_{2n+2}(\lambda)}=\partial_{x_i},~A_{n+1,n+1}|_{V_{2n+3}(\lambda)}=-D_x+k_1+k_2,
		\end{equation}
		\begin{equation}
		    A_{n,n}|_{V_{2n+2}(\lambda)}=x_{n}\partial_{x_{n}}-x_{-n}\partial_{x_{-n}}-D_y+k_2,
		\end{equation}
		\begin{equation}A_{i,n}|_{V_{2n+2}(\lambda)}=x_i\partial_{x_{n}}-x_{-n}\partial_{x_{-i}}-y_i(D_y-k_n)+\eta_y\partial_{y_{-i}},
		\end{equation}
		\begin{equation}A_{n,n+1}|_{V_{2n+2}(\lambda)}=-x_n(D_x-k_1)+\eta_x\partial_{x_{-n}}-\sum\limits_{|t|=1}^{n-1}(x_t-x_ny_t)\partial_{y_t},
		\end{equation}
		\begin{eqnarray}
			A_{i,n+1}|_{V_{2n+2}(\lambda)}&=&-x_i(D_x-k_1)+\eta_x\partial_{x_{-i}}\nonumber\\
			&&-y_i\sum\limits_{|t|=1}^{n-1}(x_t-x_ny_t)\partial_{y_t}+k_2(x_i-x_ny_i)+\eta_{xy}\partial_{y_{-i}},
		\end{eqnarray}
		\begin{eqnarray}
			A_{-n,n+1}|_{V_{2n+2}(\lmd)}&=&-x_{-n}(D_x-k_1)+\eta_x\partial_{x_{n}}\nonumber\\
			&&+\eta_y\sum\limits_{|t|=1}^{n-1}(x_t-x_ny_t)\partial_{y_t}+k_2(x_{-n}+x_n\eta_y)-\eta_{xy}(D_y-k_2).
		\end{eqnarray}
		In this case, we have
		\begin{eqnarray}
			\mathcal{S}(\lambda)&=&\big\{\prod\limits_{n-1\leqslant t\leqslant s\leqslant n}\prod\limits_{\Theta\in \mathbb{J}_{s,t}} F_{s,t}(\Theta)^{\alpha_{s,t}(\Theta)} \mid \alpha_{s,t}(\theta)\in\mathbb{N};\nonumber\\
			& &\sum\limits_{\Theta\in\mathbb{J}_{n,n} }\alpha_{n,n}(\Theta)\leqslant k_1\text{ and }\sum\limits_{s=n\!-\!1}^n\!\sum\limits_{\Theta\in\mathbb{J}_{s,n\!-\!1} }\!\!\alpha_{s,n-1}(\Theta)\leqslant k_2\big\},\label{6.20}
		\end{eqnarray}
		since $\mu_1=\mu_2=\cdots=\mu_{n-1}=0$. (\ref{6.20}) shows that $\mathcal{S}(\lambda)$ only involves $4n$ variables $\big\{x_{i},y_j\mid i\in\overline{-n,-1}\cup\overline{1,n},j\in\overline{-n+1,-1}\cup\overline{1,n-1}\big\}$, which by Theorem \ref{Th2} means
		\begin{eqnarray}
			V_{2n+2}(\lambda)\subset\mathbb{C}[x_i,y_j\mid i\in\overline{-n,-1}\cup\overline{1,n},j\in\overline{-n+1,-1}\cup\overline{1,n-1}].
		\end{eqnarray}
		Then the set
		\begin{eqnarray}
			& &\big\{\prod\limits_{-n-1\leqslant i\leqslant n;\atop i\neq0}^nx_i^{\alpha_i}\prod_{-n-1\leqslant t<s\leqslant n;\atop s,t\neq0}(x_sy_t-x_ty_s)^{\gamma_{s,t}}\prod_{|j|=1}^{n-1}y_j^{\beta_j} \mid \alpha_i,\beta_j,\gamma_{s,t}\in\mathbb{N};\nonumber\\
			& &\sum\limits_{-n-1\leqslant i\leqslant n;\atop i\neq0}\alpha_i\leqslant k_1;~\sum_{-n-1\leqslant t<s\leqslant n;\atop s,t\neq 0}\gamma_{s,t}+\sum\limits_{|j|=1}^{n-1}\beta_j\leqslant k_2\big\}
		\end{eqnarray}
		spans $V_{2n+2}(\lambda)$ by (\ref{6.20}). \psp

Next we consider the case $\lambda=k(\varepsilon_{i+1}+\varepsilon_{i+2}+\cdots+\varepsilon_{n+1})$ for some  $i\in\overline{2,n-1}$. Again we use the notation
       \begin{eqnarray}
       	F_j(\Theta)=\begin{vmatrix}
       			x_{i,\theta_1}&x_{i,\theta_2}&\cdots&x_{i,\theta_{j-i+1}}\\
       			x_{i+1,\theta_1}\!&\!x_{i+1,\theta_2}\!&\cdots&\!x_{i+1,\theta_{j-i+1}}\\
       			\vdots&\vdots&\ddots&\vdots\\
       			x_{j,\theta_1}&x_{j,\theta_2}&\cdots&x_{j,\theta_{j-i+1}}\\		
       		\end{vmatrix},\;\;\for\;\; j\in\overline{i,n},
       	\end{eqnarray}
        with $\Theta=(\theta_1,\theta_2,\cdots,\theta_{j-i+1})\in\mathbb{J}_{j,i}$.
       By Theorem \ref{Th2}, $V_{2n+2}(\lambda)$ is spanned by
       \begin{eqnarray}
       	\mathcal{T}_i^k=\big\{\prod\limits_{j=i}^n\prod\limits_{\Theta\in\mathbb{J}_{j,i}}F_j(\Theta)^{\alpha_j(\Theta)}~|~\alpha_j(\Theta)\in\mathbb{N};~\sum\limits_{j=i}^n\sum\limits_{\Theta\in\mathbb{J}_{j,i}}\alpha_{j}(\Theta)\leqslant k\big\}.
       \end{eqnarray}
       Thus $V_{2n+2}(\lambda)$ dose not involve the variables $\{x_{s,t}~|~0< |t|\leqslant s\leqslant i-1\}$. The representation of $\mathfrak{o}(2n+2)$ is given in (\ref{Th1a})-(\ref{Th1b}) with $\mu_1=\cdots=\mu_i=0$, $\mu_{i+1}=\cdots=\mu_{n+1}=k$ and all the ingredients containing $\{x_{s,t}\mid 1\leqslant|t|\leqslant s\leqslant i-1\}$ are dropped.		
		
		\section{ Singular Vectors and Combinatorial Identities}
		
In this section, we find all the $\msr G_n$-singular vectors in $\wht M_r$ and $\ol{ M}_r$ in (1.10), (1.17) and (1.18). They lead to an explicit exhibition of Zhelobenko branching rules. Then we use them and inclusion-exclusion principle to derive  analogues of the Macdonald identities and the identities in (1.24) and (1.25).

\subsection{Odd Case}
		Let $M=M_{n-1}$ in (\ref{2.32}) and let $V_{2n+3}(\lambda)=\ol{M}=M_n$ in (\ref{2.35}) be the finite-dimensional irreducible $\mathfrak{o}(2n+3)$-module constructed in Section 2 with the highest weight $\lambda=\mu_1\varepsilon_1+\mu_2\varepsilon_2+\cdots+\mu_{n+1}\varepsilon_{n+1}$, where the coefficients $\mu_i$'s satisfy the conditions
		\begin{equation}
			2\mu_1,~\mu_2-\mu_1,~\mu_3-\mu_2,~\cdots,\mu_{n+1}-\mu_n\in\mathbb{N}.
		\end{equation}
		We denote
		\begin{equation}
			k_1=2\mu_1,~k_i=\mu_i-\mu_{i-1},~i\in\overline{2,n+1}. \label{5.2}
		\end{equation}
		In this section, we will find a basis for the subspace of $\mathfrak{o}(2n+1)$-singular vectors in $V_{2n+3}(\lambda)$ by the method in \cite{X4, X3}.

		\begin{lemma}\label{Lem1}
			A rational function in $\mathcal{X}=\{x_{i,j}|0\leqslant |j|\leqslant i\leqslant n\}$ annihilated by positive root vectors $\{A_{s,t}\mid 0\leqslant |t|<s\leqslant n\}$ must be a rational functions in $\{f_{n,i}(i)\mid i\in\overline{0,n}\}$ $\cup\{x_{j,-(n+1)}\mid j\in\overline{0,n-1}\}$.
		\end{lemma}
		\begin{proof}
			Set
			\begin{eqnarray}
				y_{i,j}=\left\lbrace
				\begin{array}{ll}
					x_{i,j},&\text{if}\;0\leqslant |j|<i\leqslant n,\\
					f_{n,i}(i),&\text{if}\;j=i,\\
					x_{i-1,-(n+1)},&\text{if}\;-n\leqslant j=-i<0.
				\end{array}
				\right.
			\end{eqnarray}
			Indeed, $x_{n,n}=y_{n,n}$, and for $k<n$,
			\begin{eqnarray}
					y_{k,k}=\begin{vmatrix}
						1&0&\cdots&0&x_{k,k}\\
						x_{k+1,k+1}&1&\cdots&0&x_{k+1,k}\\
						\vdots&\vdots&\ddots&\vdots&\vdots\\
						x_{n-1,k+1}&x_{n-1,k+1}&\cdots&1&x_{n-1,k}\\
						x_{n,k+1}&x_{n,k+1}&\cdots&x_{n,n}&x_{n,k}
					\end{vmatrix}=-x_{k,k}y_{k+1,k+1}+f_{n,k+1}(k).\label{5.28}
			\end{eqnarray}
			Thus
			\begin{equation}
				x_{k,k}=-y_{k,k}y_{k+1,k+1}^{-1}+\sum_{l=k+1}^ny_{l,k}g_l,
			\end{equation}
			where $g_l$'s are rational functions in $\{y_{i,j}~|~k+1\leqslant j\leqslant i\leqslant n\}$.
			
			On the other hand, by Lemma \ref{L2},
			\begin{equation}
				y_{1,-1}=x_{0,-(n+1)}=-x_{0,0}f_{n,1}(0)-f_{n,1}(-1)+\frac{1}{2}x_{0,0}^2f_{n,1}(1).
			\end{equation}
			Hence
			\begin{eqnarray}
					x_{1,-1}&=&y_{2,2}^{-1}(f_{n,2}(-1)-f_{n,1}(-1))\nonumber\\
					&=&y_{2,2}^{-1}(f_{n,2}(-1)+x_{0,0}f_{n,1}(0)+y_{1,-1}-\frac{1}{2}x_{0,0}^2y_{1,1})\nonumber\\
					&=&y_{2,2}^{-1}(y_{1,-1}+h_1),
			\end{eqnarray}
			where $h_1$ is a rational function in $\{y_{i,j}~|~0\leqslant j\leqslant i\leqslant n\}\cup \{y_{2,-1},y_{3,-1},\cdots, y_{n,-1}\}$. For $1<k\leqslant n$,
			\begin{equation}
				y_{k,-k}=-\sum_{l=-k+1}^{k-1}x_{k-1,l}f_{n,k}(-l)-f_{n,k}(-k)+\frac{1}{2}\sum_{r=-k+1}^{k-1}x_{k-1,r}x_{k-1,-r}f_{n,k}(k),
			\end{equation}
			from which we have
			\begin{eqnarray}
					x_{k,-k}&=&y_{k+1,k+1}^{-1}(f_{n,k+1}(-k)-f_{n,k}(-k))\nonumber\\
					&=&y_{k+1,k+1}^{-1}(f_{n,k+1}(-k)+y_{k,-k}+\!\!\sum_{l=-k+1}^{k-1}\!\!x_{k-1,l}f_{n,k}(-l)-\frac{1}{2}y_{k,k}\!\!\sum_{r=-k+1}^{k-1}\!\!x_{k-1,r}x_{k-1,-r})\nonumber\\
					&=&y_{k+1,k+1}^{-1}(y_{k,-k}+h_k),
			\end{eqnarray}
			where $h_k$ is a rational function in $\{y_{l,-k}~|~k+1\leqslant l\leqslant n\}\bigcup(\bigcup\limits_{j=-k+1}^n\{y_{i,j}~|~|j|\leqslant i\leqslant n\})$.
			
			Thus the sets of variables
			\begin{eqnarray}
				\mathcal{X}=\{x_{i,j}~|~0\leqslant |j|\leqslant i\leqslant n\}~~~\text{and}~~~\mathcal{Y}=\{y_{i,j}~|~0\leqslant |j|\leqslant i\leqslant n\}
			\end{eqnarray}
			are functionally equivalent. That is, any rational function in $\mathcal{X}$ can be written as a rational function in $\mathcal{Y}$, and vice versa.
			
			Firstly, it follows by (\ref{4.42}) and (\ref{5.28}) that
			\begin{eqnarray}
				A_{s,t}(y_{k,k})=A_{s,t}(y_{k,-k})=0,\qquad\for~1\leqslant |t|\leqslant s\leqslant n ~\text{and}~0\leqslant k\leqslant n.
			\end{eqnarray}
			
			Secondly, we prove that for any $f$, which is annihilated by positive root vectors $\{A_{s,t}~|~0\leqslant |t|<s\leqslant n\}$, is independent of $\{y_{i,j}~|~0\leqslant |j|<i\leqslant n\}$ by backward induction on $|j|$. By Theorem \ref{T3},
			
			\begin{eqnarray}
					A_{n,\pm(n-1)}(f)\!\!&=&\!\!\sum_{1\leqslant|j|\leqslant i\leqslant n}\!\!\frac{\partial f}{\partial y_{i,j}}A_{n,\pm (n-1)}(y_{i,j})=\!\!\sum_{1\leqslant|j|< i\leqslant n}\frac{\partial f}{\partial y_{i,j}}A_{n,\pm (n-1)}(y_{i,j})\nonumber\\
					\!\!&=&\!\!\sum_{1\leqslant |j|<i\leqslant n}\!\!\frac{\partial f}{\partial y_{i,j}}(x_{n,n}\partial_{x_{n,\pm(n-1)}}-x_{n,\mp(n-1)}\partial_{x_{n,-n}}+\partial_{x_{n-1,\pm(n-1)}})(y_{i,j})\nonumber\\
					\!\!&=&\!\!y_{n,n}\frac{\partial f}{\partial y_{n,\pm (n-1)}}=0.
			\end{eqnarray}
			So $f$ is independent of $y_{n,n-1}$ and $y_{n,-n+1}$. Assume that $f$ is independent of $\{y_{i,j}~|~k\leqslant |j|<i\leqslant n\}$ for some $2\leqslant k\leqslant n-1$. Then
			\begin{eqnarray}
					A_{k,-\pm(k-1)}(f)\!\!&=&\!\!\sum_{1\leqslant |j|<i\leqslant n}\!\!\frac{\partial f}{\partial y_{i,j}}A_{k,\pm(k-1)}(y_{i,j})\nonumber\\
					\!\!&=&\!\!\sum_{1\leqslant |j|<i\leqslant n}\!\!\frac{\partial f}{\partial y_{i,j}}\left[\partial_{x_{k-1,\pm(k-1)}}+\!\!\sum_{l=k+1}^n\!(x_{l,k}\partial_{x_{l,\pm (k-1)}}-x_{l,\mp (k-1)}\partial_{x_{l,-k}}) \right] (y_{i,j})\nonumber\\
					\!\!&=&\!\!x_{k,k}\frac{f}{\partial y_{k,\pm (k-1)}}+\!\!\sum_{l=k+1}^n\!\left[ y_{l,k}\frac{\partial f}{\partial y_{l,\pm (k-1)}}-y_{l,\mp (k-1)}\frac{\partial f}{\partial y_{l,-k}}\right] \nonumber\\
					\!\!&=&\!\!-y_{k,k}y_{k+1,k+1}^{-1}\frac{\partial f}{\partial y_{k,\pm(k-1)}}+\!\!\sum_{l=k+1}^ny_{l,k}\big(g_l\frac{\partial f}{\partial y_{k,k-1}}+\frac{\partial f}{\partial y_{l,k-1}}\big)=0,
			\end{eqnarray}
			where $g_l$'s are rational functions in $\{y_{i,j}~|~k+1\leqslant j\leqslant i\leqslant n\}$. Since $f$ is independent of $\{y_{l,k}~|~k+1\leqslant l\leqslant n\}$, the above equations yields
			\begin{equation}
				y_{k,k}y_{k+1,k+1}^{-1}\frac{\partial f}{\partial y_{k,\pm(k-1)}}=0,
			\end{equation}
			which means $f$ is independent of $y_{k,k-1}$. Thus
			\begin{equation}
				\sum_{l=k+1}^ny_{l,k}\frac{\partial f}{\partial y_{l,k-1}}=0
			\end{equation}
			yields
			\begin{equation}
				y_{l,k}\frac{\partial f}{\partial y_{l,k-1}}=0,~~\text{for}~k+1\leqslant l\leqslant n,
			\end{equation}
			which implies $f$ is independent of $\{y_{i,j}|k-1\leqslant j<i\leqslant n\}$. In other words, $f$ is a rational function in $\{y_{k,\pm k}~|~0\leqslant k\leqslant n\}$.
		\end{proof}\pse

		\begin{lemma}\label{Lem2}
			Every polynomial in $\mathbb{C}[\mathcal{X}]$ annihilated by $\{A_{s,t}~|~0\leqslant |t|<s\leqslant n\}$ must be a polynomial in $\{f_{n,i}(i)~|~0\leqslant i\leqslant n\}\cup\{x_{j,-(n+1)}~|~0\leqslant j\leqslant n-1\}$.
		\end{lemma}
		\begin{proof}
			Suppose that a rational function in $\{f_{n,0}(0),\cdots,f_{n,n}(n)\}\cup\{x_{0,-(n+1)},\cdots,x_{n-1,-(n+1)}\}$ is
			\begin{equation}
				\frac{P(f_{n,0}(0),\cdots,f_{n,n}(n),x_{0,-(n+1)},\cdots,x_{n-1,-(n+1)})}{Q(f_{n,0}(0),\cdots,f_{n,n}(n),x_{0,-(n+1)},\cdots,x_{n-1,-(n+1)})}=H(\mathcal{X})\in\mathbb{C}_n[\mathcal{X}],  \label{5.1}
			\end{equation}
			is a polynomials in $\mathbb{C}_n[\mathcal{X}]$, where $P$ and $Q$ are polynomials in $2n+1$ variables. If we replace every $x_{i,j},~(0\leqslant |j|\leqslant i\leqslant n-1)$ by zero in $(\ref{5.1})$, then the left hand side of $(\ref{5.1})$ becomes
			\begin{equation}
				\frac{P(x_{n,0},\cdots,x_{n,n},x_{n,-1},\cdots,x_{n,-n})}{Q(x_{n,0},\cdots,x_{n,n},x_{n,-1},\cdots,x_{n,-n})},
			\end{equation}
			which is a polynomial in $\{x_{n,j}~|~-n\leqslant j\leqslant n\}$. Therefore, $(\ref{5.1})$ is a polynomial in $\{f_{n,i}(i)~|~0\leqslant i\leqslant n\}\cup\{x_{j,-(n+1)}~|~0\leqslant j\leqslant n-1\}$.
		\end{proof}\pse

		Note that $x_{n,-(n+1)}=-\frac{1}{2}x_{0,0}^2-\sum\limits_{i=1}^nx_{n,i}x_{n,-i}$ is also annihilated by positive root vectors. By Lemma \ref{Lem2}, we suppose that
		\begin{eqnarray}
			\!\!x_{n,-\!(n+1)}=\!\!\!\sum_{0\leqslant i\leqslant n\atop 0\leqslant j\leqslant n\!-\!1}\!\!\! b_{ij}f_{n,i}(i)x_{j,-(n+1)}+\!\!\!\!\!\sum_{0\leqslant i\leqslant j\leqslant n}\!\!\! c_{ij}f_{n,i}(i)f_{n,j}(j)+\!\!\!\!\!\!\sum_{0\leqslant i\leqslant j\leqslant n\!-\!1}\!\!\!\! d_{ij}x_{i,-\!(n+1)}x_{j,-\!(n+1)}.\label{5.42}
		\end{eqnarray}
		Comparing the coefficients of every monomials in above equality, we have
		\begin{eqnarray}
			x_{n,-(n+1)}=\sum_{i=1}^nf_{n,i}(i)x_{i-1,-(n+1)}-\frac{1}{2}f_{n,0}(0)^2.
		\end{eqnarray}
		Thanks to to Lemma \ref{L1}, we have
		\begin{equation}
			f_{n,i}(-(n+1))=x_{n,-(n+1)}-\sum_{k=i}^{n-1}x_{k,-(n+1)}f_{n,k+1}(k+1),
		\end{equation}
		which means $f_{n,i}(-(n+1))$ is annihilated by $\{A_{s,t}~|~0\leqslant |t|<s\leqslant n\}$. According to Theorem \ref{T3}, the polynomials
		\begin{eqnarray}
			\prod_{i=0}^nf_{n,i}(i)^{\gamma_i}\prod_{j=1}^{n}f_{n,j}(-(n+1))^{\gamma'_j}\prod_{k=0}^{n-1}x_{k,-(n+1)}^{\gamma_{-(k+1)}}
		\end{eqnarray}
		with conditions
		\begin{eqnarray}
				& &\gamma_n+\gamma'_n\in\overline{0,k_{n+1}};\label{5.14.1}\\
				& &\gamma_i+\gamma_{-(i+1)}+\gamma'_i\in\overline{0,k_{i+1}},~\text{for }i\in\overline{1,n-1};\label{5.14.2}\\
				& &\gamma_0+2\gamma_{-1}\in\overline{0,k_1}. \label{5.14.3}
		\end{eqnarray}
		form a set $\mathcal{H}_o(\lambda)$ of $\mathfrak{o}(2n+1)$-singular vectors in $V_{2n+3}(\lambda)$.
		
		Now we define a lexicographic degree as follows,
		\begin{eqnarray}
			(n,\cdot)<\cdots<(1,\cdot)<(0,\cdot), \label{5.15}
		\end{eqnarray}
		and
		\begin{equation}
			(j,j)<(j,-j)<\cdots<(j,1)<(j,-1)<(j,0),~~j=0,1,\cdots,n.\label{5.16}
		\end{equation}
		
		Given two monomials  $X^{\alpha},X^{\alpha'}\in\msr A_{(n)}$,  we write
		\begin{eqnarray}
			\text{lexdeg}(X^{\alpha})=(\alpha_{0,0},\alpha_{1,1},\alpha_{1,-1},\alpha_{1,0},\cdots,\alpha_{n,0}).
		\end{eqnarray}
		and say that $\text{lexdeg} X^{\alpha}>\text{lexdeg} X^{\alpha'}$
		if (i).~$|\alpha|>|\alpha'|$ or (ii).~$|\alpha|=|\alpha'|$ and there exists $(p,q)$ in (\ref{5.15}) and (\ref{5.16}) such that $\alpha_{i,j}=\alpha'_{i,j}$ for all indexes $(i,j)<(p,q)$ and $\alpha_{p,q}>\alpha'_{p,q}$. This gives a total order. Define a map $\mathscr{L}:\msr A_{(n)}\longrightarrow\msr A_{(n)}$ such that for every polynomial $f\in\msr A_{(n)}$, $\mathscr{L}(f)$ is the monomial with minimal lexicographic degree in the above sense appearing in $f$. Thus
		\begin{eqnarray}
			\mathscr{L}\big(\prod_{i=0}^nf_{n,i}(i)^{\gamma_i}\prod_{j=1}^{n}f_{n,j}(-(n+1))^{\gamma'_j}\prod_{k=0}^{n-1}x_{k,-(n+1)}^{\gamma_{-(k+1)}}\big)=x_{n,0}^{\gamma_0}\prod_{s=1}^nx_{n,s}^{\gamma_s+\gamma'_s}\prod_{t=1}^nx_{n,-t}^{\gamma_{-t}+\gamma'_t}. \label{5.51}
		\end{eqnarray}
		
		Define a relation in $\msr A_{(n)}$ as follows:
		\begin{equation}
			\forall f_1,~f_2\in \msr A_{(n)},~~f_1\sim f_2~~~~\text{if  }\mathscr{L}(f_1)=\mathscr{L}(f_2).
		\end{equation}
		This is an equivalence relation. Let
		\begin{equation}
			\mathcal{B}_o(\lambda)\text{ be a set of representatives of the elements in $\mathcal{H}_o(\lambda)/\sim$}
		\end{equation}
		and let
		\begin{equation}
			\mathcal{M}_o(\lambda)=\mathscr{L}(\mathcal{B}_o(\lambda)=\big\{\prod\limits_{i=-n}^nx_{n,i}^{\beta_i}~|~\exists f\in \mathcal{H}_o(\lambda)~\text{such that}~\mathscr{L}(f)=\prod\limits_{i=-n}^nx_{n,i}^{\beta_i} \big\}.
		\end{equation}
		Note that $\mathcal{B}_o(\lambda)$ is a linearly independent subset of $\mathcal{H}_o(\lambda)$ and $|\mathcal{M}_o(\lambda)|=|\mathcal{B}_o(\lambda)|$.
		
		For $f\in \mathcal{H}_o(\lambda)$, we assume that $\mathscr{L}(f)=\prod\limits_{i=-n}^nx_{n,i}^{\beta_i}$. It is easy to calculate that the weight of $f$ is
		\begin{equation}
			\lambda(f)=\sum_{i=1}^{n}(\mu_{i}+\beta_i-\beta_{-i})\varepsilon_i\label{5.52}
		\end{equation}
	    Denote
		\begin{eqnarray}
			&\kappa_i=\mu_i+\beta_i,~\nu_i=\kappa_i-\beta_{-i},~i=1,2,\cdots,n,\\
			&\kappa_0=\min\{\mu_1,\nu_1\}-\beta_0.
		\end{eqnarray}
		 Thus we obtain two series $\vec{\kappa}=(\kappa_0,\kappa_1,\cdots,\kappa_n)$ and $\vec{\nu}=(\nu_1,\nu_2,\cdots,\nu_n)$. According to condition (\ref{5.14.1})-(\ref{5.14.3}), we have
		\begin{equation}
			\begin{aligned}
				&-\mu_1\leqslant \kappa_0\leqslant\mu_1\leqslant \kappa_1\leqslant \mu_2\leqslant \cdots\leqslant \mu_n\leqslant \kappa_n\leqslant \mu_{n+1},\\
				&-\nu_1\leqslant \kappa_0\leqslant\nu_1\leqslant \kappa_1\leqslant \nu_2\leqslant \cdots\leqslant \nu_n\leqslant \kappa_n.
			\end{aligned} \label{5.54}
		\end{equation}
		and
		\begin{equation}
			\begin{aligned}
				 &\{\kappa_i,~\nu_i\mid~i\geqslant1\} \text{ being nonnegative integers }\\  &\text{or nonnegative half integers as}\; {\mu_i}'s.
			\end{aligned}\label{5.55}
		\end{equation}
	   This coincides with the Zhelobenko branching rules. Conversely, we are given two series $\vec{\kappa},~\vec{\nu}$ satisfying (\ref{5.54}) and (\ref{5.55}). For $1\leqslant i\leqslant n$, let
		\begin{equation}
			(\gamma_i,\gamma_{-i},\gamma'_i)=\left\lbrace \begin{array}{ll}
				(0,\mu_i-\nu_i,\kappa_i-\mu_i),&\text{if~}\mu_i>\nu_i,\\
				(\nu_i-\mu_i,0,\kappa_i-\nu_i),&\text{if~}\mu_i\leqslant\nu_i.
			\end{array}
			\right.\label{5.56}
		\end{equation}
		and let $\gamma_0=\min\{\mu_1,\nu_1\}-\kappa_0$. Then
		\begin{eqnarray}
			\gamma_n+\gamma'_n=\kappa_n-\mu_n\leqslant k_{n+1},	\label{5.57}
		\end{eqnarray}
		\begin{eqnarray}
				& &\gamma_i+\gamma_{-(i+1)}+\gamma'_i=\kappa_i-\mu_i+\max\{0,\mu_{i+1}-\nu_{i+1}\}\nonumber\\
				&=&\max\{\kappa_i-\mu_i,\mu_{i+1}-\mu_i-(\nu_{i+1}-\kappa_i)\}\leqslant k_{i+1},~~1\leqslant i\leqslant n-1,\label{5.58}
		\end{eqnarray}
		\begin{eqnarray}
			\gamma_0+2\gamma_{-1}=\min\{\mu_1,\nu_1\}-\kappa_0+2\max\{0,\mu_1-\nu_1\}\leqslant 2\mu_1=k_1.\label{5.59}
		\end{eqnarray}
		That is, $\{\gamma_i,\gamma'_j~|~-n\leqslant i\leqslant n,~1\leqslant j\leqslant n\}$ satisfying (\ref{5.14.1})-(\ref{5.14.3}). Therefore, we construct an one-to-one correspondence between $\mathcal{B}_o(\lambda)$ and the set of pairs of series $\{\vec{\kappa},~\vec{\nu}\}$ which satisfying (\ref{5.54}) and (\ref{5.55}). Therefore, we have
		\begin{theorem}\label{The3}
			$\mathcal{B}_o(\lambda)$ is a basis of the subspace of $\mathfrak{o}(2n+1)$-singular vectors in $V_{2n+3}(\lambda)$.
		\end{theorem}
		
		Now let $\mathcal{H}_o(\lambda)_r$ (respectively, $\mathcal{B}_o(\lambda)_r,~\mathcal{M}_o(\lambda)_r$) be the homogeneous subset of $\mathcal{H}_o(\lambda)$ (respectively, $\mathcal{B}_o(\lambda),~\mathcal{M}_o(\lambda)$) with degree $r$ in $\{x_{n,i}~|~i\in\overline{-n,n}\}$. By the definition, $\mathscr{L}|_{\mathcal{B}_o(\lambda)_r}:\mathcal{B}_o(\lambda)_r\longrightarrow\mathcal{M}_o(\lambda)_r$ is bijective and $\mathcal{B}_o(\lambda)_r\neq \emptyset$ only if $0\leqslant r\leqslant 2\mu_{n+1}$.
		
		For $\beta=(\beta_{-n},\cdots,\beta_n)\in\mathbb{N}^{2n+1}$, we set
		\begin{equation}
			\begin{aligned}
				\lambda(\beta)=&\sum_{i=1}^{n}(\mu_{i}+\beta_i-\beta_{-i})\varepsilon_i.\label{5.61}
			\end{aligned}
		\end{equation}
		Let $|\beta|\!=\!\!\sum\limits_{i=-n}^n\!\!\beta_i$ be the length of $\beta$ and let $X_n^{\beta}\!=\!\!\prod\limits_{i=-n}^n\!\!x_{n,i}^{\beta_i}$ be the monomial in $\mathbb{C}[x_{n,-n},\!\cdots\!,\!x_{n,n}]$ with degree $\beta$. Let $\widehat{\Omega}_r$ be the set of $(2n+1)$-tuples $\beta\in\mathbb{N}^{2n+1}$ with following conditions
		\begin{eqnarray}
			&0\leqslant \beta_i\leqslant k_{i+1},~i\in\overline{0,n-1},~0\leqslant \beta_{-1}\leqslant \beta_1-\frac{1}{2}\beta_0+\frac{1}{2}k_1, \label{5.62}\\
			&0\leqslant \beta_{-i}\leqslant \beta_i-\beta_{i-1}+k_i,~i\in\overline{2,n},~|\beta|=r\label{5.63}
		\end{eqnarray}
		and
		\begin{eqnarray}
			\Omega_r=\{\beta\in\widehat{\Omega}_r~|~0\leqslant \beta_n\leqslant k_{n+1}\}.
		\end{eqnarray}
		In the following, we count a singular vector up to its nonzero constant multiple. Denote by $\mathscr{A}_r$ be the subspace of homogeneous polynomials with degree $r$ in $\mathbb{C}[x_{n,j}~|~-n\leqslant j\leqslant n]$. By (\ref{5.51})-(\ref{5.59}), we have
		\begin{theorem}\label{The4}
			For $r\in\mathbb{N}$, the $\mathfrak{o}(2n+1)$-singular vectors in $\widehat{M}_r$ are
			\begin{eqnarray}
				\{f=\prod_{i=0}^nf_{n,i}(i)^{\gamma_i}\prod_{j=1}^{n}f_{n,j}(-(n+1))^{\gamma'_j}\prod_{k=0}^{n-1}x_{k,-(n+1)}^{\gamma_{-(k+1)}}~|~\mathscr{L}(f)=X_n^{\beta}~\text{with }\beta\in\widehat{\Omega}_r\}, \label{5.65}
			\end{eqnarray}
			which generate all the irreducible $\mathfrak{o}(2n+1)$-submodules in
			\begin{eqnarray}
				\widehat{M}_r=\mathscr{A}_r\otimes V_{2n+1}(\sum_{i=1}^n\mu_i\varepsilon_i). \label{5.66}
			\end{eqnarray}
			It implies the identity
			\begin{eqnarray}
				\sum_{\beta\in\widehat{\Omega}_r}d_{2n+1}(\lambda(\beta))=\binom{2n+r}{2n}d_{2n+1}(\sum_{i=1}^n\mu_i\varepsilon_i). \label{5.67}
			\end{eqnarray}
			
			For $r \in\overline{0,2\mu_{n+1}}$, the $\mathfrak{o}(2n+1)$-singular vectors in $V_{2n+3}(\lambda)_r=\overline{M}_r$ are
			\begin{eqnarray}
				\mathcal{B}_o(\lambda)_r=\{f\in\mathcal{B}_o(\lambda)~|~\mathscr{L}(f)=X_n^{\beta}~\text{with }\beta\in\Omega_r\}, \label{5.68}
			\end{eqnarray}
		\end{theorem}
		\pse

		\begin{lemma}\label{Lem5}
			$\dim(V_{2n+3}(\lambda))_{r_1}=\dim(V_{2n+3}(\lambda))_{r_2}$ for nonnegative integers $r_1,~r_2$ with $r_1+r_2=2\mu_{n+1}=k_1+2k_2+2k_3+\cdots+2k_{n+1}$.
		\end{lemma}
		\begin{proof}
			Take any $f=\prod_{i=0}^nf_{n,i}(i)^{\gamma_i}\prod_{j=1}^{n}f_{n,j}(-(n+1))^{\gamma'_j}\prod_{k=0}^{n-1}x_{k,-(n+1)}^{\gamma_{-(k+1)}}\in\mathcal{B}_o(\lambda)_{r_1}$, then
			\begin{eqnarray}
					& &\gamma_n+\gamma'_n\in\overline{0,k_{n+1}},\\
					& &\gamma_i+\gamma_{-(i+1)}+\gamma'_i\in\overline{0,k_{i+1}},~~\text{for }i\in\overline{1,n-1}\\
					& &\gamma_0+2\gamma_{-1}\in\overline{0,k_1},\\
					& &\sum_{i=-n}^n\gamma_i+2\sum_{j=1}^n\gamma_j'=r_1.
			\end{eqnarray}
		   Let
		   \begin{eqnarray}
		   		& &\gamma_n''=k_{n+1}-\gamma_n-\gamma_n',\\
		   		& &\gamma_i''=k_{i+1}-\gamma_i-\gamma_{-(i+1)}-\gamma_i',~~i\in\overline{1,n-1},\\
		   		& &\gamma_0''=k_1-\gamma_0-2\gamma_{-1}.
		   \end{eqnarray}
	   Then
	   \begin{eqnarray}
	   		& &\sum_{|i|=1}^n\gamma_i+\gamma_0''+2\sum_{j=1}^n\gamma_j''\nonumber\\
	   		&=&\sum_{|i|=1}^n\gamma_i+k_1-\gamma_0-2\gamma_{-1}+2\sum_{j=1}^{n-1}(k_{j+1}-\gamma_j-\gamma_{-(j+1)}-\gamma_j')+k_{n+1}-\gamma_n-\gamma_{n}'\nonumber\\
	   		&=&k_1+2\sum_{j=1}^nk_{j+1}-\sum_{i=-n}^n\gamma_i-2\sum_{j=1}^n\gamma_j'=r_2.
	   \end{eqnarray}
       Moreover,
       \begin{eqnarray}
       		& &\gamma_n+\gamma_n''=k_{n+1}-\gamma_n'\in\overline{0,k_{n+1}},\\
       		& &\gamma_i+\gamma_{-(i+1)}+\gamma_i''=k_{i+1}-\gamma_i'\in\overline{0,k_i+1},~~\text{for }i\in\overline{1,n-1},\\
       		& &\gamma_0''+2\gamma_{-1}=k_1-\gamma_0\in\overline{0,k_1}.
       \end{eqnarray}
   Therefore,
   \begin{equation}
   	\bar{f}=f_{n,0}(0)^{\gamma''_0}\prod_{i=1}^nf_{n,i}(i)^{\gamma_i}\prod_{j=1}^{n}f_{n,j}(-(n+1))^{\gamma''_j}\prod_{k=0}^{n-1}x_{k,-(n+1)}^{\gamma_{-(k+1)}}
   \end{equation}
   is an element in $\mathcal{H}_o(\lambda)_{r_2}$. Thus the following correspondences are one-to-one:
   \begin{equation}
   	 \begin{aligned}
   	 	\mathcal{B}_o(\lambda)_{r_1}&\longrightarrow&\mathcal{M}_o(\lambda)_{r_1}&\longrightarrow&\mathcal{M}_o(\lambda)_{r_2}&\longrightarrow&\mathcal{B}_o(\lambda)_{r_2}\\
   	 f&\longmapsto& \mathscr{L}(f)&\longmapsto&\mathscr{L}(\bar{f})&\longmapsto&(\mathscr{L}|_{\mathcal{B}_o(\lambda)})^{-1}(\mathscr{L}(\bar{f}))
   	 \end{aligned}
   \end{equation}
Note that the weight of $f$ and $\bar{f}$ are
\begin{equation}
	\lambda(f)=\lambda(\bar{f})=\sum_{i=1}^{n}(\mu_i+\gamma_i-\gamma_{-1})\varepsilon_i.
\end{equation}
As a conclusion,
\begin{equation}
		\sum_{f\in\mathcal{B}_o(\lambda)_{r_1}}\!\!\!\dim U(\mathfrak{o}(2n+1))(f)=\!\!\!\sum_{f\in\mathcal{B}_o(\lambda)_{r_2}}\!\!\!\dim U(\mathfrak{o}(2n+1))(f),
\end{equation}
which leads to the lemma.
		\end{proof}\pse

		Now we are going to use Theorem \ref{The4} and Lemma \ref{Lem5} to calculate $\dim V_{2n+3}(\lambda)_r$. First we assume that $0\leqslant r\leqslant \mu_{n+1}$. Let
		\begin{eqnarray}
			r_s=r-\sum_{i=s}^{n+1}(k_i+1),~~s\in\overline{1,n+1}
		\end{eqnarray}
		For $s>2$, let $\widehat{\Upsilon}_s$ be the set of $(2n+1)$-tuples $\beta\in\mathbb{N}^{2n+1}$ with following conditions
		\begin{eqnarray}
			&|\beta|=r_s,~0\leqslant \beta_{s-2}\leqslant k_{s-1}+k_s+1,\\
			&0\leqslant\beta_i\leqslant k_{i+1},~0\leqslant \beta_j\leqslant k_{j+2},~i\in\overline{1,s-3},~j\in\overline{s-1,n-1},\\
			&0\leqslant \beta_{-(s-1)}\leqslant \beta_{s-1}-\beta_{s-2}+k_{s-1}+k_s+1,0\leqslant \beta_{-1}\leqslant \beta_1-\frac{1}{2}\beta_0+\frac{1}{2}k_1\\
			&0\leqslant \beta_{-p}\leqslant \beta_p-\beta_{p-1}+k_p,~0\leqslant \beta_{-q}\leqslant \beta_q-\beta_{q-1}+k_{q+1},~p\in\overline{1,s-2},~q\in\overline{2,n}.
		\end{eqnarray}
		Set
		\begin{eqnarray}
			\Upsilon_s=\{\beta\in\widehat{\Upsilon}_s~|~0\leqslant \beta_{s-2}\leqslant k_{s-1}\}.
		\end{eqnarray}
		
		For $s=2$, let $\widehat{\Upsilon}_2$ be the set of $(2n+1)$-tuples $\beta\in\mathbb{N}^{2n+1}$ with following conditions
		\begin{eqnarray}
			&|\beta|=r_2,~0\leqslant \beta_j\leqslant k_{j+2},~\overline{1,n-1},~0\leqslant \beta_0\leqslant k_1+2k_2+2,\\
			&0\leqslant \beta_{-q}\leqslant \beta_q-\beta_{q-1}+k_{q+1},~q\in\overline{2,n},~0\leqslant \beta_{-1}\leqslant \beta_1-\frac{1}{2}\beta_0+\frac{1}{2}k_1+k_2+1
		\end{eqnarray}
		and
		\begin{eqnarray}
			\Upsilon_2=\{\beta\in\widehat{\Upsilon}_2~|~0\leqslant \beta_{0}\leqslant k_1\}.
		\end{eqnarray}
		We have $\widehat{\Upsilon}_2=\Upsilon_2$ since $r_2\leqslant k_1$ by our assumption. Furthermore, it is derived from the definition that
		\begin{eqnarray}
			\widehat{\Upsilon}_s\setminus\Upsilon_s=(0,\cdots,0, \stackrel{s-2}{k_{s-1}+1},0,\cdots,0)+\Upsilon_{s-1},~s\in\overline{3,n+1}. \label{5.78}
		\end{eqnarray}
		Replacing $\sum_{i=1}^n\mu_i\varepsilon_i$ by $\sum\limits_{i=1}^n\mu_i\varepsilon_i+\sum\limits_{i=s-1}^n(k_{i+1}+1)\varepsilon_i$ in (\ref{5.66}) and (\ref{5.67}), we have
		\begin{eqnarray}
			&\dim(\mathscr{A}_{r_s}\otimes V_{2n+1}(\sum\limits_{i=1}^n\mu_i\varepsilon_i+\sum\limits_{i=s-1}^n(k_{i+1}+1)\varepsilon_i)) \nonumber\\
			&=\sum\limits_{\beta\in\widehat{\Upsilon}_s}d_{2n+1}(\sum\limits_{i=1}^n\mu_i\varepsilon_i+\sum\limits_{i=s-1}^n(k_{i+1}+1)\varepsilon_i+\sum\limits_{i=1}^n(\beta_i-\beta_{-i})\varepsilon_i). \label{5.79}
		\end{eqnarray}
		Hence, by (\ref{5.78}) and (\ref{5.79}),
		\begin{eqnarray}
				& &\dim(\mathscr{A}_{r_s}\otimes V_{2n+1}(\sum\limits_{i=1}^n\mu_i\varepsilon_i+\sum\limits_{i=s-1}^n(k_{i+1}+1)\varepsilon_i))\nonumber\\
				&=&(\sum\limits_{\beta\in\Upsilon_s}+\sum\limits_{\beta\in\widehat{\Upsilon}_s\setminus\Upsilon_s})d_{2n+1}(\sum\limits_{i=1}^n\mu_i\varepsilon_i+\sum\limits_{i=s-1}^n(k_{i+1}+1)\varepsilon_i+\sum\limits_{i=1}^n(\beta_i-\beta_{-i})\varepsilon_i)\nonumber\\
				&=&\!\!\sum_{\beta\in\widehat{\Upsilon}_{s+1}\setminus \Upsilon_{s+1}}\!\!d_{2n+1}(\sum\limits_{i=1}^n\mu_i\varepsilon_i+\sum\limits_{i=s}^n(k_{i+1}+1)\varepsilon_i+\sum\limits_{i=1}^n(\beta_i-\beta_{-i})\varepsilon_i)\nonumber\\
				& &+\!\!\sum\limits_{\beta\in\widehat{\Upsilon}_s\setminus\Upsilon_s}\!\!d_{2n+1}(\sum\limits_{i=1}^n\mu_i\varepsilon_i+\sum\limits_{i=s-1}^n(k_{i+1}+1)\varepsilon_i+\sum\limits_{i=1}^n(\beta_i-\beta_{-i})\varepsilon_i).\label{5.80}
		\end{eqnarray}
		Since $\widehat{\Upsilon}_2\setminus\Upsilon_2=\emptyset$, multiplying (\ref{5.80}) by $(-1)^{n-s}$ and summing over $s\in\overline{2,n}$, we get
		\begin{eqnarray}
				& &\sum_{s=2}^n(-1)^{n-s}\dim(\mathscr{A}_{r_s}\otimes V_{2n+1}(\sum\limits_{i=1}^n\mu_i\varepsilon_i+\sum\limits_{i=s-1}^n(k_{i+1}+1)\varepsilon_i))\nonumber\\
				&=&\!\!\sum_{\beta\in\widehat{\Upsilon}_{n+1}\setminus \Upsilon_{n+1}}\!\!d_{2n+1}(\sum\limits_{i=1}^n\mu_i\varepsilon_i+(k_{n+1}+1)\varepsilon_n+\sum\limits_{i=1}^n(\beta_i-\beta_{-i})\varepsilon_i). \label{5.81}
		\end{eqnarray}
		Notice $\widehat{\Omega}_r\setminus\Omega_r=\Upsilon_{n+1}$, we conclude by (\ref{5.79}) and (\ref{5.81}) that
		\begin{eqnarray}
				& &\dim(\widehat{M}_r/(V_{2n+3})_r)\nonumber\\
				&=&\sum\limits_{\beta\in\Upsilon_{n+1}}d_{2n+1}(\sum_{i=1}^n\mu_i\varepsilon_i+(k_{n+1}+1)\varepsilon_n+\sum\limits_{i=1}^n(\beta_i-\beta_{-i})\varepsilon_i)\nonumber\\
				&=&(\sum\limits_{\beta\in\widehat{\Upsilon}_{n+1}}\!-\!\!\sum\limits_{\beta\in\widehat{\Upsilon}_{n+1}\setminus\Upsilon_{n+1}}\!\!)d_{2n+1}(\sum_{i=1}^n\mu_i\varepsilon_i+(k_{n+1}+1)\varepsilon_n+\sum\limits_{i=1}^n(\beta_i-\beta_{-i})\varepsilon_i)\nonumber\\
				&=&\sum_{s=2}^{n+1}(-1)^{n-s+1}\dim(\mathscr{A}_{r_s}\otimes V_{2n+1}(\sum\limits_{i=1}^n\mu_i\varepsilon_i+\sum\limits_{i=s-1}^n(k_{i+1}+1)\varepsilon_i))\nonumber\\
				&=&\sum_{s=2}^{n+1}(-1)^{n-s+1}\binom{2n+r_s}{2n}d_{2n+1}(\sum\limits_{i=1}^n\mu_i\varepsilon_i+\sum\limits_{i=s-1}^n(k_{i+1}+1)\varepsilon_i).
		\end{eqnarray}
		Here we make the convention that for $p\in\mathbb{Z}$ and $q\in\mathbb{N}$, \begin{eqnarray}
			\binom{p}{q}=\left\lbrace\begin{array}{ll}
				\frac{p!}{q!(p-q)!},&\text{if }p\geqslant q,\\
				0,&\text{if }p<q.
			\end{array}\right.
		\end{eqnarray}
		Hence we gain the formula of the dimension for homogeneous subspace $(V_{2n+3}(\lambda))_r$ when $0\leqslant r\leqslant \mu_{n+1}$.
		\begin{eqnarray}
				& &\dim(V_{2n+3}(\lambda))_r=\dim\widehat{M}_r-\dim(\widehat{M}_r/(V_{2n+3}(\lmd))_r)\nonumber\\
				&=&\binom{2n+r}{2n}d_{2n+1}(\sum_{i=1}^n\mu_i\varepsilon_i)+\sum_{s=2}^{n+1}(-1)^{n-s}\binom{2n+r_s}{2n}d_{2n+1}(\sum\limits_{i=1}^n\mu_i\varepsilon_i+\sum\limits_{i=s-1}^n(k_{i+1}+1)\varepsilon_i)\nonumber\\
				&=&\sum_{s=2}^{n+2}(-1)^{n-s}\binom{2n+r-\sum\limits_{j=s}^{n+1}(k_j+1)}{2n}d_{2n+1}(\sum\limits_{i=1}^n\mu_i\varepsilon_i+\sum\limits_{i=s-1}^n(k_{i+1}+1)\varepsilon_i).\label{5.84}
		\end{eqnarray}
		
		Replacing $k_1$ by $2\mu_1$ and $k_i$ by $\mu_i-\mu_{i-1}$ for $i\in\overline{2,n+1}$ in (\ref{5.84}) (cf. (\ref{5.2})),  the dimension for $(V_{2n+3}(\lambda))_r$ is equal to
		\begin{eqnarray}
		\sum\limits_{s=0}^n(-1)^{n-s}\binom{n+s+r+\mu_{s+1}-\mu_{n+1}}{2n}d_{2n+1}(\sum\limits_{i=1}^s\mu_i\varepsilon_i+\sum\limits_{i=s+1}^n(\mu_{i+1}+1)\varepsilon_i)
		\end{eqnarray}

		Lemma \ref{Lem1} shows that for $\mu_{n+1}<r\leqslant 2\mu_{n+1}$,
		\begin{eqnarray}
				& &\dim(V_{2n+3}(\lambda))_r=\dim(V_{2n+3}(\lambda))_{2\mu_{n+1}-r}\nonumber\\
				&=&\!\sum\limits_{s=0}^n(-1)^{n-s}\binom{n+s\!-\!r\!+\!\mu_{s+1}\!+\!\mu_{n+1}}{2n}d_{2n+1}(\sum\limits_{i=1}^s\mu_i
\varepsilon_i+\!\!\sum\limits_{i=s+1}^n\!(\mu_{i+1}\!+\!1)\varepsilon_i). \label{5.85}
		\end{eqnarray}
	    Therefore, if $2\mu_{n+1}$ is odd, (\ref{5.84}) and (\ref{5.85}) yield
	    \begin{eqnarray}
	    		& &\dim V_{2n+3}(\lambda)=2\sum\limits_{r=0}^{\mu_{n+1}-1/2}\dim((V_{2n+3}))_r\nonumber\\
	    		&=&2\sum\limits_{s=0}^n(-1)^{n-s}\sum\limits_{r=0}^{\mu_{n+1}-1/2}\binom{n+s+r+\mu_{s+1}-\mu_{n+1}}{2n}d_{2n+1}(\sum\limits_{i=1}^s\mu_i\varepsilon_i+\!\!\sum\limits_{i=s+1}^n(\mu_{i+1}+1)\varepsilon_i)\nonumber\\
	    		&=&2\sum_{s=2}^{n+2}(-1)^{n-s}\binom{n+s+\frac{1}{2}+\mu_{s+1}}{2n+1}d_{2n+1}(\sum\limits_{i=1}^s\mu_i\varepsilon_i+\!\!\sum\limits_{i=s+1}^n(\mu_{i+1}+1)\varepsilon_i), \label{5.86}
	    \end{eqnarray}
	   and if $2\mu_{n+1}$ is even, we have
	     \begin{eqnarray}
	     	& &\dim V_{2n+3}(\lambda)=\sum\limits_{r=0}^{\mu_{n+1}}\dim((V_{2n+3}))_r+\!\sum\limits_{r=0}^{\mu_{n+1}-1}\dim((V_{2n+3}))_r\nonumber\\
	     	&=&\!\!\sum\limits_{s=0}^n(-1)^{n-s}\sum\limits_{r=0}^{\mu_{n+1}}\binom{n\!+\!s\!+\!r\!+\!\mu_{s+1}\!-\!\mu_{n+1}}{2n}d_{2n+1}(\sum\limits_{i=1}^s\mu_i\varepsilon_i+\!\!\sum\limits_{i=s+1}^n(\mu_{i+1}+1)\varepsilon_i)\nonumber\\
	     	& &+\sum\limits_{s=0}^n(-1)^{n-s}\!\sum\limits_{r=0}^{\mu_{n+1}-1}\!\!\binom{n\!+\!s\!+\!r\!+\!\mu_{s+1}\!-\!\mu_{n+1}}{2n}d_{2n+1}(\sum\limits_{i=1}^s\mu_i\varepsilon_i+\!\!\sum\limits_{i=s+1}^n\!(\mu_{i+1}+1)\varepsilon_i)\nonumber
	     	 \\&=&\!\!\sum\limits_{s=0}^n(-1)^{n-s}\binom{n\!+\!s\!+\!1\!+\!u_{s+1}}{2n\!+\!1}d_{2n+1}
(\sum\limits_{i=1}^s\mu_i\varepsilon_i+\!\!\sum\limits_{i=s+1}^n(\mu_{i+1}+1)\varepsilon_i)\nonumber
	      \\	& &\!\!+\sum\limits_{s=0}^n(-1)^{n-s}\binom{n\!+\!s\!+\!\mu_{s+1}}{2n+1}d_{2n+1}(\sum\limits_{i=1}^s\mu_i\varepsilon_i+\!\!\sum\limits_{i=s+1}^n(\mu_{i+1}+1)\varepsilon_i)\nonumber\\
          	  &=&\!\!\!\! \!\! \sum\limits_{s=0}^n(-1)^{n-s}\binom{n\!+\!s\!+\!\mu_{s+1}}{2n}\frac{2s\!+\!2\mu_{s+1}\!+\!1}{2n\!+\!1}d_{2n+1}(\sum\limits_{i=1}^s\mu_i\varepsilon_i+\!\!\sum\limits_{i=s+1}^n(\mu_{i+1}+1)\varepsilon_i).\label{5.87}
         \end{eqnarray}
Thus we have the identities
\begin{eqnarray}d_{2n+3}(\sum_{i=1}^{n+1}\mu_i\ves_i) &=&2\sum_{s=0}^{n}(-1)^{n-s}\binom{n+s+\frac{1}{2}+\mu_{s+1}}{2n+1}\nonumber\\ & & \times d_{2n+1}(\sum\limits_{i=1}^s\mu_i\varepsilon_i+\!\!\sum\limits_{i=s+1}^n(\mu_{i+1}+1)\varepsilon_i)  \end{eqnarray}
if $2\mu_{n+1}$ is odd, and
\begin{eqnarray}d_{2n+3}(\sum_{i=1}^{n+1}\mu_i\ves_i) &=&\sum\limits_{s=0}^n(-1)^{n-s}\binom{n\!+\!s\!+\!\mu_{s+1}}{2n}\frac{2s\!+\!2\mu_{s+1}\!+\!1}{2n\!+\!1}\nonumber\\ & & \times d_{2n+1}(\sum\limits_{i=1}^s\mu_i\varepsilon_i+\!\!\sum\limits_{i=s+1}^n(\mu_{i+1}+1)\varepsilon_i)
         \end{eqnarray}
when $2\mu_{n+1}$ is even.\psp

	      As an application, we consider the case of $Steinberg$ $module$, whose highest weight is $\lambda=k(\lambda_1+\lambda_2+\cdots+\lambda_{n+1})$, where $\lambda_i$'s are the fundamental
	      weights. In this case $\mu_i=(i-1/2)k,~i\in\overline{1,n+1}$. Recall that we take the set of positive roots of $\mathfrak{o}(2n+1)$
	      \begin{eqnarray}
	      	\Phi_{B_n}^{+}=\{\varepsilon_j\pm\varepsilon_i,~\varepsilon_r~|~1\leqslant i< j\leqslant n,~1\leqslant r\leqslant n\}.
	      \end{eqnarray}
          Let
          \begin{eqnarray}
          	\rho=\frac{1}{2}\sum_{\alpha\in\Phi_{B_n}^{+}}\alpha=\frac{1}{2}\sum_{i=1}^n(2i-1)\varepsilon_i.
          \end{eqnarray}
          By the dimension formula in \cite{H}, its dimension
	      \begin{eqnarray}
	      	d_{2n+3}(\lambda)=(k+1)^{(n+1)^2}.
	      \end{eqnarray}

          Denote the weight in the last factor in (4.85) and (4.86) as
          \begin{eqnarray}
          	\lambda_s=\sum\limits_{i=1}^s\mu_i\varepsilon_i+\sum\limits_{i=s+1}^n(\mu_{i+1}+1)\varepsilon_i.
          \end{eqnarray}
          We have
          \begin{eqnarray}
          	\prod_{1\leqslant i<j\leqslant s}\frac{(\lambda_s+\rho,\varepsilon_j\pm\varepsilon_i)}{(\rho,\varepsilon_j\pm\varepsilon_i)}=\prod_{1\leqslant i<j\leqslant s} (k+1)^2=(k+1)^{s^2-s},
          \end{eqnarray}
          \begin{eqnarray}
          	\prod_{s+1\leqslant i<j\leqslant n}\frac{(\lambda_s+\rho,\varepsilon_j\pm\varepsilon_i)}{(\rho,\varepsilon_j\pm\varepsilon_i)}&=&\prod_{s+1\leqslant i<j\leqslant n} \frac{(k+1)^2}{j+i-1}\\&=&\frac{(2n)!(2s+1)!(k+1)^{(n-s)(n-s-1)}}{(n+s+1)!(n+s)!},
          \end{eqnarray}
          \begin{eqnarray}
          	\prod_{1\leqslant i\leqslant s\atop s+1\leqslant j\leqslant n}\frac{(\lambda_s+\rho,\varepsilon_j\pm\varepsilon_i)}{(\rho,\varepsilon_j\pm\varepsilon_i)}&=&\prod_{1\leqslant i\leqslant s\atop s+1\leqslant j\leqslant n}\frac{(j-i+1)(j+i)(k+1)^2}{(j-i)(j+i-1)}\\
          	&=&\frac{(n+s)!(k+1)^{2s(n-s)}}{(n-s)!(2s)!},
          \end{eqnarray}
          \begin{eqnarray}
          	\prod_{r=1}^s\frac{(\lambda_s+\rho,\varepsilon_r)}{(\rho,\varepsilon_r)}=\prod_{r=1}^s(k+1)=(k+1)^s,
          \end{eqnarray}
          \begin{eqnarray}
          \prod_{r=s+1}^n\frac{(\lambda_s+\rho,\varepsilon_r)}{(\rho,\varepsilon_r)}=\prod_{r=s+1}^n\frac{(2r+1)(k+1)}{2r-1}=\frac{(2n+1)(k+1)^{n-s}}{2s+1}.
          \end{eqnarray}
          Therefore,
	      \begin{eqnarray}
	      	d_{2n+1}(\sum\limits_{i=1}^s\mu_i\varepsilon_i+\!\!\sum\limits_{i=s+1}^n(\mu_{i+1}+1)\varepsilon_i)=\prod_{\alpha\in\Phi_{B_n}^+}\frac{(\lambda_s+\rho,\alpha)}{(\rho,\alpha)}=\binom{2n+1}{n-s}(k+1)^{n^2}. \label{5.88}
	      \end{eqnarray}
	      Thus, (4.85) becomes
	      \begin{equation}
	      	2\sum_{s=0}^{n}(-1)^{n-s}\binom{n+(s+\frac{1}{2})(k+1)}{2n+1}\binom{2n+1}{n-s}=(k+1)^{2n+1} ,~\text{if $k$ is odd.}\label{5.89}
	      \end{equation}
	      and (4.86) becomes
	      \begin{eqnarray}
	      	\sum\limits_{s=0}^{n}(-1)^{n-s}\binom{n-\frac{1}{2}+(s+\frac{1}{2})(k+1)}{2n}\binom{2n+1}{n-s}\frac{2s+1}{2n+1}=(k+1)^{2n},~\text{if $k$ is even.} \label{5.90}
	      \end{eqnarray}

       \subsection{Even Case}

       Now we turn to the even case by parallel methods. Let $\lambda=\mu_1\varepsilon_1+\mu_2\varepsilon_2+\cdots+\mu_{n+1}\varepsilon_{n+1}$ be a dominant weight of $\mathfrak{o}(2n+2)$, where the coefficients $\mu_i$'s satisfy the conditions
       \begin{equation}
       	\mu_2+\mu_1,~\mu_2-\mu_1,~\mu_3-\mu_2,~\cdots,\mu_{n+1}-\mu_n\in\mathbb{N}.\label{5.91}
       \end{equation}
       We assume that
       \begin{equation}
       	k_1=\mu_2+\mu_1,~k_i=\mu_i-\mu_{i-1},~i\in\overline{2,n+1}. \label{5.92}
       \end{equation}
       Since the representation of $\mathfrak{o}(2n+2)$ is derived by $\mathfrak{o}(2n+3)$ by dropping the variable $x_{i,0},~i\in\overline{0,n}$ and adjusting the coefficients of $\lambda$. We obtain the set of $\mathfrak{o}(2n)$-singular vectors in $V_{2n+2}(\lambda)$ from (\ref{5.42})-(\ref{5.14.3}). In such variation,  $x_{0,-(n+1)}$ becomes $f_{n,1}(-1)$, and $f_{n,0}(0)$ becomes $0$, and the others keep the notations unchanged in the sense of (\ref{3.6})-(\ref{3.13}). Namely, the polynomials
       \begin{equation}
       	f_{n,1}(-1)^{\gamma_{-1}}\prod_{i=1}^nf_{n,i}(i)^{\gamma_i}\prod_{j=2}^{n}f_{n,j}(-(n+1))^{\gamma'_j}\prod_{k=1}^{n-1}x_{k,-(n+1)}^{\gamma_{-(k+1)}},\label{5.94}
       \end{equation}
       with conditions
       \begin{eqnarray}
       		& &\gamma_n+\gamma'_n\in\overline{0,k_{n+1}};\label{5.95.1}\\
       		& &\gamma_i+\gamma_{-(i+1)}+\gamma'_i\in\overline{0,k_{i+1}},~i\in\overline{2,n-1};\label{5.95.2}\\
       		& &\gamma_1+\gamma_{-2}\in\overline{0,k_2};\label{5.95.3}\\
       		& &\gamma_{-1}+\gamma_{-2}\in\overline{0,k_1};  \label{5.95.4}
       \end{eqnarray}
       constitute a set of $\mathfrak{o}(2n)$-singular vectors in $V_{2n+2}(\lambda)$, which we write as $\mathcal{H}_e(\lambda)$. We again define a total order in $\mathbb{C}_n^0[X]$. Given two monomials  $X^{\alpha},X^{\alpha'}\in\mathbb{C}_n^0[X]$, we write
        \begin{eqnarray}
        	\text{lexdeg}(X^{\alpha})=(\alpha_{0,0},\alpha_{1,1},\alpha_{1,-1},\alpha_{1,0},\cdots,\alpha_{n,0}).\label{5.96}
        \end{eqnarray}
        and say that
        \begin{eqnarray}
        	\text{lexdeg} X^{\alpha}>\text{lexdeg} X^{\alpha'} \label{5.97}
        \end{eqnarray}
        if $|\alpha|>|\alpha'|$, or $|\alpha|=|\alpha'|$ and $\alpha>\alpha'$ according to the order
       \begin{eqnarray}
       	\alpha_{1,1},\alpha_{1,-1},\cdots,\alpha_{n,1},\alpha_{n,-1},\cdots,\alpha_{n,n},\alpha_{n,-n}.\label{5.98}
       \end{eqnarray}

       For $f\in\mathbb{C}_n^0[X]$, let $\mathscr{L}(f)$ be the monomial with minimal lexicographic degree  in $f$ in the above sense. Thus
       \begin{eqnarray}& &       	\mathscr{L}\big(f_{n,1}(-1)^{\gamma_{-1}}\prod_{i=1}^nf_{n,i}(i)^{\gamma_i}\prod_{j=2}^{n}f_{n,j}
       (-(n+1))^{\gamma'_j}\prod_{k=1}^{n-1}x_{k,-(n+1)}^{\gamma_{-(k+1)}}\big)\nonumber\\&=&\prod_{s=1}^nx_{n,s}^{\gamma_s+\gamma'_s}\prod_{t=1}^nx_{n,-t}^{\gamma_{-t}+\gamma'_t}. \label{5.99}
       \end{eqnarray}
       Define a relation in $\mathbb{C}^0_n[X]$ as follows:
       \begin{equation}
       	\forall f_1,~f_2\in \mathbb{C}^0_n[X],~~f_1\sim f_2~~~~\text{if }\mathscr{L}(f_1)=\mathscr{L}(f_2).
       \end{equation}
       This is an equivalence relation. Let
       \begin{equation}
       	\mathcal{B}_e(\lambda)\text{ be a set of representatives of the elements in $\mathcal{H}_e(\lambda)/\sim$}
       \end{equation}
       and let
       \begin{equation}
       	\mathcal{M}_e(\lambda)=\mathscr{L}(\mathcal{B}_e(\lambda))=\big\{\prod\limits_{|i|=1}^nx_{n,i}^{\beta_i}\mid\exists f\in \mathcal{H}_e(\lambda)~\text{such that}~\mathscr{L}(f)=\prod\limits_{|i|=1}^nx_{n,i}^{\beta_i} \big\}
       \end{equation}
       Note that $\mathcal{B}_e(\lambda)$ is a linearly independent subset of $\mathcal{H}_e(\lambda)$ and $|\mathcal{M}_e(\lambda)|=|\mathcal{B}_e(\lambda)|$.

       Take $f\in \mathcal{H}_e(\lambda)$ and write $\mathscr{L}(f)=\prod\limits_{|i|=1}^nx_{n,i}^{\beta_i}$. We calculate by (\ref{3.15}) that the weight of $f$ is
       \begin{equation}
       	\lambda(f)=\sum_{i=1}^{n}(\mu_{i}+\beta_i-\beta_{-i})\varepsilon_i.\label{5.100}
       \end{equation}
       Denote
       \begin{eqnarray}
       	&\kappa_i=\mu_i+\beta_i,~\nu_i=\kappa_i-\beta_{-i},~i=2,\cdots,n,\label{5.100.1}\\
       	&\nu_1=\mu_1+\beta_1-\beta_{-1},\kappa_1=\max\{|\mu_1|,|\nu_1|\}+\min\{\beta_{-1},\beta_1\}.\label{5.100.2}
       \end{eqnarray}
       Thus we obtain two series $\vec{\kappa}=(\kappa_1,\cdots,\kappa_n)$ and $\vec{\nu}=(\nu_1,\nu_2,\cdots,\nu_n)$. Condition (\ref{5.95.1})-(\ref{5.95.4}) leads to the inequality
       \begin{equation}
       	\begin{aligned}
       		&|\mu_1|\leqslant \kappa_1\leqslant \mu_2\leqslant\kappa_2 \cdots\leqslant \mu_n\leqslant \kappa_n\leqslant \mu_{n+1},\\
       		&|\nu_1|\leqslant \kappa_1\leqslant \nu_2\leqslant\kappa_2 \cdots\leqslant \nu_n\leqslant \kappa_n.
       	\end{aligned} \label{5.102}
       \end{equation}
       Moreover, (\ref{5.100.1}) and (\ref{5.100.2}) imply
       \begin{equation}
       	\begin{aligned}
       		&\text{all the } \kappa_i \text{ and } \nu_i \text{ being nonnegative integers }\\  &\text{or nonnegative half integers as } {\mu_i}'s. \label{5.103}
       	\end{aligned}
       \end{equation}
It is just the Zhelobenko branching rule for $\mathfrak{o}(2n+2)\downarrow \mathfrak{o}(2n)$. Conversely, we are given two series $\vec{\kappa}$ and $\vec{\nu}$ satisfying (\ref{5.102}) and (\ref{5.103}).
       For $2\leqslant i\leqslant n$, let
       \begin{equation}
       	(\gamma_i,\gamma_{-i},\gamma'_i)=\left\lbrace \begin{array}{ll}
       		(0,\mu_i-\nu_i,\kappa_i-\mu_i),&\text{if~}\mu_i>\nu_i,\\
       		(\nu_i-\mu_i,0,\kappa_i-\nu_i),&\text{if~}\mu_i\leqslant\nu_i
       	\end{array}\label{5.104}
       	\right.
       \end{equation}
       and
        \begin{equation}
       	(\gamma_1,\gamma_{-1})=\kappa_1-\max\{|\mu_1|,|\nu_1|\}+\left\lbrace \begin{array}{ll}
       		(0,\mu_1-\nu_1),&\text{if~}\mu_1>\nu_1,\\
       		(\nu_1-\mu_1,0),&\text{if~}\mu_1\leqslant\nu_1.
       	\end{array}\label{5.105}
       	\right.
       \end{equation}
       Then we have
       \begin{eqnarray}
       	\gamma_n+\gamma'_n=\kappa_n-\mu_n\leqslant k_{n+1},	\label{5.106}
       \end{eqnarray}
       \begin{eqnarray}
       		& &\gamma_i+\gamma_{-(i+1)}+\gamma'_i=\kappa_i-\mu_i+\max\{0,\mu_{i+1}-\nu_{i+1}\}\nonumber\\
       		&=&\max\{\kappa_i-\mu_i,\mu_{i+1}-\mu_i-(\nu_{i+1}-\kappa_i)\}\leqslant k_{i+1},~~2\leqslant i\leqslant n-1,\label{5.107}
       \end{eqnarray}
        When $\mu_1>\nu_1$ and $\mu_2>\nu_2$,
        \begin{eqnarray}
        		\gamma_1+\gamma_{-2}&=&\kappa_1-\max\{|\mu_1|,|\nu_1|\}+\mu_2-\nu_2\nonumber\\
        		&\leqslant & -\max\{|\mu_1|,|\nu_1|\}+\mu_2\leqslant -|\mu_1|+\mu_2\leqslant k_2.\label{5.108}\\
         		\gamma_{-1}+\gamma_{-2}&=&\mu_1-\nu_1+\kappa_1-\max\{|\mu_1|,|\nu_1|\}+\mu_2-\nu_2\nonumber\\
         		&\leqslant &\mu_1+\mu_2-\nu_1-|\nu_1|\leqslant \mu_1+\mu_2=k_1.\label{5.109}
         \end{eqnarray}

         When $\mu_1>\nu_1$ and $\mu_2\leqslant\nu_2$,
         \begin{eqnarray}
         		&\gamma_1+\gamma_{-2}=\kappa_1-\max\{|\mu_1|,|\nu_1|\}
         		\leqslant \mu_2-\max\{|\mu_1|,|\nu_1|\}\leqslant k_2,\label{5.110}\\
         		&\gamma_{-1}+\gamma_{-2}=\mu_1-\nu_1+\kappa_1-\max\{|\mu_1|,|\nu_1|\}
         		\leqslant \mu_1-\nu_1+\mu_2-|\nu_1|\leqslant k_1.\label{5.111}
         \end{eqnarray}

          When $\mu_1\leqslant \nu_1$ and $\mu_2>\nu_2$,
          \begin{eqnarray}
          		\gamma_1+\gamma_{-2}&=&\nu_1-\mu_1+\kappa_1-\max\{|\mu_1|,|\nu_1|\}+\mu_2-\nu_2\nonumber\\
          		&\leqslant & k_2+\nu_1-\max\{|\mu_1|,|\nu_1|\}\leqslant k_2,\label{5.112}
          \end{eqnarray}
          \begin{eqnarray}
          		\gamma_{-1}+\gamma_{-2}&=&\kappa_1-\max\{|\mu_1|,|\nu_1|\}+\mu_2-\nu_2\nonumber\\
          		&\leqslant &\mu_2-|\mu_1|\leqslant k_1.\label{5.113}
          \end{eqnarray}

          When $\mu_1\leqslant \nu_1$ and $\mu_2\leqslant \nu_2$,
          \begin{eqnarray}
          		&\gamma_1+\gamma_{-2}=\nu_1-\mu_1+\kappa_1-\max\{|\mu_1|,|\nu_1|\}
          		\leqslant k_2+\nu_1-\max\{|\mu_1|,|\nu_1|\}\leqslant k_2,\label{5.114}\\
          		&\gamma_{-1}+\gamma_{-2}=\kappa_1-\max\{|\mu_1|,|\nu_1|\}
          		\leqslant \mu_2-|\mu_1|\leqslant k_1.\label{5.115}
          \end{eqnarray}

          Note that the inequalities (\ref{5.106})-(\ref{5.115})  about $\gamma_i,\gamma_{-1},\gamma_i'$  coincide with conditions (\ref{5.95.1})-(\ref{5.95.4}), Thus we construct an one-to-one correspondence between $\mathcal{M}_e(\lambda)$ and the set of pairs of series $\{\vec{\kappa},\vec{\nu}\}$ satisfying (\ref{5.102}) and (\ref{5.103}). Therefore, we have
          \begin{theorem}\label{Th36}
          	$\mathcal{B}_e(\lambda)$ is a basis of the subspace of $\mathfrak{o}(2n)$-singular vectors in $V_{2n+2}(\lambda)$.
          \end{theorem}

          Now let $\mathcal{H}_e(\lambda)_r$ (respectively, $\mathcal{B}_e(\lambda)_r,~\mathcal{M}_e(\lambda)_r$) be the homogeneous subset of $\mathcal{H}_e(\lambda)$ (respectively, $\mathcal{B}_e(\lambda),~\mathcal{M}_e(\lambda)$) with degree $r$ in $\{x_{n,\pm i}~|~i\in\overline{1,n}\}$. By the definition, $\mathscr{L}|_{\mathcal{B}_e(\lambda)_r}:\mathcal{B}_e(\lambda)_r\longrightarrow\mathcal{M}_e(\lambda)_r$ is bijective and $\mathcal{B}_e(\lambda)_r\neq \emptyset$ only if $0\leqslant r\leqslant 2\mu_{n+1}$.

          For $\beta=(\beta_{-n},\cdots,\beta_{-1},\beta_1,\cdots,\beta_n)\in\mathbb{N}^{2n}$, we set
          \begin{eqnarray}
          		\lambda(\beta)=\sum_{i=1}^{n}(\mu_{i}+\beta_i-\beta_{-i})\varepsilon_i. \label{5.116}
        \end{eqnarray}

       Let $|\beta|=\sum\limits_{|i|=1}^n\beta_i$ be the length of $\beta$ and let $X_n^{\beta}=\prod\limits_{|i|=1}^nx_{n,i}^{\beta_i}$ be the monomial in $\mathbb{C}[x_{n,\pm 1},\cdots,x_{n,\pm n}]$ with degree $\beta$. Let $\widehat{\Omega}_r$ be the set of $2n$-tuples $\beta\in\mathbb{N}^{2n}$ satisfying following conditions
       \begin{eqnarray}
       	&0\leqslant \beta_i\leqslant k_{i+1},~i\in\overline{0,n-1},~0\leqslant \beta_{-1}\leqslant k_1, \label{5.117}\\
       	&0\leqslant \beta_{-i}\leqslant \beta_i-\beta_{i-1}+k_i,~i\in\overline{3,n},~|\beta|=r,\\
       	&\beta_{-2}\leqslant \mu_2+\beta_2-\max\{|\frac{k_1-k_2}{2}|,|\frac{k_1-k_2}{2}+\beta_1-\beta_{-1}|\}\label{5.118}
       \end{eqnarray}
       and
       \begin{eqnarray}
       	\Omega_r=\{\beta\in\widehat{\Omega}_r~|~0\leqslant \beta_n\leqslant k_{n+1}\}.\label{5.119}
       \end{eqnarray}
       Denote by $\msr A_r'$ the subspace of homogeneous polynomials in $\mathbb{C}[x_{n,\pm 1},\cdots,x_{n,\pm n}]$ with degree $r$. We have:

       \begin{theorem}\label{The7}
       	For $r\in\mathbb{N}$, the $\mathfrak{o}(2n)$-singular vectors in $\widehat{M}_r$ are
       	\begin{eqnarray}
       		\{f=f_{n,1}(-1)^{\gamma_{-1}}\prod_{i=1}^nf_{n,i}(i)^{\gamma_i}\prod_{j=2}^{n}f_{n,j}(-(n+1))^{\gamma'_j}\prod_{k=1}^{n-1}x_{k,-(n+1)}^{\gamma_{-(k+1)}}~|~\nonumber\\
       		\mathscr{L}(f)=X_n^{\beta}~\text{with }\beta\in\widehat{\Omega}_r\}, \label{5.120}
       	\end{eqnarray}
       	which generate all the irreducible $\mathfrak{o}(2n)$-submodules in
       	\begin{eqnarray}
       		\widehat{M}_r=\mathscr{A}'_r\otimes V_{2n}(\sum_{i=1}^n\mu_i\varepsilon_i). \label{5.121}
       	\end{eqnarray}
       	It implies the identity
       	\begin{eqnarray}
       		\sum_{\beta\in\widehat{\Omega}_r}d_{2n}(\lambda(\beta))=\binom{2n+r-1}{2n-1}d_{2n}(\sum_{i=1}^n\mu_i\varepsilon_i). \label{5.122}
       	\end{eqnarray}
       	
       	For $r \in\overline{0,2\mu_{n+1}}$, the $\mathfrak{o}(2n)$-singular vectors in $(V_{2n+2}(\lambda))_r=\overline{M}_r$ are
       	\begin{eqnarray}
       		\mathcal{B}_e(\lambda)_r=\{f\in\mathcal{M}_e(\lambda)~|~\mathscr{L}(f)=X_n^{\beta}~\text{with }\beta\in\Omega_r\}. \label{5.123}
       	\end{eqnarray}
       \end{theorem}\psp

       For $r\in\overline{0,2\mu_{n+1}}$, we take any $f\in\mathcal{B}_e(\lambda)_r$. Suppose
        \begin{eqnarray}
       	f=f_{n,1}(-1)^{\gamma_{-1}}\prod_{i=1}^nf_{n,i}(i)^{\gamma_i}\prod_{j=2}^{n}f_{n,j}(-(n+1))^{\gamma'_j}\prod_{k=1}^{n-1}x_{k,-(n+1)}^{\gamma_{-(k+1)}}.
       \end{eqnarray}
     Denote
       \begin{eqnarray}
       		& &\gamma_n''=k_{n+1}-\gamma_n-\gamma_n';\\
       		& &\gamma_i''=k_{i+1}-\gamma_i-\gamma_{-(i+1)}-\gamma_i',~i\in\overline{2,n-1};\\
       		& &\gamma_1''=k_2-\gamma_1-\gamma_{-2};\\
       		& &\gamma_{-1}''=k_1-\gamma_{-1}-\gamma_{-2}.
       \end{eqnarray}
       Let
       \begin{eqnarray}
       	\bar{f}=f_{n,1}(-1)^{\gamma_{-1}''}f_{n,1}(1)^{\gamma_1''}\prod_{i=2}^nf_{n,i}(i)^{\gamma_i}\prod_{j=2}^{n}f_{n,j}(-(n+1))^{\gamma''_j}\prod_{k=1}^{n-1}x_{k,-(n+1)}^{\gamma_{-(k+1)}}\in\mathcal{B}_e(\lambda)_r.
       \end{eqnarray}
       Similar to the proof of Lemma \ref{Lem5}, we have an one-to-one map between $\mathcal{M}_e(\lambda)_{r}$ and $\mathcal{M}_e(\lambda)_{2\mu_{n+1}-r}$ as $\mathscr{L}(f)\mapsto\mathscr{L}(\bar{f})$. Note that the weights of $f$ and $\bar{f}$ are
       \begin{eqnarray}
       	\lambda(f)=\lambda(\bar{f})=\sum_{i=1}^n(\mu_i+\gamma_i-\gamma_{-i})\varepsilon_i.
       \end{eqnarray}
       Thus we have
       \begin{lemma}\label{Lem8}
       	$\dim(V_{2n+2}(\lambda))_{r}=\dim(V_{2n+2}(\lambda))_{2\mu_{n+1}-r}$ for integer $r\in\overline{0,2\mu_{n+1}}$.
       \end{lemma}

        Applying Theorem \ref{The7} and Lemma \ref{Lem8} and using the same inclusion-exclusion process as in the odd case, we can deduce the formula of the dimension of homogeneous subspace with $0\leqslant r\leqslant \mu_{n+1}$,
       	\begin{eqnarray}
       		\!\!\!\!\!\!& &\!\!\dim(V_{2n+2}(\lambda))_r=\dim (V_{2n+2}(\lambda))_{2\mu_{n+1}-r}\nonumber\\
       	\!\!\!\!\!\!&=&\!\!\sum_{s=0}^{n}(-1)^{n-s}\binom{n\!+\!s\!+\!r\!-\!1\!+\!\mu_{s+1}\!-\!\mu_{n+1}}{2n\!-\!1}d_{2n}(\sum\limits_{i=1}^s\mu_i\varepsilon_i+\!\!\sum\limits_{i=s+1}^n(\mu_{i+1}\!+\!1)\varepsilon_i).\label{5.124}
       \end{eqnarray}
       Therefore, if $2\mu_{n+1}$ is odd,
       \begin{eqnarray}
       	\!\!\!\!& &\!\!\dim V_{2n+2}(\lambda)=2\sum\limits_{r=0}^{\mu_{n+1}-1/2}\dim(V_{2n+2}(\lambda))_r\nonumber\\
       	\!\!\!\!&=&\!\!\sum_{s=0}^{n}(-1)^{n-s}\sum\limits_{r=0}^{\mu_{n+1}-1/2}\binom{n\!+\!s\!+\!r\!-\!1\!+\!\mu_{s+1}\!-\!\mu_{n+1}}{2n\!-\!1}d_{2n}(\sum\limits_{i=1}^s\mu_i\varepsilon_i+\!\!\sum\limits_{i=s+1}^n(\mu_{i+1}\!+\!1)\varepsilon_i)\nonumber\\
       	\!\!\!\!&=&\!\!\sum_{s=0}^{n}(-1)^{n-s}\binom{n\!+\!s\!-\!\frac{1}{2}\!+\!\mu_{s+1}\!}{2n}
       d_{2n}(\sum\limits_{i=1}^s\mu_i\varepsilon_i+\!\!\sum\limits_{i=s+1}^n(\mu_{i+1}\!+\!1)\varepsilon_i),
       \end{eqnarray}
       and if $2\mu_{n+1}$ is even,
       \begin{eqnarray}
       	\!\!\!\!& &\!\!\dim V_{2n+2}(\lambda)=\sum\limits_{r=0}^{\mu_{n+1}}\dim(V_{2n+2}(\lambda))_r+\sum\limits_{r=0}^{\mu_{n+1}-1}\dim(V_{2n+2}(\lambda))_r\nonumber\\
       	\!\!\!\!&=&\!\!\sum_{s=0}^{n}(-1)^{n-s}\binom{n\!+\!s\!-\!1\!+\!\mu_{s+1}\!}{2n-1}\frac{s+\mu_{s+1}}{n}d_{2n}
       (\sum\limits_{i=1}^s\mu_i\varepsilon_i+\!\!\sum\limits_{i=s+1}^n(\mu_{i+1}\!+\!1)\varepsilon_i).
       \end{eqnarray}
Thus we have the following identities
        \begin{equation}d_{2n+2}(\sum_{i=1}^{n+1}\mu_i\ves_i)
       =\!\!\sum_{s=0}^{n}(-1)^{n-s}\binom{n\!+\!s\!-\!\frac{1}{2}\!+\!\mu_{s+1}\!}{2n}
       d_{2n}(\sum\limits_{i=1}^s\mu_i\varepsilon_i+\!\!\sum\limits_{i=s+1}^n(\mu_{i+1}\!+\!1)\varepsilon_i) \label{4.151}
       \end{equation}
       if $2\mu_{n+1}$ is odd, and
            \begin{eqnarray}d_{2n+2}(\sum_{i=1}^{n+1}\mu_i\ves_i)
              &=&\!\!\sum_{s=0}^{n}(-1)^{n-s}\binom{n\!+\!s\!-\!1\!+\!\mu_{s+1}\!}{2n-1}\frac{s+\mu_{s+1}}{n}
       \nonumber\\ && \times d_{2n}(\sum\limits_{i=1}^s\mu_i\varepsilon_i+\!\!\sum\limits_{i=s+1}^n(\mu_{i+1}\!+\!1)\varepsilon_i) \label{4.152}
       \end{eqnarray}
when $2\mu_{n+1}$ is even.

       As an example, we also consider the $Steinberg$ $module$ with highest weight $\lambda=k(\lambda_1+\lambda_2+\cdots+\lambda_{n+1})$, where $\lambda_i$'s are the fundamental
       weights. In this case, $\mu_i=(i-1)k$ for $i\in\overline{1,n+1}$. Recall that the set of positive roots of $\mathfrak{o}(2n)$,
       \begin{eqnarray}
       	\Phi_{D_n}^+=\{\varepsilon_j\pm\varepsilon_i~|~1\leqslant i<j\leqslant n\},
       \end{eqnarray}
       and
       \begin{eqnarray}
       	\rho=\frac{1}{2}\sum_{\alpha\in \Phi_{D_n}^+}\alpha=\sum_{i=2}^n(i-1)\varepsilon_i.
       \end{eqnarray}
       By Weyl's dimension formula (e.g., cf. \cite{H}), we have
       \begin{eqnarray}
       	\dim V_{2n+2}(\lambda)=(k+1)^{n^2+n}.    \label{5.125}
       \end{eqnarray}
       Denote by
       \begin{eqnarray}
       	\lambda_s=\sum\limits_{i=1}^s\mu_i\varepsilon_i+\sum\limits_{i=s+1}^n(\mu_{i+1}+1)\varepsilon_i.
       \end{eqnarray}
       We have
       \begin{eqnarray}
       	\prod_{1\leqslant i<\leqslant j\leqslant s}\frac{(\lambda_s+\rho,\varepsilon_j\pm\varepsilon_i)}{(\rho,\varepsilon_j\pm\varepsilon_i)}=\prod_{1\leqslant i<j\leqslant s}(k+1)^2=(k+1)^{s(s-1)},
       \end{eqnarray}
        \begin{eqnarray}
        	\prod_{s+1\leqslant i<j\leqslant n}\frac{(\lambda_s+\rho,\varepsilon_j\pm\varepsilon_i)}{(\rho,\varepsilon_j\pm\varepsilon_i)}&=&\prod_{s+1\leqslant i<j\leqslant n}\frac{(j+i)(k+1)^2}{j+i-2}\\
        	&=&\frac{(2s)!(2n-1)!(k+1)^{(n-s)(n-s-1)}}{(n+s-1)!(n+s)!},
        \end{eqnarray}
        \begin{eqnarray}
        	\prod_{1\leqslant i\leqslant s\atop s+1\leqslant j\leqslant n}\frac{(\lambda_s+\rho,\varepsilon_j\pm\varepsilon_i)}{(\rho,\varepsilon_j\pm\varepsilon_i)}&=&\prod_{1\leqslant i\leqslant s\atop s+1\leqslant j\leqslant n}\frac{(j-i+1)(j+i-1)(k+1)^2}{(j-i)(j+i-2)}\\
        	&=&\frac{n(n+s-1)!(k+1)^{2s(n-s)}}{s(n-s)!(2s-1)}.
        \end{eqnarray}
       Therefore
       \begin{eqnarray}
       	d_{2n}(\sum\limits_{i=1}^s\mu_i\varepsilon_i+\sum\limits_{i=s+1}^n(\mu_{i+1}+1)\varepsilon_i)=\binom{2n}{n-s}(k+1)^{n^2-n}.\label{5.126}
       \end{eqnarray}
       Note that $2\mu_{n+1}=2nk$ is even and (1.142)
       implies a classical combinatorial identity
       \begin{equation}
       	\sum_{s=1}^n(-1)^{n-s}s\binom{n+s(k+1)-1}{2n-1}\binom{2n}{n-s}=n(k+1)^{2n-1}. v\label{4.163}
       \end{equation}

\appendix

\section{Proof of Lemma \ref{L4}}

We will prove Lemma \ref{L4} in Section 2. Recall the convention we made in $(\ref{4.24})-(\ref{con3})$ about the notation $x_{i,j}$ and the first-order differential operators $\mathcal{D}_{i,j}$ defined in $(\ref{d1})-(\ref{d6})$,

\begin{lemma}
	For $0\leqslant |i|< j\leqslant n+1,\;r\in\overline{0,n}$ and $s\in\overline{-(n+1),n+1}$, we have
	\begin{enumerate}[(i)]
		\item $\mathcal{D}_{j,i}(x_{r,s})=\delta_{i,s}x_{r,j}-\delta_{-j,s}x_{r,-i}$.
		\item $\mathcal{D}_{j,j}(x_{r,s})=(\delta_{j,s}-\delta_{-j,s}-\delta_{j-1,r})x_{r,s}$.
		\item Moreover, if $j\leqslant r$, $$\mathcal{D}_{i,j}(x_{r,s})=\delta_{j,s}x_{r,i}-\delta_{-i,s}x_{r,-j},$$
		if $j>r,\;i\geqslant -r$, $$\mathcal{D}_{i,j}(x_{r,s})=\delta_{j,s}x_{r,i}-\delta_{-i,s}x_{r,-j}-x_{r,i}f_{j-1,r}(s),$$
		if $j>r,\;i<-r$,
		$$\mathcal{D}_{i,j}(x_{r,s})=\delta_{j,s}x_{r,i}-\delta_{-i,s}x_{r,-j}-x_{r,i}f_{j-1,r}(s)+x_{r,-j}f_{-i-1,r}(s).$$
	\end{enumerate}
\end{lemma}
	 \begin{proof}
	We prove this lemma case by case.\psp
	
	{\bf Proof of $(i)$.}\pse
	
	Suppose that $s>r$. Note that $x_{r,s}$ is a constant by (\ref{4.24}), thus $\mathcal{D}_{j,i}(x_{r,s})=0$. Since $j>0$, $\delta_{-j,s}=0$. If $i\neq s$, then $\delta_{i,s}=0$ and $(i)$ holds. If $i=s$, then $j>s$ and $x_{r,j}$=0, $(i)$ hols as well.
	
	Suppose that $s\in\overline{-r,r}$.In this situation, $x_{r,s}$ is a variable according to (\ref{4.24}) thus $\mathcal{D}_{j,i}(x_{r,s})$ are just the coefficient polynomial of $\partial_{x_{r,s}}$ in the expression of $\mathcal{D}_{j,i}$. Applying (\ref{d2}) and (\ref{d4}) directly, we obtain $(i)$.
	
	Suppose that $s\in\overline{-(n+1),-(r+1)}$. We prove this case by backward induction on $s$. If $s=-(r+1)$, according to (\ref{con2}),
	\begin{eqnarray}
		\mathcal{D}_{j,i}(x_{r,-(r+1)})&=&-\mathcal{D}_{j,i}\big(\frac{1}{2}\sum_{t=-r}^rx_{r,t}x_{r,-t}\big)=-\sum_{t=-r}^r\mathcal{D}_{j,i}(x_{r,t})x_{r,-t}\nonumber\\
		&=&\left\lbrace \begin{array}{ll}
			x_{r,-i},&\text{if}\;j=r+1;\\
			0,&\text{otherwise}.
		\end{array}\right.
	\end{eqnarray} 
	This coincides with $(i)$ when $s=-(r+1)$. Consider $s<-(r+1)$. By (\ref{con3}) and induction assumption,
	\begin{eqnarray}
		& &\mathcal{D}_{j,i}(x_{r,s})=-\mathcal{D}_{j,i}\big(\sum_{t=s+1}^{-s-1}x_{r,t}x_{-s-1,-t}\big)\nonumber\\
		&=&-\sum_{t=s+1}^{-s-1}\mathcal{D}_{j,i}(x_{r,t})x_{-s-1,-t}-\sum_{t=s+1}^{-s-1}x_{r,t}\mathcal{D}_{j,i}(x_{-s-1,-t})\nonumber\\
		&=&-\sum_{t=s+1}^{-s-1}\!x_{-s-1,-t}(\delta_{i,t}x_{r,j}\!-\!\delta_{-j,t}x_{r,-i})-\!\sum_{t=s+1}^{-s-1}\!x_{r,t}(\delta_{i,-t}x_{-s-1,j}\!-\!\delta_{j,t}x_{-s-1,-i}).\label{i1}
	\end{eqnarray}
	If $j<-s$,
	\begin{eqnarray}
		(\ref{i1})=-x_{-s-1,-i}x_{r,j}+x_{-s-1,j}x_{r,-i}-x_{r,-i}x_{-s-1,j}+x_{r,j}x_{-s-1,-i}=0.\label{i2}
	\end{eqnarray}
	If $j=-s$, $x_{r,j}=0$ and $x_{-s-1,j}=1$, then
	\begin{eqnarray}
		(\ref{i1})=-x_{-s-1,-i}x_{r,j}-x_{r,-i}x_{-s-1,j}=-x_{r,-i}.\label{i3}
	\end{eqnarray}	 	
	If $j>-s$, $x_{r,j}=x_{-s-1,j}=0$, then
	\begin{eqnarray}
		(\ref{i1})=-x_{-s-1,-i}x_{r,j}-x_{r,-i}x_{-s-1,j}=0.\label{i4}
	\end{eqnarray}
	(\ref{i2})-(\ref{i4}) coincide with $(i)$ when $s<-(r+1)$, thus $(i)$ is proved. \psp
	
	{\bf Proof of $(ii)$.}\pse

	Now we turn to prove $(ii)$. By the similar argument with the proof of $(i)$, we can prove $(ii)$ holds for $s>-r$. 
	
	Suppose that $s\in\overline{-(n+1),-(r+1)}$. By (\ref{con2}),
	\begin{eqnarray}
		\mathcal{D}_{j,j}(x_{r,-(r+1)})=-\sum_{t=-r}^r\mathcal{D}_{j,j}(x_{r,t})x_{r,-t}=\left\lbrace \begin{array}{ll}
			-2x_{r,-(r+1)},&\text{if}\;j=r+1,\\
			0,&\text{otherwise},
		\end{array}\right.
	\end{eqnarray}
	which coincides with $(ii)$ when $s=-(r+1)$. Assume $s<-(r+1)$, by induction,
	\begin{eqnarray}
		& &\mathcal{D}_{j,j}(x_{r,s})=-\mathcal{D}_{j,j}\big(\sum_{t=s+1}^{-s-1}x_{r,t}x_{-s-1,-t}\big)\nonumber\\
		&=&-\sum_{t=s+1}^{-s-1}\mathcal{D}_{j,j}(x_{r,t})x_{-s-1,-t}-\sum_{t=s+1}^{-s-1}x_{r,t}\mathcal{D}_{j,j}(x_{-s-1,-t})\nonumber\\
		&=&-\sum_{t=s+1}^{-s-1}(\delta_{j,t}-\delta_{-j,t}-\delta_{j-1,r}-\delta_{j,-t}+\delta_{j,t}+\delta_{j-1,-s-1})x_{r,t}x_{-s-1,-t}\nonumber\\
		&=&\sum_{t=s+1}^{-s-1}(\delta_{j-1,r}+\delta_{j-1,-s-1})x_{r,t}x_{-s-1,-t}\nonumber\\
		&=&-(\delta_{j-1,r}+\delta_{j,-s})x_{r,s}.\label{ii1}
	\end{eqnarray}
	Since $\delta_{j,s}=0$, (\ref{ii1}) coincides with $(ii)$ when $s<-(r+1)$. \psp
	
	{\bf Proof of $(iii)$.}\pse
	
	We can still prove $(iii)$ for $s>-r$ or $j\leqslant r$ similarly to the formal cases. Suppose that $s\in\overline{-(n+1),-(r+1)}$ and $j>r$. 
	
	If $i\geqslant -r$,
	\begin{eqnarray}
		\mathcal{D}_{i,j}(x_{r,-(r+1)})&=&-\sum_{t=-r}^r\mathcal{D}_{i,j}(x_{r,t})x_{r,-t}\nonumber\\
		&=&-\sum_{t=-r}^r\big(\delta_{j,t}x_{r,i}-\delta_{-i,t}x_{r,-j}-x_{r,i}f_{j-1,r}(t)\big)x_{r,-t}. \label{iii1}
	\end{eqnarray}
	Note that
	\begin{eqnarray}
		\sum_{t=-r}^r\delta_{j,t}x_{r,-t}=0\;\;\;\;\;\text{and}\;\;\;\;\sum_{t=-r}^r\delta_{-i,t}x_{r,-t}=\left\lbrace \begin{array}{ll}
			x_{r,i},&\text{if}\;i\in\overline{-r,r},\\
			0,&\text{otherwise}.
		\end{array}\right.\label{iii2}
	\end{eqnarray}
	By Lemma \ref{L2},
	\begin{eqnarray}
		\sum_{t=-r}^rf_{j-1,r}(t)x_{r,-t}=-x_{r,-j}-f_{j-1,r}(-(r+1)),
	\end{eqnarray}
	Thus
	\begin{eqnarray}
		(\ref{iii1})&=&-x_{r,i}(x_{r,-j}+f_{j-1,r}(-(r+1)))+x_{r,-j}\sum_{t=-r}^r\delta_{-i,t}x_{r,-t}\nonumber\\
		&=&\left\lbrace \begin{array}{ll}
			-x_{r,i}f_{j-1,r}(-(r+1)),&\text{if}\;i\in\overline{-r,r},\\
			-f_{j-1,r}(-(r+1))-x_{r,-j},&\text{if}\;i=r+1,\\
			0,&\text{if}\;i>r+1.
		\end{array}\right.\label{iii4}
	\end{eqnarray}
	
	If $i<-r$,
	\begin{eqnarray}
		\mathcal{D}_{i,j}(x_{r,-(r+1)})&=&-\sum_{t=-r}^r\mathcal{D}_{i,j}(x_{r,t})x_{r,-t}\nonumber\\
		&=&\sum_{t=-r}^r\big(x_{r,i}f_{j-1,r}(t)+
		-x_{r,-j}f_{-i-1,r}(t)\big)x_{r,-t}\nonumber\\
		&=&-x_{r,i}f_{j-1,r}(-(r+1))+x_{r,-j}f_{-i-1,r}(-(r+1)), \label{iii5}
	\end{eqnarray}
	the last equality is obtained by Lemma \ref{L2}. Thus $(iii)$ holds for $s=-(r+1)$ by (\ref{iii4}) and (\ref{iii5}). 
	
	Assume $s<-(r+1)$. By (\ref{con3}),
	\begin{eqnarray}
		\mathcal{D}_{i,j}(x_{r,s})&=&-\sum_{t=s+1}^{-s-1}\mathcal{D}_{i,j}(x_{r,t}x_{-s-1,-t})-\sum_{t=s+1}^{-s-1}x_{r,t}\mathcal{D}_{i,j}(x_{-s-1,-t}).\label{iii6}
	\end{eqnarray}
	
	By Lemma \ref{L2} and induction assumption, if $j\in\overline{r+1,-s-1},\;i\in\overline{-r,j-1}$,
	\begin{eqnarray}
		(\ref{iii6})&=&-\sum_{t=s+1}^{-s-1}\big[\delta_{j,t}x_{r,i}-\delta_{-i,t}x_{r,-j}-x_{r,i}f_{j-1,r}(t)\big]x_{-s-1,-t}\nonumber\\
		&&-\sum_{t=s+1}^{-s-1}\big[\delta_{j,-t}x_{-s-1,i}-\delta_{i,t}x_{-s-1,-j}\big]x_{r,t}\nonumber\\
		&=&-x_{r,i}x_{-s-1,-j}\!+\!x_{r,-j}x_{-s-1,i}\!-\!x_{r,i}f_{j-1,r}(s)\!-\!x_{-s-1,i}x_{r,-j}\!+\!x_{-s-1,-j}x_{r,i}\nonumber\\
		&=&-x_{r,i}f_{j-1,r}(s);   \label{iii7}		
	\end{eqnarray}
	if $j\in\overline{-s,n+1},\;i\in\overline{-r,-s-1}$,
	\begin{eqnarray}
		(\ref{iii6})&=&-\sum_{t=s+1}^{-s-1}\big[\delta_{j,t}x_{r,i}-\delta_{-i,t}x_{r,-j}-x_{r,i}f_{j-1r}(t)\big]x_{-s-1,-t}\nonumber\\
		&&-\sum_{t=s+1}^{-s-1}\big[\delta_{j,-t}x_{-s-1,i}-\delta_{i,t}x_{-s-1,-j}-x_{-s-1,i}f_{j-1,-s-1}(-t)\big]x_{r,t}\nonumber\\
		&=&x_{r,-j}x_{-s-1,i}\!-\!x_{r,i}(f_{j-1,r}(s)\!+\!x_{-s-1,-t})\!+\!x_{-s-1,-j}x_{r,i}\!-\!x_{-s-1,i}x_{r,-j}\nonumber\\
		&=&-x_{r,i}f_{j-1,r}(s);   \label{iii8}	
	\end{eqnarray}
	if $j\in\overline{-s,n+1},\;i\in\overline{-s,j-1}$,
	\begin{eqnarray}
		(\ref{iii6})&=&-\sum_{t=s+1}^{-s-1}\big[\delta_{j,t}x_{r,i}-\delta_{-i,t}x_{r,-j}-x_{r,i}f_{j-1r}(t)\big]x_{-s-1,-t}\nonumber\\
		&&-\sum_{t=s+1}^{-s-1}\big[\delta_{j,-t}x_{-s-1,i}-\delta_{i,t}x_{-s-1,-j}-x_{-s-1,i}f_{j-1,-s-1}(-t)\big]x_{r,t}\nonumber\\
		&=&-x_{-s-1,i}x_{r,-j}=\left\lbrace \begin{array}{ll}
			x_{r,-j},&\text{if}\;i=-s,\\
			0,&\text{if}\;i>-s.
		\end{array}\right.   \label{iii9}	
	\end{eqnarray}
    Therefore, the second equation in $(iii)$ hold for $s<-(r+1)$.
    
    If $j\in\overline{r+1,-s-1},\;i\in\overline{-j+1,-r-1}$,
    \begin{eqnarray}
    	(\ref{iii6})&=&-\sum_{t=s+1}^{-s-1}\big[\delta_{j,t}x_{r,i}\!-\!\delta_{-i,t}x_{r,-j}\!-\!x_{r,i}f_{j-1,r}(t)\!+\!x_{r,-j}f_{-i-1,r}(t)\big]x_{-s-1,-t}\nonumber\\
    	&&-\sum_{t=s+1}^{-s-1}\big[\delta_{j,-t}x_{-s-1,i}-\delta_{i,t}x_{-s-1,-j}\big]x_{r,t}\nonumber\\
    	&=&-x_{r,i}x_{-s-1,-j}+x_{r,-j}x_{-s-1,i}-x_{r,i}f_{j-1,r}(s)+x_{r,-j}f_{-i-1,r}(s)\nonumber\\
    	& &-x_{-s-1,i}x_{r,-j}+x_{-s-1,-j}x_{r,i}\nonumber\\
    	&=&-x_{r,i}f_{j-1,r}(s)+x_{r,-j}f_{-i-1,r}(s)
    	\label{iii10}		
    \end{eqnarray} 
    if $j\in\overline{-s,n+1},\;i\in\overline{s+1,-r-1}$,
    \begin{eqnarray}
    	(\ref{iii6})&=&-\sum_{t=s+1}^{-s-1}\big[\delta_{j,t}x_{r,i}\!-\!\delta_{-i,t}x_{r,-j}\!-\!x_{r,i}f_{j-1,r}(t)\!+\!x_{r,-j}f_{-i-1,r}(t)\big]x_{-s-1,-t}\nonumber\\
    	&&-\sum_{t=s+1}^{-s-1}\big[\delta_{j,-t}x_{-s-1,i}-\delta_{i,t}x_{-s-1,-j}-x_{-s-1,i}f_{j-1,-s-1}(-t)\big]x_{r,t}\nonumber\\
    	&=&x_{r,-j}x_{-s-1,i}-x_{r,i}(f_{j-1,r}(s)+x_{-s-1,-j})+x_{r,-j}f_{-i-1,r}(s)\nonumber\\
    	& &+x_{-s-1,-j}x_{r,i}-x_{-s-1,i}x_{r,-j}\nonumber\\
    	&=&-x_{r,i}f_{j-1,r}(s)+x_{r,-j}f_{-i-1,r}(s)
    	\label{iii11}		
    \end{eqnarray}
    if $j\in\overline{-s,n+1},\;i\in\overline{-j+1,s}$,
    \begin{eqnarray}
    	\!\!\!(\ref{iii6})\!\!&\!=\!\!&\!\!-\!\!\sum_{t=s+1}^{-s-1}\!\!\big[\delta_{j,t}x_{r,i}\!-\!\delta_{-i,t}x_{r,-j}\!-\!x_{r,i}f_{j-1,r}(t)\!+\!x_{r,-j}f_{-i-1,r}(t)\big]x_{-s-1,-t}\nonumber\\
    	&\!\!\!&\!\!-\!\!\sum_{t=s+1}^{-s-1}\!\!\big[\delta_{j,-t}x_{-s-\!1,i}\!-\!\delta_{i,t}x_{-s-\!1,-j}\!-\!x_{-s-\!1,i}f_{j-1,-s-\!1}(\!-t)\!+\!x_{-s-\!1,-j}f_{-i-\!1,-s-\!1}(\!-t)\big]x_{r,t}\nonumber\\
    	&\!=\!&\!\!-x_{r,i}(f_{j-1,r}(s)+x_{-s-1,-j})+x_{r,-j}(f_{-i-1,r}(s)+x_{-s-1,i})\nonumber\\
    	&\!\!\! &\!\!-x_{-s-1,i}x_{r,-j}+x_{-s-1,-j}x_{r,i}\nonumber\\
    	&\!=\!\!&\!\!\!-x_{r,i}f_{j-1,r}(s)+x_{r,-j}f_{-i-1,r}(s)
    	  \label{iii12}		
    \end{eqnarray}
    Therefore, the last equation in $(iii)$ hold for $s<-(r+1)$.
\end{proof}\pse


\begin{thebibliography}{99}
       	
      	
\bibitem{A} A. Alex, M. Kalus, A. Huckleberry, T. Alan and J. von Delft, A numerical algorithm for the explicit calculation of $SU(N)$ and $SL(N,\mbb C)$ Clebsch-Gordan coefficients, {\it J. Math. Phys.} {\bf 52} (2011), no. 2, 023507, 21pp.

\bibitem{B} C. Bohning and H. Graf von Bothmer, A Clebsch-Gordan formula for $SL_3(\mbb C)$ and applications to rationality,
{\it Adv. Math.} {\bf 224} (2010), no. 1, 246--259.

\bibitem{C} I. V. Cherednik, A new interpretation of Gelfand-Tsetlin bases, {\it Duke Math. J. } {\bf 54}(1987), 563--577.

\bibitem{dS} J. J. de Swart, The octet model and its Clebsch-Gordan coefficients, {\it Rev. Mod. Phys.} {\bf 35} (1963), no. 4, 916--939.


\bibitem{E} A. R. Edmonds, {\it Angular Momentum in Quantum Mechanics,} Princeton University Press, Princeton, New Jersey, 1957.


\bibitem{FGR1} V. M. Futorny, D. Grantcharov and L. E. Ramirez, Singular Gelfand-Tsetlin modules of $gl(n)$, {\it Adv. Math.} {\bf 290} (2016), 453--482.

\bibitem{FGR2} V. M. Futorny, D. Grantcharov and L. E. Ramirez, New Singular Gelfand-Tsetlin $gl(n)$-modules of index 2, {\it
Commun. Math. Phys.} {\bf 355} (2017), 1209--1241.

\bibitem{GT1} I. M. Gelfand and M. L. Tsetlin, Finite-dimensional representations of the group of unimodular matrices, {\it Dokl. Akad. Nauk SSSR} {\bf 71} (1950), 825--828.

\bibitem{GT2} I. M. Gelfand and M. L. Tsetlin, Finite-dimensional representations of the group of orthogonal matrices, {\it Dokl. Akad. Nauk SSSR} {\bf 71} (1950), 1017--1020.

\bibitem{G1} M. D. Gould, On the matrix elements of the $U(n)$ generators, {\it J. Math. Phys.} {\bf 22} (1981), 15--22.

\bibitem{G2} M. D. Gould, Wigner coefficients for a semisimple Lie group and the matrix elements of the $O(n)$ generators,
{\it J. Math. Phys.} {\bf 22} (1981), 2376--2388.

\bibitem{G3} M. D. Gould, Representation theory of the symplectic groups. I,
{\it J. Math. Phys.} {\bf 30} (1989), 1205--1218.


\bibitem{GK} M. D. Gould and E. G. Kalnins, A projection-based solutions to the $Sp(2n)$ state labeling problem,,
{\it J. Math. Phys.} {\bf 30} (1989), 1205--1218.

 \bibitem{H}J. E. Humphreys. {\it Introduction to lie algebras and representation theory}, Graduate Texts in Mathematics
9, Springer, 1972.

\bibitem{K} V. Kac, {\it Infinite dimensional Lie algebras,}
Birkhauser, Boston, 1983.

\bibitem{KZ} V. G. Knizhnik and A. B. Zamolodchikov, Current algebra and Wess-Zumino models in two dimensions, {\it Nucl. Phy. B} {\bf 247} (1984), 83--103.

\bibitem{L} X. Li and J. Paldus, Unitary group tensor operator algebras for many-electron systems. I: Clebsch-Gordan and Racah coefficients, {\it J. Math. Chem.} {\bf 4} (1990), no. 1-4, 295--253.

\bibitem{Lp}{}P. Littelmann. An algorithm to compute bases and representation matrices for $sl_{n+1}$ representations,
{\it J. Pure  Appl. Algebra} {\bf 117/118} (1997),  447--468.

\bibitem{M1} A. I. Molev, Weight bases of Gelfand-Tsetlin type for representations of classical Lie algebras,
       	{\it J. Phys. A: Math. Gen.} {\bf 33} (2000), 4143--4158.
       	
\bibitem{M2} A. I. Molev, Yangian and transvector algebras, {\it Discr. Math}. {\bf 246} (2002), 231-253.
       	
\bibitem{M3} A. I. Molev, Gelfand-Tsetlin basis for representations of Yangians, {\it Lett. Math. Phys.} {\bf 30} (1994), 53-60.
       	
\bibitem{M4} A. I. Molev, A basis for representations of even orthogonal Lie algebras, {\it Adv. Studies in Pure Math.} {\bf 28} (2002), 223-242.

\bibitem{Mi} Macdonald, I. G., Affine root systems and Dedekind's
$\eta$-function, {\it Invent. Math.} {\bf 15} (1972), 91-143.

\bibitem{N1} J. G. Nagel and M. Moshinsky, Operators that lows or raise the irreducible vector spaces of $U_{n-1}$ contained in an irreducible vector space of $U_n$. {\it J. Math. Phys.} {\bf 6} (1995), 682-694.

\bibitem{P1} S. C. Pang and K. T. Hecht, Lowering and raising operators for the orthogonal group in the chain $O(n)\supset  O(n-1)\supset \cdots$, and their graphs. {\it J. Math. Phys.} {\bf 8} (1967), 1233-1251.
    	

\bibitem{R} J. Ramos, Differential operators in terms of Clebsch-Gordan coefficients and the wave equation of massless tensor fields, {\it Gen. Ralativity Gravitations} {\bf 38} (2006), no. 5. 773--783.

\bibitem{Ri}J. Riordan, {\it Combinatorial identities}, Robert E. Krieger Publishing Company, 1979.

\bibitem{S} G. Shen, Graded modules of graded Lie algebras of Cartan
type (I)---mixed product of modules, {\it Science in China A} {\bf
29} (1986), 570--581.

\bibitem{TK} A. Tsuchiya and Y. Kanie,  Vertex operators in conformal field theory on ${\bf P}^1$ and monodromy representations of braid group, {\it Adv. St. Pure Math.} {\bf 16} (1988), 297--372.

\bibitem{W1} M. K. F. Wong, Representations of the orthogonal group. I. Lowering and raising operators operators of the orthogonal group and matrix elements of the generators. {\it J. Math. Phys.} {\bf b} (1967), 1899-1911.
     	
       	
\bibitem{X1} X. Xu, {\it Introduction to Vertex Operator Superalgebras and Their Modules}, Kluwer Academic Publishers, Dordrecht/Boston/London, 1998.
       	
\bibitem{X4} X. Xu, Differential invariants of classical groups, {\it Duke J. Math.} {\bf 94} (1998), no. 3, 543--572.
       	
\bibitem{X3} X. Xu, Differential operator representation of $S_n$ and singular vectors in Verma modules, {\it Algebr. Represent. Theory} {\bf 15} (2012), 211--231.
       	
\bibitem{X2} X. Xu, {\it Representations of Lie Algebras and Partial
       		Differential Equations}, Springer, Singapore, 2017.
       	
\bibitem{XZ}  X. Xu and Y. Zhao,  Extensions of the conformal representations for orthogonal Lie algebras,
 {\it J. Algebra} {\bf 377} (2013), 97-124.

\bibitem{DPZ1} D. P. Zhelobenko, The classical groups. Spectral analysis of their finite-dimensional representations, {\it Russ. Math. Surv} {\bf 17} (1962), 1-94.
       	
\bibitem{DPZ2} D. P. Zhelobenko, On Gelfand-Zetlin bases for classical Lie algebra, in "Representations of Lie groups and Lie algebra" (A. A. Kirilloc, Ed.) pp. 79-106. Budapest: Akademiai Kiado 1988.
       	
       \end{thebibliography}
\end{document}